\documentclass{article}[utf8, 12pt]%
\usepackage{amsmath}
\usepackage{amsfonts}
\usepackage{amssymb}
\usepackage{mathrsfs}
\usepackage{graphicx}
\usepackage{hyperref}
\usepackage{stmaryrd}
\usepackage{authblk}
\usepackage{enumerate}
\usepackage{bbm}
\usepackage{color}%
\usepackage{esint}
\usepackage{mathtools}
\usepackage{etoolbox}
\apptocmd{\thebibliography}{\renewcommand{\sc}{}}{}{}
\usepackage[title,titletoc]{appendix}
\usepackage{xpatch,amsthm}
\makeatletter
\xpatchcmd{\@thm}{\fontseries\mddefault\upshape}{}{}{} 
\makeatother
\providecommand{\U}[1]{\protect\rule{.1in}{.1in}}
\newtheorem{theorem}{Theorem}

\newtheorem{claim}[theorem]{Claim}

\newtheorem{corollary}[theorem]{Corollary}

\newtheorem{definition}[theorem]{Definition}

\newtheorem{lemma}[theorem]{Lemma}

\newtheorem*{question*}{Question}
\newtheorem{proposition}[theorem]{Proposition}
\newtheorem{remark}[theorem]{Remark}

\begin{document}
	\author[a]{Rohit Mahato\thanks{rohitmahato@iisc.ac.in}}
	\author[b]{Swarnendu Sil\thanks{swarnendusil@iisc.ac.in}\thanks{Supported by the ANRF-SERB MATRICS Project grant MTR/2023/000885}}
	
	\affil[a,b]{Department of Mathematics \\ Indian Institute of Science\\ Bangalore, India \\ }
	
	\title{Weighted estimates for Hodge-Maxwell systems}

	\maketitle 
	\begin{abstract}
		We establish up to the boundary regularity estimates in weighted $L^{p}$ spaces with Muckenhoupt weights $A_{p}$ for weak solutions to the Hodge systems 
		\begin{align*}
			d^{\ast}\left(Ad\omega\right) + B^{\intercal}dd^{\ast}\left(B\omega\right) = \lambda B\omega + f \quad \text{ in } \Omega
		\end{align*}
		with either $\nu \wedge \omega $ and $\nu \wedge d^{\ast}\left(B\omega\right)$	or $\nu \lrcorner B\omega$ and $\nu \lrcorner Ad\omega$ prescribed on $\partial\Omega.$ As a consequence, we prove the solvability of Hodge-Maxwell systems and derive Hodge decomposition theorems in weighted Lebesgue spaces. Our proof avoids potential theory, does not rely on representation formulas and instead uses decay estimates in the spirit of `Campanato method'  to establish weighted $L^{p}$ estimates. 
	\end{abstract}
	\smallskip  
	 
	\tableofcontents 
	\section{Introduction}
	\subsection{Hodge systems} 
	Let $n \geq 2, 1 \leq k \leq n-1$ and $N \geq 1$ be integers. Let $\Omega \subset \mathbb{R}^{n}$ be an open, bounded subset with at least $C^{2,1}$ boundary. We are interested in the general \emph{Hodge systems}, given by the operator 
	\begin{align*}
		\mathfrak{L}u := d^{\ast}\left(A\left(x\right)du\right) + \left[B(x)\right]^{\intercal}dd^{\ast}\left(B(x)u\right), 
	\end{align*}
	where $u$ is an $\mathbb{R}^{N}$-valued differential $k$-form in $\Omega,$ $d$ is the exterior derivative, $d^{\ast}$ is the codifferential, $A, B$ are matrix fields on $\Omega$ of appropriate dimension, with suitable regularity and ellipticity properties and the superscript $\vphantom{\left[B(x)\right]}^{\intercal}$ denotes the matrix transpose.  We are concerned with the boundary value problems 
	\begin{align*}
		\left\lbrace \begin{aligned}
			\mathfrak{L}u &= f &&\text{ in } \Omega, \\
			\mathfrak{b}_{\text{t}}u &= \mathfrak{b}_{\text{t}}u_{0} &&\text{ on } \partial\Omega,
		\end{aligned}\right. \quad \text{ and } \quad \left\lbrace \begin{aligned}
			\mathfrak{L}u &= f &&\text{ in } \Omega, \\
			\mathfrak{b}_{\text{n}}u &= \mathfrak{b}_{\text{n}}u_{0} &&\text{ on } \partial\Omega. 
		\end{aligned}\right.
	\end{align*}
	Here $\mathfrak{b}_{\text{t}}u := \left( \nu \wedge u, \nu \wedge d^{\ast}\left(B(x)u\right)\right)$  and $	\mathfrak{b}_{\text{n}}u := \left( \nu \lrcorner \left(B(x)u\right), \nu \lrcorner \left(A\left(x\right)du\right)\right)$ specifies the tangential and the normal Hodge boundary condition, respectively.

	\paragraph*{Hodge Laplacian} When $A, B$ are constant coefficient identity matrices, these problems reduce to the tangential and normal boundary value problems for the Hodge Laplacian, 
	\begin{align*}
		\left\lbrace \begin{aligned}
			\left( d^{\ast}d + d d^{\ast}\right)u &= f &&\text{ in } \Omega, \\
			\nu \wedge u &= \nu \wedge u_{0} &&\text{ on } \partial\Omega,\\
			\nu \wedge d^{\ast}u &= \nu \wedge d^{\ast}u_{0} &&\text{ on } \partial\Omega,
		\end{aligned}\right.  \text{ and }  \left\lbrace \begin{aligned}
			\left( d^{\ast}d + d d^{\ast}\right)u &= f &&\text{ in } \Omega, \\
			\nu \lrcorner u &= \nu \lrcorner u_{0} &&\text{ on } \partial\Omega,\\
			\nu \lrcorner du &= \nu \lrcorner du_{0} &&\text{ on } \partial\Omega. 
		\end{aligned}\right.
	\end{align*}
	The two problems above, which are Hodge dual of each other, have been extensively studied throughout the literature. Regularity estimates for these problems imply the important Helmholtz-Hodge-Kodaira-Morrey decomposition theorems, which are crucial tools to study a large class of problems in mathematical physics, namely time-harmonic Maxwell's equations in electromagnetic theory, div-curl or Cauchy-Riemann systems, Stokes and the Navier-Stokes operators in fluid mechanics etc.  The history of these problems goes back to Hodge \cite{Hodge1934} and Weyl \cite{Weyl1940}. The important Gaffney inequality was presented in
	\cite{GaffneyHarmonicoperator} (see also  \cite{Gaffney1954}, \cite{Gaffney1954a}, \cite{GaffneyHarmonicintegrals} and Friedrichs \cite{FriedrichsGaffney}). The variational
	method was applied to general compact Riemannian manifolds without boundary in Morrey and Eells \cite{Morrey1955} and to such manifolds with the boundary in \cite{MorreyHarmonic2}. The corresponding Schauder and $L^{p}$ estimates were derived in \cite[Chapter 7]{Morrey1966}. The applications of Morrey's up to the boundary $L^{p}$ estimates to other problems in analysis and mathematical physics are too vast to sketch here. The entire book of Schwarz \cite{SchwarzHodge} is devoted to the subject (see also \cite[Part II]{CsatoDacKneuss} and \cite{csatothesis},\cite{silthesis}). We note that Morrey's proof is based on potential theory. These boundary value problems for the Hodge Laplacian have also been investigated in the context of possibly non-smooth domains, again based on potential theory and layer potentials (see \cite{Mitrea2016}, \cite{Mitrea_Mitrea_Taylor_LayerPotentials_HodgeLplacian}, \cite{Mitrea_Gaffneyineq} and the references therein). Recently, Morrey's proof has been extended in \cite{Balci_Sil_Surnachev_HodgeVariableexponent} to establish estimates in variable Lebesgue spaces for the Hodge Laplacian.

	\paragraph*{General Hodge systems} Compared to this rich history, the general Hodge systems received very little attention. One might surmise that the reason for this is that these general cases are not as important and/or perhaps a somewhat routine generalization of the Hodge Laplacian. However, nothing could be further from the truth. Indeed, consider the time-harmonic Maxwell equation in a bounded domain in $\mathbb{R}^{3}$,   
	\begin{align*}
		\left\lbrace \begin{aligned}
			\operatorname{curl} H &= i\omega \varepsilon E + J_{e} &&\text{ in } \Omega, \\
			\operatorname{curl} E &= -i\omega \mu H + J_{m} &&\text{ in } \Omega, \\
			\nu \times E &= \nu \times E_{0} &&\text{ on } \partial\Omega. 
		\end{aligned}\right. \end{align*}
	Formally, this is equivalent to the following second order system in $E$, obtained by eliminating $H$ from the above equations 
	\begin{align*}
		\left\lbrace \begin{aligned}
			\operatorname{curl} \left( \mu^{-1}	\operatorname{curl} E \right)  &= \omega^{2} \varepsilon E -i\omega J_{e} + 	\operatorname{curl}\left(\mu^{-1}J_{m}\right) &&\text{ in } \Omega, \\
			\operatorname{div} \left( \varepsilon E \right) &= \frac{i}{\omega}\operatorname{div} J_{e} &&\text{ in } \Omega, \\
			\nu \times E &= \nu \times E_{0} &&\text{ on } \partial\Omega. 
		\end{aligned}\right. \end{align*}
	This is an instance of the general Hodge-Maxwell systems and is intimately related to the operator $$ \mathfrak{L}E := \operatorname{curl} \left( \mu^{-1}	\operatorname{curl} E \right) + \varepsilon^{\intercal} \nabla \operatorname{div} \left( \varepsilon E \right),$$   which is precisely our general Hodge operator when $n=3,$ $N=1$ and $k=1$ with $A= \mu^{-1}$ and $B= \varepsilon.$ The estimate for these systems also can not be obtained cheaply from the case of the Hodge Laplacian. Morrey's proof is based on potential theory and used the fact that $d^{\ast}d + dd^{\ast},$ i.e. the Hodge Laplacian is just the componentwise scalar Laplacian and after flattening the boundary, the boundary conditions $\nu \wedge u=0$ and $\nu \wedge d^{\ast}u=0$ implies that $u_{I}=0$  if $n \notin I$ and $\partial u_{I}/\partial \nu = 0$ if $n \in I.$ Thus, modulo certain error terms (due to localizing and flattening) which can be controlled, the whole system decouples into $\tbinom{n}{k}$ scalar Poisson problem, out of which $\tbinom{n-1}{k}$ has Dirichlet boundary conditions and the other  $\tbinom{n-1}{k-1}$ has Neumann boundary conditions. Other available proofs all use potential theory and verify either the Lopatinski\u{\i}-Shapiro condition (cf. Schwarz \cite{SchwarzHodge} for the verification, see also \cite{Lopatinskii}, \cite{Shapiro} and \cite{Visik_LScondition} for the condition)  or the Agmon-Douglis-Nirenberg condition (cf. Csato \cite{csatothesis} for the verification and \cite{ADN1}, \cite{ADN2}, \cite{Agranovich_ADNconditions} for the conditions) and these verifications also rely on the fact that the principal symbols of the operator and the boundary operators are rather simple and become prohibitively tedious if $A, B$ does not take values in scalar matrices.

	The first breakthrough to overcome this hurdle is Sil \cite{Sil_linearregularity}, where Schauder and $L^{p}$ estimates for the general Hodge systems is deduced using the Campanato method, completely avoiding potential theory. Estimates in Campanato spaces were derived directly, which imply the $BMO$-estimate. This coupled with the easy $L^{2}$ estimate, Stampacchia interpolation and duality imply the $L^{p}$ estimates. Due to the reliance on interpolation, this technique could not establish Morrey estimates. This was achieved in Sengupta-Sil \cite{SenguptaSil_MorreyLorentz_Hodge}, adapting an idea of Lieberman \cite{Lieberman_morrey_from_Lp}. However, the question of reaching weighted $L^{p}$ estimates remained open. 
	\subsection{Main techniques}	
	\paragraph*{Pointwise maximal inequalities and weighted estimates} In the present work, building on the ideas of \cite{Sil_linearregularity}, \cite{SenguptaSil_MorreyLorentz_Hodge}, we establish an appropriate \emph{pointwise maximal inequality} (see Lemma \ref{maximal ineq lemma}) for the Hessian of the solution near a boundary point, after suitable localization and flattening of the boundary, where we estimate the localized truncated sharp maximal function of the Hessian by localized  maximal functions of the error terms and the right hand side. 
	
	The strategy is to use these pointwise maximal estimates to derive weighted estimates via a Fefferman-Stein type inequality for weighted spaces, established by Phuc \cite[Corollary 2.7]{Phuc_Nguyen_WeightedFeffermanStein}. The basic scheme of establishing pointwise maximal inequalities to prove regularity goes back a long way and first appeared in the context of quasilinear equations, as no representation formula is available there. Iwaniec \cite{Iwaniec_Lpprojection} and later Dibenedetto-Manfredi \cite{DiBenedettoManfredi} used this idea to derive $L^{p}$ estimates for the gradient for quasilinear equations, see also Kinnunen-Zhou \cite{Kinnunen_Zhou_localestiNonlinear}, \cite{Kinnunen_Zhou_boundaryestiNonlinear}. This idea was used in the context of weighted estimates, again for quasilinear equations by Phuc \cite{Phuc_Nguyen_WeightedFeffermanStein} and recently in the context of boundary estimates for linear elliptic equations with discontinuous coefficients by Adimurthi-Mengesha-Phuc \cite{Karthik_Phuc2}. However, the work \cite{Karthik_Phuc2} deals only with gradient estimates for linear elliptic equations with the usual Dirichlet boundary conditions and uses the representation formula. In contrast, our method is completely free from potential theory and does not need any representation formula and work for Hessian estimates for \emph{systems} with the Hodge boundary conditions.

	\paragraph*{Decay estimates to pointwise maximal estimates} For this plan to work, the main technical task is twofold---first, one needs to sharpen the Campanato-style decay estimates, which are typically in the natural energy scales $L^{2}$-$L^{2}$, to obtain estimates in $L^{1}$-$L^{q}$ scale for any $1 < q < \infty.$  This is achieved in Lemma \ref{poinwise decay lemma}, using decay estimates established in \cite{Sil_linearregularity} , coupled with Poincar\'{e}-Sobolev type inequalities for the Hessian tailored to the Hodge boundary conditions derived in \cite{SenguptaSil_MorreyLorentz_Hodge} and the well-known self-improving properties of reverse H\"{o}lder inequalities. The second task is to have good control over the parameters so that the coefficient of the first term on the right in estimate \eqref{pointwise maximal ineq estimate} can be made arbitrarily small and use appropriate localizations such that this smallness can be used to absorb the dangerous term on the right in the left hand side  to derive the estimate \eqref{near the boundary weighted estimate}. This is quite delicate and is achieved by a double localization, one in the domain $\Omega$ and the other in the image of the flattening diffeomorphism in the half space, coupled with introducing and fine-tuning a number of parameters (namely $h$, $d$, $\bar{R}$, see Lemma \ref{poinwise decay lemma} and Lemma \ref{maximal ineq lemma}). The details can be found in Section \ref{boundary estimate section}, which is the technical core of this paper.

	In this work, we present only the apriori estimates in weighted Lebesgue spaces as the main fruit of our analysis. But we believe that the technique of \emph{using decay estimates to establish pointwise maximal function inequalities} is of independent interest, as these reduce the study of regularity of solutions to a classical topic, namely the \emph{mapping properties of maximal functions} and can be used even when a representation formula for the solution is either unavailable or difficult to obtain. Thus our method can also be used to prove regularity results \emph{in any function spaces where the boundedness of the relevant maximal functions are known}. We believe, in principle,  \emph{our extension of the Campanato approach is fully capable of substituting a representation formula} and can work for any type of linear boundary value problems for elliptic systems, as long as suitable decay estimates are available for the frozen constant coefficient system in the case of flat boundary.

	\subsection{Discussion of results}
	\paragraph*{Results and extensions}
	Our main result in the present article is a global apriori estimate for Hodge systems in weighted Lebesgue spaces (see Theorem \ref{mainthm} for a precise statement).	By standard approximation procedures, this immediately imply the regularity results in weighted Lebesgue spaces for 
	\begin{itemize}
		\item Hodge systems (Theorem \ref{Hodge system}),
		\item the Hodge-Maxwell systems (Theorem \ref{Hodge-Maxwell system}), 
		\item `div-curl' type systems (Theorem \ref{div-curl}),
	\end{itemize} 
	along with corresponding Gaffney-type inequalities (Theorem \ref{Gaffney ineq}) and Hodge decomposition theorems (Theorem \ref{Hodge decomposition}). To keep our presentation streamlined, we present the results for an open bounded $C^{r+2,1}$  domain $\Omega \subset \mathbb{R}^{n}$ with $r \ge 0$. But as our method uses only local estimates, all our results mentioned above remain valid, with essentially the same proof, if the domain is changed to a compact orientable $n$-dimensional Riemannian manifold $M$ with boundary $\partial M$ of class $C^{r+2,1},$ $r \ge 0.$ Also, we present only Hessian estimates for Hodge and Hodge-Maxwell systems, as they require more work. However, gradient estimates for these systems with the right hand side in divergence form  can be derived using our method with less work.

	\paragraph*{Regularity of the coefficient matrix fields}  In the present work, we restrict ourselves to regular enough coefficient fields and do not treat discontinuous coefficients. However, the assumed regularity on the coefficients are not expected to be sharp and the weighted $L^{p}$ estimates should be valid for a class of discontinuous coefficient fields as well. This is investigated in an upcoming work \cite{Mahato_Sil_discontinuousWeightedLp_InPrep}.

	\paragraph*{Extrapolation and weights} To conclude, we note that weighted estimates, apart from their independent interest, also serve as the starting point for deriving a variety of other  estimates in Orlicz and Orlicz-Musielak type spaces via Rubio de Francia type extrapolation theorems. To keep our presentation to a reasonable length, such applications would be presented in a different work \cite{Mahato_Sil_OrliczMusielak_InPrep}.

	\subsection{Organization} The rest of the article is organized as follows. Section \ref{notations} summarizes the our notations. Section \ref{prelim} collects the basic tools that we would use, including the relevant function spaces, several Poincar\'{e}-Sobolev type inequalities, properties of maximal functions, weights and weighted spaces, Fefferman-Stein type inequalities, properties of reverse H\"{o}lder inequalities  etc. Section \ref{main estimates} establishes our main estimates. Section \ref{main theorems} state and prove our main results as a consequence of these estimates.

	\section{Notations}\label{notations}
	We now fix the notations, for further details we refer to
	\cite{CsatoDacKneuss}, \cite{silthesis}, \cite{Sil_linearregularity} and \cite{SenguptaSil_MorreyLorentz_Hodge}. Let $n \geq 2,$ $N \geq 1$ and $0 \leq k \leq n$ be integers. 
	\paragraph*{Multilinear algebra} 
	\begin{itemize}
		\item We write $\Lambda^{k}\mathbb{R}^{n}$ to denote the vector space of all alternating $k-$linear maps
		$f:\underbrace{\mathbb{R}^{n}\times\cdots\times\mathbb{R}^{n}}_{k-\text{times}%
		}\rightarrow\mathbb{R}.$ For $k=0,$ we set $\Lambda^{0} \mathbb{R}%
		^{n}  =\mathbb{R}.$ Note that $\Lambda^{k}  \mathbb{R}%
		^{n}  =\{0\}$ for $k>n$ and, for $k\leq n,$ $\operatorname{dim}\left(
		\Lambda^{k}  \mathbb{R}^{n}  \right)  ={\binom{{n}}{{k}}}.$  Let $\left\{  e_{1},\cdots,e_{n}\right\}  $ stand for the standard basis of $\mathbb{R}%
		^{n}.$ Then the dual basis, denoted by $\left\{  e^{1},\cdots,e^{n}\right\} $ is a basis for $\Lambda^{1}\mathbb{R}^{n}$. Consequently,   
		$\left\{  e^{i_{1}}\wedge\cdots\wedge e^{i_{k}}:1\leq i_{1}<\cdots<i_{k}\leq
		n\right\}$	is a basis of $\Lambda^{k}\mathbb{R}^{n}.$ 
		
		\item $\wedge$ and $\left\langle \ ;\ \right\rangle $ denote the exterior product and the scalar product respectively, defined the usual way. The symbol $\ast$ denotes the Hodge star operator $\ast:\Lambda^{k} \rightarrow \Lambda^{n-k}$,  defined by 
		\begin{align*}
			\left\langle \ast \xi , \zeta \right\rangle e^{1}\wedge\ldots \wedge e^{n} = \xi \wedge \zeta \qquad \text{ for } \xi \in \Lambda^{k}, \zeta \in \Lambda^{n-k}.  
		\end{align*}
		The symbol $\lrcorner$  stands the  interior product, defined as 
		\begin{align*}
			\xi \lrcorner \zeta := \left(-1\right)^{n(k-l)} \ast \left( \xi \wedge \ast \zeta \right) \qquad \text{ for } \xi \in \Lambda^{l}, \zeta \in \Lambda^{k} \text{ with } 0 \le l \le k \le n.
		\end{align*}
		\item For any two finite dimensional vector spaces $X, Y,$ we use the notation $\operatorname{Hom}\left(X, Y \right)$ to denote the vector space of all linear maps from $X$ to $Y.$ When $X=Y,$ we shall just write $\operatorname{Hom}\left(X\right)$ instead.
		
		\item As we would be dealing with vector-valued forms a lot, we introduce some shorthand notation to avoid clutter. The integers $n \geq 2,$ $N \geq 1$ would remain fixed but arbitrary for the rest.  The only relevant point is the degree of the form.  To this end, for any integer  $0 \leq k \leq n,$ we denote 
		\begin{align*}
			\varLambda^{k}:= \Lambda^{k}\mathbb{R}^{n}\otimes \mathbb{R}^{N}.
		\end{align*}
		For $\alpha \in \varLambda^{k},$ it is easier to think of $\alpha$ as an $N$-tuple of $k$-forms, i.e. $$\alpha = \left( \alpha_{1}, \ldots, \alpha_{N}\right),$$ where $\alpha_{1}, \ldots, \alpha_{N}$ are usual (scalar-valued) exterior forms.

		\item The symbols $\wedge,$ $\lrcorner\,,$ $\left\langle \ ;\ \right\rangle $ and,
		respectively, $\ast$ continue to denote the exterior product, the interior product, the
		scalar product and, respectively, the Hodge star operator, extended componentwise in the obvious fashion to vector-valued forms. 
		More precisely, for $\alpha = \left( \alpha_{1}, \ldots, \alpha_{N}\right) \in \varLambda^{k_{1}}$ and $\beta = \left( \beta_{1}, \ldots, \beta_{N}\right) \in \varLambda^{k_{2}},$ we have 
		\begin{align*}
			\alpha \wedge \beta &= \left( \alpha_{1}\wedge \beta_{1}, \ldots, \alpha_{N}\wedge \beta_{N}\right) \in \varLambda^{k_{1}+k_{2}}, \\
			\alpha \lrcorner \beta &= \left( \alpha_{1}\lrcorner \beta_{1}, \ldots, \alpha_{N}\lrcorner \beta_{N}\right) \in \varLambda^{k_{2}-k_{1}},\\
			\ast \alpha &= \left( \ast \alpha_{1}, \ldots, \ast \alpha_{N}\right) \in \varLambda^{n-k_{1}}, \\ 
			\left\langle \alpha , \beta \right\rangle &= \sum\limits_{i=1}^{N} 	\left\langle\alpha_{i}, \beta_{i}	\right\rangle_{\Lambda^{k}}, \qquad \text{ if } k_{1}=k_{2},  
		\end{align*}
		where $\left\langle\cdot, \cdot	\right\rangle_{\Lambda^{k}}$ denote the scalar product for usual (scalar-valued) exterior forms.

		\item When $\xi \in \mathbb{R}^{n}$ is a vector and $\alpha \in \varLambda^{k}$, with a slight abuse of notation, we would write 
		\begin{align*}
			\xi \wedge \alpha  = \left(\tilde{\xi} \wedge \alpha_{1}, \ldots, \tilde{\xi} \wedge \alpha_{N}\right) \quad \text{ and } \quad \xi \lrcorner \alpha  = \left(\tilde{\xi} \lrcorner \alpha_{1}, \ldots, \tilde{\xi} \lrcorner \alpha_{N}\right),
		\end{align*}
		where $\tilde{\xi} = \sum_{i=1}^{n} \xi_{i} e^{i}$ is the exterior $1$-form with the same components as $\xi = (\xi_{1}, \ldots, \xi_{n}). $ 
	\end{itemize}
	
	\paragraph*{Differential forms}
	For us, $\Omega \subset \mathbb{R}^{n}$ will always be assumed to be an open, bounded subset with at least $C^{2,1}$ boundary. More regularity of the boundary will be explicitly specified as necessary. $\nu$ will always denote the outward pointing unit normal field to $\partial\Omega,$ which would be identified with abuse of notation, with the $1$-form with same components $\nu = \sum_{i=1}^{n} \nu_{i}(x)dx_{i}.$ 
	\begin{definition}
		An $\mathbb{R}^{N}$-valued differential $k$-form $\omega$ on $\Omega$ is a measurable function $\omega:\Omega\rightarrow\varLambda^{k}.$
	\end{definition} The usual Lebesgue, Sobolev and H\"{o}lder spaces are defined componentwise in the usual way and are denoted by their usual symbols. Two important differential operators on differential forms are
	\begin{definition}
		A $\mathbb{R}^{N}$-valued 
		$(k+1)$-form $\varphi\in L^{1}_{\text{loc}}(\Omega;\varLambda^{k+1})$
		is called the exterior derivative of $\omega\in
		L^{1}_{\text{loc}}\left(\Omega;\varLambda^{k}\right),$ denoted by $d\omega$,  if
		\begin{align*}
			\int_{\Omega} \eta\wedge\varphi=(-1)^{k}\int_{\Omega} d\eta\wedge\omega, \qquad \text{ for all } \eta\in C^{\infty}_{c}\left(\Omega;\varLambda^{n-k-1}\right),
		\end{align*}
		where $d\eta$ is the classical exterior derivative, defined in the usual manner. 
		The Hodge codifferential of $\omega\in L^{1}_{\text{loc}}\left(\Omega;\varLambda^{k}\right)$ is
		an $\mathbb{R}^{N}$-valued $(k-1)$-form, denoted $d^{\ast}\omega\in L^{1}_{\text{loc}}\left(\Omega;\varLambda^{k-1}\right),$
		defined as
		$$
		d^{\ast}\omega:=(-1)^{n(k+1)} \ast d \ast \omega. $$
	\end{definition}
	Note that our definitions imply that 
	\begin{align*}
		\int_{\Omega} \left\langle  \eta, d\omega \right\rangle =  -\int_{\Omega} \left\langle d^{\ast} \eta, \omega \right\rangle, \qquad \text{ for all } \eta\in C^{\infty}_{c}\left(\Omega;\varLambda^{k+1}\right). 
	\end{align*}
	In general, we have the integration by parts formula 
	\begin{align*}
		\int_{\Omega} \left\langle   d\alpha, \beta \right\rangle + 	\int_{\Omega} \left\langle   \alpha, d^{\ast}\beta \right\rangle = \int_{\partial\Omega}\left\langle   \nu \wedge \alpha, \beta \right\rangle = \int_{\partial\Omega}\left\langle    \alpha, \nu \lrcorner \beta \right\rangle
	\end{align*}
	for every $\alpha \in C^{1}\left(\overline{\Omega}; \varLambda^{k}\right)$, $\beta \in C^{1}\left(\overline{\Omega}; \varLambda^{k+1}\right)$. See \cite{CsatoDacKneuss} for the properties of  these operators. The operator $d^{\ast}d + d d^{\ast}$ is the Hodge Laplacian and it is easy to see that 
	\begin{align*}
		\left(d^{\ast}d + d d^{\ast}\right)\alpha = \Delta \alpha = \operatorname{div}\left( \nabla \alpha\right) \quad \text{ for all } \alpha \in C^{2}\left(\overline{\Omega}; \varLambda^{k}\right),
	\end{align*}
	where the Laplacian on the right is understood to act componentwise. 
	
	\paragraph*{Notations for sets}
	\begin{itemize}
		\item For any $z \in \mathbb{R}^{n}$ and any $r>0,$ the open ball with center $z$ and radius $r$ is denoted by $B_{r}\left( z\right) := \left\lbrace x \in \mathbb{R}^{n}: \left\lvert x - z \right\rvert < r\right\rbrace$. We would just write $B_{r}$ when the center of the ball is the origin, i.e. when $z=0 \in \mathbb{R}^{n}.$
		\item The open upper half space is denoted by 
		\begin{align*}
			\mathbb{R}^{n}_{+} := \left\lbrace x = \left(x', x_{n}\right) \in \mathbb{R}^{n}: x_{n} >0\right\rbrace, 
		\end{align*} The boundary of the open upper half space is denoted as 
		\begin{align*}
			\partial\mathbb{R}^{n}_{+} := \left\lbrace x = \left(x', 0\right) \in \mathbb{R}^{n}: x' \in \mathbb{R}^{n-1}\right\rbrace. 
		\end{align*} 
		
		\item For any $z \in \partial\mathbb{R}^{n}_{+}$ and any $r>0,$ the open upper half ball with center $z$ and radius $r$ are denoted by $$B^{+}_{r}\left( z\right) := \left\lbrace x = \left(x', x_{n}\right) \in \mathbb{R}^{n}: \left\lvert x - z \right\rvert < r , x_{n} >0 \right\rbrace. $$  
		We would just write $B^{+}_{r}$ when the center of the balls is the origin, i.e. when $z=0 \in \mathbb{R}^{n}.$ For us, $\Gamma_{r}\left( z\right)$  and $\Sigma_{r}\left( z\right)$ would denote the flat part and the curved part, respectively, of the boundary of the half ball $B_{r}^{+}\left( z\right).$ More precisely, 
		\begin{align*}
			\Gamma_{r}\left( z\right):= \partial B_{r}^{+}\left( z\right) \cap 	\partial\mathbb{R}^{n}_{+}\qquad  \text{ and } \qquad \Sigma_{r} := \partial B_{r}^{+}\left( z\right)\setminus \Gamma_{r}\left( z\right).  
		\end{align*}
		By a mild abuse of notation, we would also use $B_{r}^{+}\left(z\right)$ to denote 
		\begin{align*}
			B_{r}^{+}\left(z\right) := B_{r}\left(z\right) \cap \mathbb{R}^{n}_{+} \quad \text{ even when } z \notin \partial\mathbb{R}^{n}_{+}.
		\end{align*} 
		\item  For any  open subset $\Omega\subset\mathbb{R}^n,$ and for any $z \in \mathbb{R}^{n}$ and any $r>0,$ we denote 
		\begin{align*}
			\Omega_{\left(r, z\right)}:= B_{r}\left( z\right) \cap \Omega \quad \text{ and } \quad 	\Omega^{+}_{\left(r, z\right)} := \Omega \cap B_{r}^{+}\left(z\right) (\text{when } z \in \partial\mathbb{R}^{n}_{+}). 
		\end{align*}
		We suppress writing the center and just write $\Omega_{\left(r\right)}$ and $\Omega^{+}_{\left(r\right)}$ when $z=0 \in \mathbb{R}^{n}.$
		
		\item Let $\mathcal{U}\subset\mathbb{R}_+^n$ be a smooth open set which is star-shaped about the origin such that 
		\begin{align*}
			B^{+}_{1/2} \subset B^{+}_{3/4} \subset \mathcal{U} \subset B^{+}_{7/8} \subset B^{+}_{1}.  
		\end{align*}
		Note that this implies $\mathcal{U}$ is contractible and $	\Gamma_{3/4} \subset \partial\mathcal{U}.$	For any $z \in \partial \mathbb{R}^{n}_{+},$ we set 
		\begin{align*}
			\mathcal{U}_{R}\left(z\right) :=  \left\lbrace z + R x : x \in \mathcal{U} \right\rbrace = \left\lbrace x \in \mathbb{R}_{+}^{n}: \frac{1}{R}\left( x - z\right) \in \mathcal{U} \right\rbrace. 
		\end{align*}
		We also write $\mathcal{U}_{R} := \mathcal{U}_{R}\left(0\right).$ The notation $\mathcal{C}_{R}(z)$ denotes the curved part of $\partial \mathcal{U}_{R}\left(z\right).$ 
	\end{itemize}
	\paragraph*{Measurable sets and integral means}
	For any Lebesgue measurable subset $A \subset \mathbb{R}^{n}, $ we denote its $n$-dimensional Lebesgue measure  by $\left\lvert A \right\rvert.$ If $\left\lvert A \right\rvert < \infty,$  we use the notation $\left( \cdot \right)_{A}$ to denote the integral average over the set $A$, i.e.  
	$$ \left( f \right)_{A} := \frac{1}{\left\lvert A \right\rvert } \int_{A} f\left(x\right)\ \mathrm{d}x := \fint_{A} f\left(x\right)\ \mathrm{d}x \qquad \text{ for any } f \in L^{1}\left( A \right).$$ This notation is also extended componentwise to vector-valued functions. We record an important result. 
	\begin{lemma}[minimality of mean]\label{minimality of mean}
		Let $A \subset \mathbb{R}^{n}$ be measurable with $\left\lvert A \right\rvert < \infty.$ For any $1 \le q < \infty$ and any $f \in L^{q}\left(A; \mathbb{R}^{m}\right),$ we have 
		\begin{align*}
			\int_{A} \left\lvert f - \left(f\right)_{A}\right\rvert^{q} \leq c_{m,q} 	\int_{A} \left\lvert f - \xi\right\rvert^{q} \qquad \text{ for any } \xi \in \mathbb{R}^{m},
		\end{align*}
		for some constant $c_{m,q} \ge 1,$ depending only on $m \in \mathbb{N}$ and $q$.
	\end{lemma}
	
	\paragraph*{Ellipticity notions}
	\begin{definition}\label{legendre-hadamard condition}
		A linear map $A: \mathbb{R}^{n}\otimes \mathbb{R}^{d} \rightarrow \mathbb{R}^{n}\otimes \mathbb{R}^{d}$ is called \textbf{Legendre-Hadamard} elliptic if there exists a constant $ \gamma >0$ such that 
		$$ \langle A (a\otimes b) \  ;\  a\otimes b \rangle  \geq \gamma \left\vert  a \right\vert^{2}\left\lvert b \right\vert^{2}, 
		\qquad  \text{ for every } a \in \mathbb{R}^{n}, b \in \mathbb{R}^{d} . $$
	\end{definition}
	\begin{definition}\label{legendre condition}
		A bounded measurable map $A \in L^{\infty}\left( \Omega; \operatorname{Hom}(\varLambda^{k}) \right)$ is called \textbf{uniformly Legendre elliptic} if there exists a constant $ 0 < \gamma \le 1$ such that  we have 
		$$ \gamma \left\vert  \xi \right\vert^{2} \le \langle A (x)  \xi \  ;\  \xi \rangle  \le \frac{1}{\gamma} \left\vert  \xi \right\vert^{2}
		\qquad  \text{ for a.e. }  x \in \Omega, \text{ for every } \xi \in \varLambda^{k}.$$ 
		
	\end{definition}
	\section{Preliminaries}\label{prelim} 
	Let $n \geq 2, 1 \leq k \leq n-1$ and $N \geq 1$ be integers. This assumption will be in force throughout the rest of the article and will be used without further comment. 	
	\subsection{Algebraic lemmas} 
	Let $\bar{A} \in \operatorname{Hom}\left( \varLambda^{k+1}\right)$ and $\bar{B} \in \operatorname{Hom}\left( \varLambda^{k}\right)$ be given maps. Define the linear maps $M_{T}, M_{N} \in \operatorname{Hom}\left( \varLambda^{k}\otimes \mathbb{R}^{n} \right)$ by the pointwise algebraic identities: For any $a_{1}, a_{2} \in \mathbb{R}^{n}$ and any $b_{1}, b_{2} \in \varLambda^{k}$,  
	\begin{align*}
		\left\langle M_{T} \left(a_{1}\otimes b_{1}\right), a_{2} \otimes b_{2}\right\rangle &= \left\langle \bar{A} \left(a_{1}\wedge b_{1}\right), a_{2} \wedge b_{2}\right\rangle +  \left\langle  a_{1}\lrcorner \bar{B}b_{1}, a_{2} \lrcorner \bar{B}b_{2}\right\rangle, \\
		\left\langle M_{N} \left(a_{1}\otimes b_{1}\right), a_{2} \otimes b_{2}\right\rangle &= \left\langle \bar{A} \left(a_{1}\wedge \bar{B}^{-1}b_{1}\right), a_{2} \wedge \bar{B}^{-1}b_{2}\right\rangle +  \left\langle  a_{1}\lrcorner b_{1}, a_{2} \lrcorner b_{2}\right\rangle. 
	\end{align*}
	Then we have 
	\begin{align*}
		\left\langle M_{T} \left(a\otimes b\right), a \otimes b\right\rangle &= \left\langle \bar{A} \left(a\wedge b\right), a \wedge b\right\rangle +  \left\lvert   a\lrcorner \bar{B}b\right\rvert^{2}, \\
		\left\langle M_{N} \left(a\otimes b\right), a \otimes b\right\rangle &= \left\langle \bar{A} \left(a\wedge \bar{B}^{-1}b\right), a \wedge \bar{B}^{-1}b\right\rangle +   \left\lvert   a\lrcorner b\right\rvert^{2}.
	\end{align*}
	\begin{lemma}\label{ellipticity lemma}
		Let $\bar{B}:\varLambda^{k} \rightarrow \varLambda^{k}$  and $\bar{A}:\varLambda^{k+1} \rightarrow \varLambda^{k+1}$ satisfy 
		$$ \gamma_{\bar{B}} \left\lvert b \right\rvert^{2} \le \left\langle Bb, b \right\rangle \le \frac{1}{\gamma_{\bar{B}}} \left\lvert b \right\rvert^{2} \qquad \text{ and }\qquad   \left\langle A\left( a \wedge b\right), a \wedge b \right\rangle \ge \gamma_{\bar{A}} \left\lvert a \wedge b \right\rvert^{2}$$  for all $a \in \mathbb{R}^{n}, b \in \varLambda^{k}$ for some constants $\gamma_{\bar{A}}, \gamma_{\bar{B}} \in (0, 1]$. Then $M_{T}, M_{N}$ are both Legendre-Hadamard elliptic and there exists a constant $c>0,$ depending only on the numbers $n, k, N, \gamma_{\bar{A}}, \gamma_{\bar{B}},$ such that 
		\begin{align*}
			\left\langle M_{T}\left( \xi \otimes \eta \right), \xi \otimes \eta \right\rangle \ge c \left\lvert \xi \right\rvert^{2}\left\lvert \eta \right\rvert^{2}\qquad \text{ for all } \xi \in \mathbb{R}^{n}, \eta \in \varLambda^{k}, \\
			\left\langle M_{N}\left( \xi \otimes \eta \right), \xi \otimes \eta \right\rangle \ge c \left\lvert \xi \right\rvert^{2}\left\lvert \eta \right\rvert^{2}\qquad \text{ for all } \xi \in \mathbb{R}^{n}, \eta \in \varLambda^{k}.  
		\end{align*}
	\end{lemma}
	\begin{proof}
		By definition of $M_{T}$, we have 
		\begin{align*}
			\left\langle M_{T}\left( \xi \otimes \eta \right), \xi \otimes \eta \right\rangle = \left\langle \bar{A} \left(\xi\wedge \eta\right), \xi \wedge \eta\right\rangle +  \left\lvert   \xi\lrcorner \bar{B}\eta\right\rvert^{2}\ge \gamma_{\bar{A}}\left\lvert   \xi \wedge \eta\right\rvert^{2} +  \left\lvert   \xi \lrcorner \bar{B}\eta\right\rvert^{2}. 
		\end{align*}
		Thus it remains to prove that there exists a constant $c>0$ such that 
		$$ \gamma_{\bar{A}}\left\lvert   \xi \wedge \eta\right\rvert^{2} +  \left\lvert   \xi \lrcorner \bar{B}\eta\right\rvert^{2}
		\ge c \left\lvert \xi \right\rvert^{2}\left\lvert \eta \right\rvert^{2}. $$
		If the claim is false, then there exist sequences $\left\lbrace \xi_{s} \right\rbrace_{s \in \mathbb{N}}$, and $\left\lbrace \eta_{s} \right\rbrace_{s \in \mathbb{N}}$ such that 
		$$ \left\lvert \xi_{s} \right\rvert = 1 =  \left\lvert \eta_{s} \right\rvert \qquad \text{ for all } s \in \mathbb{N},$$
		and 
		$$ \gamma_{\bar{A}}\left\lvert   \xi_{s} \wedge \eta_{s}\right\rvert^{2} +  \left\lvert   \xi_{s} \lrcorner \bar{B}\eta_{s}\right\rvert^{2}  < \frac{1}{s} \qquad \text{ for all } s \in \mathbb{N}.$$ This implies, passing to a subsequence if necessary, that there exist $\xi, \eta$ such that $$ \xi_{s} \rightarrow \xi, \qquad \eta_{s} \rightarrow \eta \qquad \text{ with } \left\lvert \xi \right\rvert = 1 =  \left\lvert \eta \right\rvert,$$ such that 
		$$  \xi \wedge \eta = 0 \qquad \text{ and } \qquad  \xi \lrcorner \bar{B}\eta=0. $$
		But this implies 
		\begin{align*}
			0 = \left\langle  \xi \lrcorner \bar{B}\eta,  \xi \lrcorner \eta\right\rangle =  \left\langle \bar{B}\eta, \xi \wedge \left( \xi \lrcorner \eta\right) \right\rangle. 
		\end{align*}
		Observe that we have the identity 
		\begin{align}\label{1form wedge interior identity}
			|a|^{2}b = a\lrcorner \left( a \wedge b \right) + a \wedge \left( a \lrcorner b\right) \qquad \text{ for all } a \in \mathbb{R}^{n}, b \in \varLambda^{k}.
		\end{align}
		For $N=1$, this follows from \cite[Equation (2.7), Page 39]{CsatoDacKneuss}, which extends componentwise to the vector-valued case. Using this, we get 
		\begin{align*}
			0 =  \left\langle \bar{B}\eta, \xi \wedge \left( \xi \lrcorner \eta \right)\right\rangle = \left\langle \bar{B}\eta, \left\lvert \xi \right\rvert^{2} \eta\right\rangle = \left\lvert \xi \right\rvert^{2}\left\langle \bar{B}\eta, \eta\right\rangle \ge \gamma_{\bar{B}}\left\lvert \eta\right\rvert^{2} = \gamma_{\bar{B}} >0, 
		\end{align*}
		where we used the fact that $\left\lvert \xi \right\rvert = 1 = \left\lvert v \right\rvert.$ This contradiction proves the claim for $M_{T}.$ For $M_{N}$, we get \begin{align*}
			\left\langle M_{N}\left( \xi \otimes \eta \right), \xi \otimes \eta \right\rangle &= \left\langle \bar{A} \left(\xi\wedge \bar{B}^{-1}\eta\right), \xi \wedge \bar{B}^{-1}\eta\right\rangle +  \left\lvert   \xi\lrcorner \eta\right\rvert^{2}\\&\ge \gamma_{\bar{A}}\left\lvert   \xi \wedge \bar{B}^{-1}\eta\right\rvert^{2} +  \left\lvert   \xi \lrcorner \eta\right\rvert^{2}. 
		\end{align*}
		We claim that there exists a constant $c>0$ such that 
		$$ \gamma_{\bar{A}}\left\lvert   \xi \wedge \bar{B}^{-1}\eta\right\rvert^{2} +  \left\lvert   \xi \lrcorner \eta\right\rvert^{2}
		\ge c \left\lvert \xi \right\rvert^{2}\left\lvert \eta \right\rvert^{2}. $$
		If the claim is false, then there exist sequences $\left\lbrace \xi_{s} \right\rbrace_{s \in \mathbb{N}}$, and $\left\lbrace \eta_{s} \right\rbrace_{s \in \mathbb{N}}$ such that 
		$$ \left\lvert \xi_{s} \right\rvert = 1 =  \left\lvert \eta_{s} \right\rvert \qquad \text{ for all } s \in \mathbb{N},$$
		and 
		$$ \gamma_{\bar{A}}\left\lvert   \xi_{s} \wedge \bar{B}^{-1}\eta_{s}\right\rvert^{2} +  \left\lvert   \xi_{s} \lrcorner \eta_{s}\right\rvert^{2}  < \frac{1}{s} \qquad \text{ for all } s \in \mathbb{N}.$$ This implies, passing to a subsequence if necessary, that there exist $\xi, \eta$ such that $$ \xi_{s} \rightarrow \xi, \qquad \eta_{s} \rightarrow \eta \qquad \text{ with } \left\lvert \xi \right\rvert = 1 =  \left\lvert \eta \right\rvert,$$ such that 
		$$  \xi \wedge \bar{B}^{-1}\eta = 0 \qquad \text{ and } \qquad  \xi \lrcorner \eta=0. $$
		But once again  using the identity \eqref{1form wedge interior identity}, this implies   
		\begin{align*}
			0 = \left\langle  \xi \wedge \bar{B}^{-1}\eta,  \xi \wedge \eta\right\rangle =  \left\langle \bar{B}^{-1}\eta, \xi \lrcorner \left( \xi \wedge \eta\right) \right\rangle = \left\langle \bar{B}^{-1}\eta, \eta \right\rangle \le \gamma_{\bar{B}} >0. 
		\end{align*}
		This contradiction completes the proof. 
	\end{proof}

	\subsection{Function spaces}
	We define the spaces $W^{d,2}\left(  \Omega;\varLambda^{k}\right)$ and  $W^{d^{\ast},2}\left(  \Omega;\varLambda^{k}\right)  $ as 
	\begin{align*}
		W^{d,2}\left(  \Omega;\varLambda^{k}\right) &= \left\lbrace \omega \in L^{2}\left(  \Omega;\varLambda^{k}\right): du \in L^{2}\left(  \Omega;\varLambda^{k+1}\right) \right\rbrace, \\
		W^{d^{\ast},2}\left(  \Omega;\varLambda^{k}\right) &= \left\lbrace \omega \in L^{2}\left(  \Omega;\varLambda^{k}\right): d^{\ast}u \in L^{2}\left(  \Omega;\varLambda^{k+1}\right)  \right\rbrace. 
	\end{align*}
	Note that $W^{d,2}\left(  \Omega;\varLambda^{k}\right)$ is just $L^{2}\left(  \Omega;\varLambda^{k}\right)$ when $k=n$ and $W^{d^{\ast},2}\left(  \Omega;\varLambda^{k}\right)$ is just $L^{2}\left(  \Omega;\varLambda^{k}\right)$ when $k=0.$
	We would also need the following subspaces 
	\begin{align*}
		W^{d,2}_{T}\left(  \Omega;\varLambda^{k}\right) &= \left\lbrace \omega \in W^{d, 2}\left(  \Omega;\varLambda^{k}\right):  \nu\wedge\omega=0\text{ on
		}\partial\Omega \right\rbrace, \\
		W^{d^{\ast},2}_{N}\left(  \Omega;\varLambda^{k}\right) &= \left\lbrace \omega \in W^{d^{\ast}, 2}\left(  \Omega;\varLambda^{k}\right):  \nu\lrcorner\omega=0\text{ on
		}\partial\Omega \right\rbrace. 
	\end{align*}
	The boundary conditions are to be interpreted in the trace sense. 
	Similarly, we define 
	\begin{align*}
		W_{T}^{1,2}\left(  \Omega;\varLambda^{k}\right)  &=\left\{  \omega\in
		W^{1,2}\left(  \Omega;\varLambda^{k}\right)  :\nu\wedge\omega=0\text{ on
		}\partial\Omega\right\}, \\ W_{N}^{1,2}\left(  \Omega;\varLambda^{k}\right)  &=\left\{  \omega\in
		W^{1,2}\left(  \Omega;\varLambda^{k}\right)  :\nu\lrcorner\omega=0\text{ on
		}\partial\Omega\right\}. 
	\end{align*}
	We also need the following subspaces 
	\begin{align*}
		W_{d^{\ast}, T}^{1,2}(\Omega; \varLambda^{k}) &= \left\lbrace \omega \in W_{T}^{1,2}(\Omega; \varLambda^{k}) : d^{\ast}\omega = 0 \text{ in }
		\Omega \right\rbrace, \\
		W_{d, N}^{1,2}(\Omega; \varLambda^{k}) &= \left\lbrace \omega \in W_{N}^{1,2}(\Omega; \varLambda^{k}) : d\omega = 0 \text{ in }
		\Omega \right\rbrace, \\
		\mathcal{H}^{k}_{T}\left(  \Omega;\varLambda^{k}\right)  &=\left\{  \omega\in
		W_{T}^{1,2}\left(  \Omega;\varLambda^{k}\right)  :d\omega=0\text{ and }%
		d^{\ast}\omega=0\text{ in }\Omega\right\}, \\
		\mathcal{H}^{k}_{N}\left(  \Omega;\varLambda^{k}\right)  &=\left\{  \omega\in
		W_{N}^{1,2}\left(  \Omega;\varLambda^{k}\right)  :d\omega=0\text{ and }%
		d^{\ast}\omega=0\text{ in }\Omega\right\}. 
	\end{align*}
	For any $p>1,$ we define  
	\begin{align*}
		W^{1,p}&\left(  \Omega;\varLambda^{k}\right)\cap \mathcal{H}^{\perp}_{T} \\&:= \left\lbrace u \in W^{1,p}\left(  \Omega;\varLambda^{k}\right): \int_{\Omega} \left\langle u, h \right\rangle = 0 \quad \text{ for all } h \in \mathcal{H}^{k}_{T}\left(  \Omega;\varLambda^{k}\right) \right\rbrace. 
	\end{align*}
	Similarly, we set 
	\begin{align*}
		W^{1,p}&\left(  \Omega;\varLambda^{k}\right)\cap \mathcal{H}^{\perp}_{N} \\&:= \left\lbrace u \in W^{1,p}\left(  \Omega;\varLambda^{k}\right): \int_{\Omega} \left\langle u, h \right\rangle = 0 \quad \text{ for all } h \in \mathcal{H}^{k}_{N}\left(  \Omega;\varLambda^{k}\right) \right\rbrace. 
	\end{align*}
	We also need the following subspaces. 
	\begin{align*}
		W&_{T, \text{flat}}^{1,2}(B_{R}^{+}\left(x\right) ; \varLambda^{k}) \\&= \left\lbrace \psi \in W^{1,2}(B_{R}^{+}\left(x\right) ; \varLambda^{k}):
		e_n \wedge \psi = 0 \text{ on }  \Gamma_{R}\left(x\right), \psi = 0 \text{ near } \Sigma_{R}\left(x\right) \right\rbrace ,  
	\end{align*}
	\begin{align*}
		W&_{T, \text{flat}}^{1,2}(\mathcal{U}_{R}\left(x\right) ; \varLambda^{k}) \\&= \left\lbrace \psi \in W^{1,2}(\mathcal{U}_{R}\left(x\right) ; \varLambda^{k}):
		e_n \wedge \psi = 0 \text{ on }  \Gamma_{R}\left(x\right), \psi = 0 \text{ near } \mathcal{C}_{R}\left(x\right) \right\rbrace.  
	\end{align*}
	Similarly, we define 
	\begin{align*}
		W&_{N, \text{flat}}^{1,2}(B_{R}^{+}\left(x\right) ; \varLambda^{k}) \\&= \left\lbrace \psi \in W^{1,2}(B_{R}^{+}\left(x\right) ; \varLambda^{k}):
		e_n \lrcorner \psi = 0 \text{ on }  \Gamma_{R}\left(x\right), \psi = 0 \text{ near } \Sigma_{R}\left(x\right) \right\rbrace,   
	\end{align*}
	\begin{align*}
		W&_{N, \text{flat}}^{1,2}(\mathcal{U}_{R}\left(x\right) ; \varLambda^{k}) \\&= \left\lbrace \psi \in W^{1,2}(\mathcal{U}_{R}\left(x\right) ; \varLambda^{k}):
		e_n \lrcorner \psi = 0 \text{ on }  \Gamma_{R}\left(x\right), \psi = 0 \text{ near } \mathcal{C}_{R}\left(x\right) \right\rbrace.   
	\end{align*}
	\subsection{\texorpdfstring{$A_{p}$}{Ap} weights and Weighted Sobolev spaces}	
	\begin{definition}
		A locally integrable function $w \in L^{1}_{\text{loc}}\left(\mathbb{R}^{n}\right)$ is called a \emph{weight} if $w\left(x\right) >0$ for a.e. $x \in \mathbb{R}^{n}.$ We shall denote the corresponding measure $w\left(x\right)\mathrm{d}x$ also by $w.$
	\end{definition}
	For any given weight $w$ and any $1< p < \infty,$ we define the weighted Lebesgue space as
	\begin{align*}
		L^p_w(\Omega)&=\left\{f:\Omega \rightarrow \mathbb{R} \text{ is measurable}:\|f\|^p_{L^p_w(\Omega)}=\int_{\Omega}|f(x)|^pw(x)dx<\infty\right\}. 
	\end{align*}
	The weighted Sobolev space $W^{1,p}_{w}$ is defined as  
	\begin{align*}
		W^{1,p}_w(\Omega)&=\left\{f\in L^p_w(\Omega):\nabla f\in L^p_w(\Omega)\right\} 
	\end{align*} equipped with the norm 
	\begin{align*}
		\|f\|^p_{W^{1,p}_w(\Omega)}&=\|f\|^p_{L^p_w(\Omega)}+\|\nabla f\|^p_{L^p_w(\Omega)}.
	\end{align*} Similarly, the space $W^{2,p}_{w}$ is defined as  
	\begin{align*}
		W^{2,p}_w(\Omega)&=\left\{f\in W^{1,p}_w(\Omega):\nabla^{2} f\in L^p_w(\Omega)\right\}
	\end{align*} equipped with the norm 
	\begin{align*}
		\|f\|^p_{W^{2,p}_w(\Omega)}&=\|f\|^p_{L^p_w(\Omega)}+\|\nabla f\|^p_{L^p_w(\Omega)} + +\|\nabla^{2} f\|^p_{L^p_w(\Omega)}. 
	\end{align*}
	Higher order Sobolev spaces, i.e. $W^{r,p}_{w}$ for $r \geq 2$, are defined similarly. 
	\begin{definition}
		Let $1\leq p < \infty.$ A weight $w$ is said to be an $A_{p}$ weight if there exists a positive constant $A$ such that for every ball $B \subset \mathbb{R}^{n}$, we have 
		\begin{align*}
			\left\lbrace \begin{aligned}
				&\left( \fint_{B} w\  \mathrm{d}x\right)\left( \fint_{B} w^{-\frac{1}{(p-1)}}\ \mathrm{d}x\right)^{p-1} \leq A  &&\text{ if } p >1, \\
				&\left( \fint_{B} w\  \mathrm{d}x\right)\left( \sup\limits_{x \in B} w \right) \leq A &&\text{ if } p =1. 
			\end{aligned}\right. 
		\end{align*}
		The smallest constant $A$ will be denoted by $\left[w\right]_{A_{p}}.$
	\end{definition}
	We recall the well-known reverse H\"{o}lder property of $A_{p}$ weights for $p >1.$  
	\begin{proposition}
		Let $M_{0}>0$ and let $1< p < \infty.$ Then there exist constants $C =C \left(n, p , M_{0}\right)>0$ and $\varepsilon = \varepsilon \left(n, p , M_{0}\right) >0$ such that for every $w \in A_{p}$ with $\left[w\right]_{A_{p}} \leq M_{0},$ we have 
		\begin{align*}
			\left( \fint_{B} w^{1+\varepsilon} \ \mathrm{d}x\right)^{\frac{1}{1+\varepsilon}} \leq C  \left( \fint_{B} w \ \mathrm{d}x \right) \qquad \text{ for every ball } B \subset \mathbb{R}^{n}. 
		\end{align*}
	\end{proposition}
	See \cite[Theorem 9.2.2]{Grafakos_modernFourier} for a proof. As  a consequence, we have the following results. The proofs are standard and is skipped.     
	\begin{proposition}\label{higher power Ap}
		Let $M_{0}>0$ and let $1< p < \infty.$ Then there exist  $p_{0} = p_{0} \left(n, p , M_{0}\right) \in (1, q)$ such that for every $w \in A_{p}$ with $\left[w\right]_{A_{p}} \leq M_{0},$ we have $w \in A_{p_{0}}.$ 
	\end{proposition}
	\begin{proposition}
		Let $M_{0}>0$ and let $1< p < \infty.$ Then there exist  $p_{0} = p_{0} \left(n, p , M_{0}\right) \in (1, q)$ such that for any $w \in A_{p}$ with $\left[w\right]_{A_{p}} \leq M_{0}$ and for any open, bounded subset $\Omega \subset \mathbb{R}^{n},$ we have the continuous embedding 
		\begin{align*}
			L^{p}_{w}\left( \Omega \right) \hookrightarrow L^{p_{0}}\left( \Omega \right).   
		\end{align*}
	\end{proposition}
	\begin{remark}
		Note that $p_{0}$ depends only on $n, p$ and $M_{0}$ and not on $w$ or $\Omega.$ However, the constant in the continuity estimate depends on $w$ and $\Omega.$
	\end{remark}
	\subsection{Maximal functions}	
	\begin{definition}
		Suppose $f\in L^1_{\text{loc}}(\mathbb{R}^n),$ then its Hardy-Littlewood maximal function, denoted by $Mf,$ defined point-wise as \begin{align*}
			Mf(x)=\sup_{r>0}\fint_{B(x,r)}|f(y)|dy
		\end{align*}
		and the Fefferman-Stein sharp maximal function, denoted by $M^\sharp f,$ defined point-wise as \begin{align*}
			M^\sharp f(x)=\sup_{x\in B} \fint_{B}\left|f(y)-(f)_B\right|dy
		\end{align*}
		and the truncated sharp maximal function, denoted by $M^\sharp_\rho f,$ defined pointwise as
		\begin{align*}
			M^\sharp_\rho f(x)=\sup\limits_{\substack{x\in B(y,r), \\ r\le \rho}} \fint_{B(y,r)}\left|f(z)-(f)_{B(y,r)}\right|dz.
		\end{align*}
	\end{definition}
	\noindent The next two lemmas assert the weighted inequalities for maximal functions. The first goes back to Muckenhoupt \cite{Muckenhoupt_Apweights} (see also \cite{Stein_HarmonicAnalysis}) and is the main reason for  the introduction of $A_{p}$ weights. 
	\begin{lemma}\label{Ap maximal fn boundedness}
		Let $w\in A_p,$ for $1<p<\infty.$ Then there exists a constant $C=C\left(n,p,[w]_{A_p}\right)>0$ such that \begin{align}\label{Mf}
			\left \lVert Mf\right\rVert_{L^p_w(\mathbb R^n)}\le C\left\lVert f\right\rVert_{L^p_w(\mathbb R^n)}
		\end{align}
		for all $f\in L^p_w(\mathbb R^n).$ Conversely, if \eqref{Mf} holds for all $f\in L^p_w(\mathbb R^n),$ then $w\in A_p.$
	\end{lemma}
	\noindent Our use of $M^\sharp_\rho$ lies in the following key estimates, established in \cite[Corollary 2.7]{Phuc_Nguyen_WeightedFeffermanStein}, which is a localized, weighted version of the well-known Fefferman-Stein inequality. 
	\begin{lemma}\label{Feff-Stein}
		Let $M_0>0,$ $1<p<\infty$ and $w\in A_p$ such that $[w]_{A_p}\le M_0.$ Then there exists  constants $k_0=k_0\left(n,p,M_0\right)>\sqrt{n}$ and $C_{1}=C_{1}\left(n,p,M_0\right)>0$ such that for $f\in L^p_w(\mathbb R^n)$ with $\operatorname{supp}(f)\subset B_\rho(x_0)$ for some $\rho>0,$ we have 
		\begin{align*}
			\int_{B_\rho(x_0)}|f(x)|^pw(x)dx&\le C_{1}\int_{B_{k_0\rho}(x_0)}\left(M^\sharp_{k_0\rho}f(x)\right)^pw(x)dx.
		\end{align*}
	\end{lemma}	
	
	\noindent	To use the above lemma to prove our boundary estimate, we need to extend our function from upper half plane to the whole of $\mathbb R^n.$
	\begin{definition}\label{extension by reflection}
		Suppose $\psi:\mathbb R^n_+\to\mathbb R$ be a locally integrable function. We define its extension $\Psi:\mathbb R^n\to \mathbb R$ as 
		\begin{align*}
			\Psi(x_1,x_2,\dots,x_{n-1},x_n)=
			\left\lbrace\begin{aligned}
				&\psi(x_1,x_2,\dots,x_{n-1},x_n) &&\text{ if } x_n\ge0,\\
				&\psi(x_1,x_2,\dots,x_{n-1},-x_n) &&\text{ if } x_n<0.
			\end{aligned}\right.
		\end{align*}
	\end{definition}
	\noindent	We see that $\Psi$ is symmetric with respect to $\partial\mathbb R^n_+.$ 
	\begin{lemma}\label{extension estimate lemma}
		For any $x\in\mathbb R^n_+$ and for all $R>0,$ we have 
		\begin{align*}
			\fint_{B(x,R)}|\Psi-\left( \Psi \right)_{B(x,R)}|
			&\le c\fint_{B(x,R)\cap\mathbb R^n_+}|\psi-(\psi)_{B(x,R)\cap\mathbb R^n_+}|,
		\end{align*}
		for some $c>0$ depending only on $N, k, n$. 
	\end{lemma}
	\begin{proof}
		First observe that there is nothing to prove if $B(x,R)\subset\mathbb R^n_+.$ So assume that $B(x,R)\cap\overline{\mathbb R^n_{-}}\neq \emptyset.$ Let $\lambda$ be a real number and $T:\mathbb R^n_-\to\mathbb R^n_+$ defined as \begin{align*}
			T(x_1,x_2,\dots,x_{n-1},x_n)
			&=(x_1,x_2,\dots,x_{n-1},-x_n).
		\end{align*}
		Then using the fact that $T \in \mathbb{O}\left(n\right)$ and the change of variables \begin{align*}
			\int_{B(x,R)}|\Psi-\lambda|
			&=\int_{B(x,R)\cap\mathbb R^n_+}|\psi-\lambda|+\int_{B(x,R)\cap\mathbb R^n_-}|\Psi-\lambda|\\
			&=\int_{B(x,R)\cap\mathbb R^n_+}|\psi-\lambda|+\int_{T\left(B(x,R)\cap\mathbb R^n_-\right)}|\psi-\lambda|\left\lvert \operatorname{det} DT^{-1} \right\rvert\\
			&\le 2\int_{B(x,R)\cap\mathbb R^n_+}|\psi-\lambda|.
		\end{align*}
		Therefore, by minimality of mean, we obtain  \begin{align*}
			\int_{B(x,R)}|\Psi-\left( \Psi \right)_{B(x,R)}|
			&\le c\int_{B(x,R)}|\Psi-\lambda| \le 2c\int_{B(x,R)\cap\mathbb R^n_+}|\psi-\lambda|.
		\end{align*}
		Hence by choosing $\lambda=(\psi)_{B(x,R)\cap\mathbb R^n_+},$ we complete the proof of the lemma.
	\end{proof}
	A similar straight-forward computation, which is skipped, establishes the following result. 
	\begin{lemma}\label{extension of weight lemma}
		Suppose $1<p<\infty$ and $w\in A_p.$ Let us define $\tilde{w}$ as 
		\begin{align*}
			\tilde{w}(x_1,x_2,\dots,x_{n-1},x_n)=
			\left\lbrace\begin{aligned}
				&w(x_1,x_2,\dots,x_{n-1},x_n) &&\text{ if } x_n\ge0,\\
				&w(x_1,x_2,\dots,x_{n-1},-x_n) &&\text{ if } x_n<0.
			\end{aligned}\right.
		\end{align*}
		Then $\tilde{w}\in A_p$ with the estimate
		\begin{align*}
			[\tilde{w}]_{A_p}&\le 2^p[w]_{A_p}.
		\end{align*}
	\end{lemma}
	\subsection{Reverse H\"{o}lder inequalities}
	We start with the well-known  self-improving property of the reverse H\"{o}lder inequalities with increasing support. See \cite[Theorem 6.38]{giaquinta-martinazzi-regularity} for a proof.   
	\begin{theorem}\label{gehring}
		Suppose that $1<r<2,$ $f\in L^2_{loc}(\Omega)$ be a non-negative function such that for some $b>0,$ $R_0>0$
		\begin{align*}
			\left( \fint_{B_R(x)}f^2\right)^\frac{1}{2}
			&\le b\left(\fint_{B_{2R}(x)}f^r\right)^\frac{1}{r},
		\end{align*}
		for all $x\in \Omega,$ $0<R<\min\left(R_0,\frac{\operatorname{dist}(x,\partial\Omega)}{2}\right).$ Then $f\in L^s_{loc}(\Omega),$ for some $s>2$ and there exists a constant $C=C(n,r,s,b)>0$ such that 
		\begin{align*}
			\left(\fint_{B_R(x)}f^s\right)^\frac{1}{s}
			&\le C\left(\fint_{B_{2R}(x)}f^2\right)^\frac{1}{2}.
		\end{align*}
	\end{theorem}
	\noindent The self-improving property also goes the other way. See \cite[Remark 6.12]{Giusti_DCV}.  
	\begin{theorem}\label{s-2to2-1}
		Let $s>2.$ Suppose there exists $C>0$ such that 
		\begin{align*}
			\left(\fint_{B^+_\rho(x)}|f|^s\right)^\frac{1}{s}&\le \left(\fint_{B^+_{2\rho}(x)}|f|^2\right)^\frac{1}{2},
		\end{align*}
		for any ball $B_\rho(x)\subset B_{2\rho}(x)\subset B_R(x).$ Then there exists $C>0$ such that 
		\begin{align*}
			\left(\fint_{B^+_\frac{R}{4}(x)}|f|^2\right)^\frac{1}{2}&\le C \fint_{B^+_\frac{R}{2}(x)}|f|.
		\end{align*}
	\end{theorem}
	\subsection{Poincar\'{e}-Sobolev inequalities for the Hessian}
	\begin{proposition}\label{Poincare_zero_mean_Lorentz}
		Let $R>0$ and $1 < p < \infty$ be real numbers. Then, for any $u \in W^{1,p}\left(B_{R}^{+}\right)$ such that  either $ \fint_{B^{+}_{R}}u = 0$ or $u \equiv 0$ on $\Gamma_{R},$ there exists a constant $C=C(n,p)>0$ such that
		\begin{align*}
			\left\lVert u \right\rVert_{L^{p}\left(B_{R}^{+}\right)} \leq C R	\left\lVert \nabla u \right\rVert_{L^{p}\left(B_{R}^{+}; \mathbb{R}^{n}\right)}.
		\end{align*}
	\end{proposition}

	\begin{lemma}\label{Hessian Poincare Sobolev}
		Let $R>0$ and $1 <p < \infty$ be real numbers. Then for any $u \in W^{2,p}\left( \mathcal{U}_{R}; \varLambda^{k}\right)$ satisfying
		\begin{align*}
			\text{ either }\quad 	e_{n}\wedge u &=0  \qquad\text{ or } \qquad e_{n}\lrcorner u = 0 &&\text{ on } \Gamma_{3R/4},
		\end{align*}
		and for any $0 < \rho \leq 3R/4,$ there exists $\bar{u}^{\rho} \in W^{2,p}\left( \mathcal{U}_{R}; \varLambda^{k}\right)$ such that
		\begin{align*}
			D^{2}u &=D^{2}\bar{u}^{\rho} &&\text{ in } \mathcal{U}_{R},
		\end{align*} $\bar{u}^{\rho}$ satisfies the same boundary condition as $u$ on $\Gamma_{3R/4}$ and there exists a constant $C = C \left( n, k, N, p \right) >0$ such that
		\begin{align*}
			\frac{1}{\rho^{2}}\left\lVert \bar{u}^{\rho} \right\rVert_{L^{p}\left( B^{+}_{\rho}\right)} + \frac{1}{\rho}\left\lVert \nabla \bar{u}^{\rho} \right\rVert_{L^{p}\left( B^{+}_{\rho}\right)} &\leq C\left\lVert D^{2}\bar{u}^{\rho} \right\rVert_{L^{p}\left( B^{+}_{\rho}\right)}.
		\end{align*}
	\end{lemma}
	\begin{proof}
		The construction is detailed in \cite[Lemma 3.8]{SenguptaSil_MorreyLorentz_Hodge}. We first prove the case $e_{n}\wedge u = 0$ on $\Gamma_{3R/4}.$ By a simple scaling argument, we can assume $R=1$ and $ 0 < \rho \leq 3/4.$ Let us define
		\begin{align*}
			\bar{u}^{\rho}_{I, j}\left( x\right) := \left\lbrace \begin{aligned}
				&u_{I, j}\left( x\right) - \left( \fint_{B^{+}_{\rho}} \frac{\partial 	u_{I, j}}{\partial x_{n}}\right)x_{n} &&\text{ if } n \notin I, \\
				&\begin{multlined}[t]
					u_{I, j}\left( x\right) - \left( \fint_{B^{+}_{\rho}}u_{I, j}\right) - \left\langle  x, \left( \fint_{B^{+}_{\rho}} \nabla u_{I, j}\right)\right\rangle \\+  \left\langle  \left( \fint_{B^{+}_{\rho}} x \right) , \left( \fint_{B^{+}_{\rho}} \nabla u_{I, j}\right)\right\rangle
				\end{multlined} &&\text{ if } n \in I,
			\end{aligned}\right.
		\end{align*}
		for all $1 \leq j \leq N,$ where $\fint_{B^{+}_{\rho}} x$ denotes the constant vector in $\mathbb{R}^{n}$ formed by the components
		$$ \left( \fint_{B^{+}_{\rho}} x\right)_{i} := \fint_{B^{+}_{\rho}} x_{i}\ \mathrm{d}x \qquad \text{ for } 1 \leq i \leq n.$$ From now on, every statement below is assumed to hold for every $1 \leq j \leq N.$ Now note that, since $	e_{n}\wedge u =0$  on $ \Gamma_{\rho},$ we have $u_{I, j} \equiv 0$ on $ \Gamma_{\rho}$ if $ n \notin I$, and, consequently, we also have
		\begin{align*}
			\frac{\partial u_{I, j}}{\partial x_{l}} &\equiv 0 &&\text{ on } \Gamma_{\rho} \quad \text{ if } n \notin I, 1 \leq l \leq n-1.
		\end{align*}
		Now it is easy to check that this implies, by our construction, that every component of $\bar{u}_{\rho}$ and all its first-order derivatives either vanish on $\Gamma_{\rho}$ or have zero integral averages on $B_{\rho}^{+}.$ The desired estimate easily follows from this by using Proposition \ref{Poincare_zero_mean_Lorentz}. Since we also have $D^{2}u =D^{2}\bar{u}^{\rho} $ in $\mathcal{U},$ this completes the proof. For the case $e_{n}\lrcorner u = 0$ on $\Gamma_{3R/4},$ we interchange the cases $n \in I$ and $n \notin I$ in the definition of $\bar{u}^{\rho}_{I, j}$ and argue similarly.
	\end{proof}
	
	\subsection{Spectrum of the Hodge systems}
	Let $\Omega \subset \mathbb{R}^{n}$ be an open, bounded, $C^{2,1}$ subset. Let $A \in C^{0,1}\left(\Omega; \operatorname{Hom}\left(\varLambda^{k+1}\right)\right)$ and $B \in C^{1,1}\left(\Omega; \operatorname{Hom}\left(\varLambda^{k}\right)\right)$. Consider the linear Hodge operator 
	\begin{align*}
		\mathfrak{L}u := d^{\ast}\left(A\left(x\right)du\right) + \left[B(x)\right]^{\intercal}dd^{\ast}\left(B(x)u\right)
	\end{align*}
	We set the tangential boundary operator as  
	\begin{align*}
		\mathfrak{b}_{\text{t}}u := \left( \nu \wedge u, \nu \wedge d^{\ast}\left(B(x)u\right)\right)
	\end{align*}
	and the normal boundary operator as 
	\begin{align*}
		\mathfrak{b}_{\text{n}}u := \left( \nu \lrcorner \left(B(x)u\right), \nu \lrcorner \left(A\left(x\right)du\right)\right)
	\end{align*}
	We now define the notion of weak solutions for these systems. 
	\begin{definition}
		$u \in W^{1,2}\left(\Omega; \varLambda^{k}\right)$ is called a weak solution of  
		\begin{align*}
			\left\lbrace \begin{aligned}
				\mathfrak{L}u &= \lambda B(x)u + f  &&\text{ in } \Omega, \\
				\mathfrak{b}_{\text{t}}u &= \mathfrak{b}_{\text{t}}u_{0} &&\text{ on } \partial \Omega,
			\end{aligned}\right. 
		\end{align*}
		if $u -u_{0} \in W_{T}^{1,2}\left(\Omega; \varLambda^{k}\right)$ and we have 
		\begin{align}\label{weak_formulation_tangential}
			\int_{\Omega} \left\langle A\left(x\right)du, d\phi\right\rangle &+ 	\int_{\Omega} \left\langle d^{\ast}\left( B\left(x\right)u\right), d^{\ast}\left( B\left(x\right)\phi\right)\right\rangle - 	\lambda\int_{\Omega} \left\langle B\left(x\right)u, \phi \right\rangle  \notag \\&= 	\int_{\Omega} \left\langle f, \phi\right\rangle \qquad \text{ for all } \phi \in W^{1,2}_{T}\left(\Omega; \varLambda^{k}\right).  
		\end{align}
		Similarly, we say $u \in W^{1,2}\left(\Omega; \varLambda^{k}\right)$ is a weak solution of  
		\begin{align*}
			\left\lbrace \begin{aligned}
				\mathfrak{L}u &= \lambda B(x)u + f  &&\text{ in } \Omega, \\
				\mathfrak{b}_{\text{n}}u &= \mathfrak{b}_{\text{n}}u_{0} &&\text{ on } \partial \Omega,
			\end{aligned}\right. 
		\end{align*}
		if $v = B(x)u$ satisfies $v - B(x)u_{0} \in W_{N}^{1,2}\left(\Omega; \varLambda^{k}\right)$ and we have 
		\begin{align}\label{weak_formulation_normal}
			\int_{\Omega}  &\langle A (x)d \left( B^{-1}(x) v \right) , d \left( B^{-1}(x) \phi \right) \rangle   + \int_{\Omega}\langle d^{\ast} v  , d^{\ast} \phi \rangle 
			+ \lambda  \int_{\Omega}\langle  v , B^{-1}(x)\phi \rangle \notag \\
			&\qquad =\int_{\Omega} \langle  f, B^{-1}(x)\phi \rangle   \qquad \text{ for all }   \phi \in W_{N}^{1,2}(\Omega; \varLambda^{k}). 
		\end{align}
	\end{definition}
	\begin{definition}
		A real number $\lambda \in \mathbb{R}$ is called a tangential eigenvalue for $\mathfrak{L}$ if there exists a nontrivial weak solution $\alpha \in W^{1,2}\left(\Omega; \varLambda^{k}\right)$ to the following boundary value problem 
		\begin{align*}
			\left\lbrace \begin{aligned}
				\mathfrak{L}\alpha &= \lambda B(x)\alpha &&\text{ in } \Omega, \\
				\mathfrak{b}_{\text{t}}\alpha &= 0 &&\text{ on } \partial\Omega. 
			\end{aligned}\right. 
		\end{align*} 
		Similarly, $\lambda \in \mathbb{R}$ is called a normal eigenvalue for $\mathfrak{L}$ if there exists a nontrivial weak solution $\alpha \in W^{1,2}\left(\Omega; \varLambda^{k}\right)$ to the following boundary value problem 
		\begin{align*}
			\left\lbrace \begin{aligned}
				\mathfrak{L}\alpha &= \lambda B(x)\alpha &&\text{ in } \Omega, \\
				\mathfrak{b}_{\text{n}}\alpha &= 0 &&\text{ on } \partial\Omega. 
			\end{aligned}\right. 
		\end{align*} 
	\end{definition}    
	\begin{remark}
		By standard variational method, $L^{2}$ estimates and Fredholm theory, it is easy to see that if $A$  and  $B$ are both uniformly Legendre elliptic, then the set of tangential eigenvalues and the set of normal eigenvalues are at most countable subsets of $(-\infty, 0]$ and either has no limit points except $-\infty.$ Moreover, both the tangential and the normal eigenforms are as regular as the data allows and each tangential and normal eigenspace is finite dimensional. 
	\end{remark}
	Now we investigate when $\lambda \ge 0$ can be an eigenvalue. 
	\begin{proposition}
		Assume $A$, $B$ are uniformly Legendre-elliptic. A real number $\lambda >0$ can never be a tangential (respectivelly, normal) eigenvalue for $\mathfrak{L}.$ Furthermore,
		\begin{enumerate}[(a)]
			\item $0 \in \mathbb{R}$ is a tangential eigenvalue for $\mathfrak{L}$ if and only if $\mathcal{H}_{T}\left(\Omega; \varLambda^{k}\right) \neq \left\lbrace 0 \right\rbrace.$ Moreover, any $g \in L^{2}(\Omega,\varLambda^{k})$ with $d^{\ast}g=0$ in $\Omega$ is $L^{2}$-orthogonal to the eigenspace for $0$ if and only if $g \in \mathcal{H}^\perp_T(\Omega,\varLambda^{k})$. 
			\item $0 \in \mathbb{R}$ is a normal eigenvalue for $\mathfrak{L}$ if and only if $\mathcal{H}_{N}\left(\Omega; \varLambda^{k}\right) \neq \left\lbrace 0 \right\rbrace.$ Moreover, any $g \in L^{2}(\Omega,\varLambda^{k})$ with $d^{\ast}g=0$ in $\Omega$ is $L^{2}$-orthogonal to the eigenspace for $0$ if and only if $g \in \mathcal{H}^\perp_N(\Omega,\varLambda^{k})$ and $\nu \lrcorner g = 0$ on $\partial\Omega$.
		\end{enumerate} 
	\end{proposition}
	\begin{proof}
		If $\alpha \in W^{1,2}$ is an tangential eigenform for $\mathfrak{L}$, then 
		\begin{align*}
			\gamma_A\int_{\Omega}\left\lvert d\alpha\right\rvert^{2}+ \int_{\Omega}\left\lvert d^{\ast}\left(B\alpha\right) \right\rvert^{2} \leq -\int_{\Omega}\left\langle \mathfrak{L}\alpha, \alpha\right\rangle = \lambda\int_{\Omega}\left\lvert \alpha\right\rvert^{2} . 
		\end{align*}
		Thus, $\lambda >0$ is impossible and if $\lambda=0,$ then $d\alpha = 0$ and $d^{\ast}\left(B\alpha\right) = 0.$ Then by the tangential Hodge decomposition (see \cite[Theorem 6.9 (i)]{CsatoDacKneuss}), we must have $\alpha = d\theta + h,$ where $h \in \mathcal{H}_{T}\left(\Omega; \varLambda^{k}\right)$ and $\theta \in W^{2,2}\cap \mathcal{H}^\perp_T(\Omega,\varLambda^{k-1})$ is the unique solution of  
		\begin{align*}
			\begin{cases}
				\begin{aligned}
					d^{\ast}\left(Bd\theta\right)&= - d^{\ast}\left(Bh\right) &&\text{ in } \Omega, \\
					d^{\ast}\theta &= 0 &&\text{ in } \Omega, \\
					\nu \wedge \theta &= 0 &&\text{ on } \partial\Omega. 
				\end{aligned}
			\end{cases}
		\end{align*}
		This proves the first statement. For the second, we have 
		\begin{align*}
			\int_{\Omega}\left\langle g, \alpha\right\rangle =  \int_{\Omega}\left\langle g, d\theta + h \right\rangle= -\int_{\Omega}\left\langle d^{\ast}g, \theta \right\rangle + \int_{\Omega}\left\langle g, h \right\rangle = \int_{\Omega}\left\langle g, h \right\rangle. 
		\end{align*}
		This proves (a). Now if $\alpha \in W^{1,2}$ is a normal eigenform for $\mathfrak{L}$, then once again $\lambda>0$ is impossible and for $\lambda=0,$ we must have $d\alpha = 0$, $d^{\ast}\left(B\alpha\right)=0$ in $\Omega$ and $\nu\lrcorner B\alpha=0$ on $\partial\Omega.$  Using the standard Hodge decomposition theorem with normal boundary conditions (see \cite[Theorem 6.9 (ii)]{CsatoDacKneuss}), we can write 
		\begin{align*}
			B\alpha = d\psi + d^{\ast}\theta + h,
		\end{align*}
		where $\psi\in W^{1,2}_{N}\cap\mathcal{H}^\perp_N \left(\Omega; \varLambda^{k-1}\right)$, $, \theta \in W^{1,2}_{N}\cap\mathcal{H}^\perp_N\left(\Omega; \varLambda^{k+1}\right)$ and $h \in \mathcal{H}_{N}\left(\Omega; \varLambda^{k}\right),$ with $d^{\ast}\psi=0,$ $d\theta = 0$ in $\Omega$ and $\nu \lrcorner d\psi=0 $ on $\partial\Omega.$ Thus, $\psi$ is a solution to the equation 
		\begin{align*}
			\begin{cases}
				\begin{aligned}
					( d^{\ast}d+dd^{\ast})\psi&=0&&\qquad\text{ in }\Omega,\\
					\nu\lrcorner \psi&=0&&\qquad\text{ on }\partial\Omega,\\
					\nu\lrcorner d\psi&=0&&\qquad\text{ on }\partial\Omega.
				\end{aligned}
			\end{cases}
		\end{align*}
		By uniqueness, we have $\psi = 0$. Thus, we see that we can write $B\alpha = d^{\ast}\theta + h,$ where $h \in \mathcal{H}_{N}\left(\Omega; \varLambda^{k}\right)$ and $\theta \in W^{2,2}\cap \mathcal{H}^\perp_N(\Omega,\varLambda^{k-1})$ is the unique solution of  
		\begin{align*}
			\begin{cases}
				\begin{aligned}
					d\left(B^{-1}d^{\ast}\theta\right)&= - d\left(B^{-1}h\right) &&\text{ in } \Omega, \\
					d\theta &= 0 &&\text{ in } \Omega, \\
					\nu \lrcorner \theta &= 0 &&\text{ on } \partial\Omega. 
				\end{aligned}
			\end{cases}
		\end{align*}
		This proves the first statement. For the second, observe that since $d\alpha=0,$ we can write $\alpha = d\phi + h,$ where $h \in \mathcal{H}_{N}\left(\Omega; \varLambda^{k}\right)$ and $\phi \in W^{1,2}_{N}\cap\mathcal{H}^\perp_N \left(\Omega; \varLambda^{k-1}\right).$ Thus, we have 
		\begin{align*}
			\int_{\Omega} \left\langle g, \alpha\right\rangle = 	\int_{\Omega} \left\langle g, d\phi + h\right\rangle =	\int_{\partial\Omega} \left\langle \nu \lrcorner g, \phi \right\rangle -	\int_{\Omega} \left\langle d^{\ast}g, \phi\right\rangle +\int_{\Omega} \left\langle g, h\right\rangle.
		\end{align*}
		As $d^{\ast}g=0,$ this completes the proof.
	\end{proof}
	\section{Crucial estimates}\label{main estimates}
	\subsection{\texorpdfstring{$L^{p}$}{Lp} estimates for constant coefficients}
	We begin by recalling the global $L^{p}$-estimate for a constant-coefficient Legendre-Hadamard elliptic system with the usual Dirichlet boundary conditions. This is well-known. See \cite[Section 7.1]{giaquinta-martinazzi-regularity}.    
	\begin{proposition}\label{Global Lp estimates LH system prop}
		Let $U \subset \mathbb{R}^{n}$ be open, bounded and smooth, $m \in \mathbb{N}$ and let $1< p < \infty$ be a real number. Let $M \in \operatorname{Hom}\left(\mathbb{R}^{mn}\right)$ be Legendre-Hadamard elliptic with constant $\gamma_{M}>0.$ Let $u \in W^{1,2}\left(U; \mathbb{R}^{m}\right)$ be a weak solution of 
		\begin{align*}
			\begin{cases}
				\begin{aligned}
					-\operatorname{div} \left(M \nabla u\right) &=f &&\text{ in } U, \\
					u&=0 &&\text{ on } \partial U,
				\end{aligned}
			\end{cases}
		\end{align*}
		with $f \in L^{p}\left(U; \mathbb{R}^{m}\right).$ Then $u \in W^{2, p}\left(U; \mathbb{R}^{m}\right)$ and there exists a constant $C = C(n, m, \gamma_{M}, \left\lVert M \right\rVert_{L^{\infty}}, \Omega )>0$ such that we have the estimate 
		\begin{align*}
			\left\lVert \nabla^2 u \right\rVert_{L^{p}\left(\Omega; \mathbb{R}^{mn^2}\right)} \le C \left\lVert f \right\rVert_{L^{p}\left(\Omega; \mathbb{R}^{m}\right)}. 
		\end{align*}
	\end{proposition}
	Now we turn to some interior decay estimates. 
	\begin{theorem}[Interior Hessian decay estimates]\label{interior hessian decay}
		Let $U \subset \mathbb{R}^{n}$ be open and let $\bar{A} \in \operatorname{Hom}\left(\varLambda^{k+1}\right)$, $\bar{B} \in \operatorname{Hom}\left(\varLambda^{k}\right)$ be both Legendre elliptic with constant $\gamma_{\bar{A}}$ and $\gamma_{\bar{B}}$, respectively. Let $\alpha \in W^{1,2}(U ; \varLambda^{k})$ and assume one of the following holds.\smallskip 
		\begin{enumerate}[(i)] 
			\item For all $ \psi \in W_{0}^{1,2}(U ; \varLambda^{k}),$ we hve 
			\begin{align*}
				\int_{U} \langle \bar{A}d\alpha ; d\psi \rangle + \int_{U} \langle d^{\ast}\left( \bar{B} \alpha \right) ; d^{\ast} \left( \bar{B}\psi\right) \rangle   = 0.  
			\end{align*}  
			\item For all $\psi \in W_{0}^{1,2}(U ; \varLambda^{k}),$ we have 
			\begin{align*}
				\int_{U} \langle \bar{A}d\left( \bar{B}^{-1} \alpha \right) ; d\left( \bar{B}^{-1} \psi \right) \rangle + \int_{U} \langle d^{\ast}\alpha ; d^{\ast} \psi \rangle   = 0. 
			\end{align*}
		\end{enumerate}
		\noindent Then there exists a constant $C = C\left( \gamma_{\bar{A}}, \gamma_{\bar{B}}, \left\lVert \bar{A} \right\rVert_{L^{\infty}}, \left\lVert \bar{B} \right\rVert_{L^{\infty}}, k, n, N \right) >0$ such that for any two concentric balls $B_{\rho} \subset B_{r} \subset \subset U$ with $0 < \rho \le r, $ we have 
		\begin{align}\label{ball hessian mean osc decay}
			\left( \fint_{B_{\rho}} \left\lvert \nabla^2 \alpha - \left( \nabla^2 \alpha \right)_{B_{\rho}}\right\rvert^{2} \right)^{\frac{1}{2}}\le C \left(\frac{\rho}{r}\right)\left( \fint_{B_{r}} \left\lvert \nabla^2 \alpha  - \left( \nabla^2 \alpha \right)_{B_{r}}\right\rvert^{2}\right)^{\frac{1}{2}}. 
		\end{align} 
		Moreover, for any $1 < \bar{p} < 2$, there exists a constant $$C = C\left( \bar{p}, \gamma_{\bar{A}}, \gamma_{\bar{B}}, \left\lVert \bar{A} \right\rVert_{L^{\infty}}, \left\lVert \bar{B} \right\rVert_{L^{\infty}}, k, n, N \right) >0$$ such that for any ball $B_{r} \subset U,$ we have the estimate
		\begin{align}\label{ball Hessian decay}
			\left( \fint_{B_{r/2}} \left\lvert \nabla^2 \alpha \right\rvert^{2} \right)^{\frac{1}{2}}&\le C    \left( \fint_{B_{3r/4}} \left\lvert \nabla^2 \alpha \right\rvert^{\bar{p}} \right)^{\frac{1}{\bar{p}}},
		\end{align}
		where $B_{r/2} \subset B_{3r/4} \subset B_{r}$ are concentric balls. 
	\end{theorem}
	\begin{proof}
		By Lemma \ref{ellipticity lemma}, we infer that $\alpha \in W^{1,2}(U ; \varLambda^{k})$ is a weak solution to a Legendre-Hadamard elliptic system. Now the estimates follow from interior regularity estimates for such systems (see \cite[Chapter 4 and 5]{giaquinta-martinazzi-regularity}). We only sketch the proof. If $u$ is a solution to a Legendre-Hadamard elliptic system, we have the well-known decay estimate 
		\begin{align*}
			\int_{B_{\rho}} \left\lvert \nabla u \right\rvert^{2} \le C \left(\frac{\rho}{r}\right)^{n} \int_{B_{r}} \left\lvert \nabla u \right\rvert^{2}. 
		\end{align*}
		By a simple scaling argument, we can assume $r=1.$ Now combining this with Poincar\'e inequality and Caccioppoli inequality, for any $0 < \rho < 1/2,$ we have 
		\begin{align*}
			\int_{B_{\rho}} \left\lvert u - \left( u \right)_{B_{\rho}}\right\rvert^{2} &\le c\rho^{2}\int_{B_{\rho}} \left\lvert \nabla u \right\rvert^{2} \\&\le c\rho^{n+2}\int_{B_{1/2}} \left\lvert \nabla u \right\rvert^{2} \le c\rho^{n+2}\int_{B_{1}} \left\lvert u - \left(u\right)_{B_{1}} \right\rvert^{2}. 
		\end{align*}
		Scaling back, this implies the estimate 
		\begin{align*}
			\left(\fint_{B_{\rho}} \left\lvert u - \left( u \right)_{B_{\rho}}\right\rvert^{2}\right)^{\frac{1}{2}} \le c \left(\frac{\rho}{r}\right)\left( \fint_{B_{r}} \left\lvert u - \left( u \right)_{B_{r}}\right\rvert^{2}\right)^{\frac{1}{2}}. 
		\end{align*}
		Our estimate \eqref{ball hessian mean osc decay} follows from this by observing that for every $1\le i, j \le n,$ $\partial^2 \alpha/\partial x_{i}\partial x_{j}$ is also a solution to the same Legendre-Hadamard elliptic system. Estimate \eqref{ball Hessian decay} follows similarly from the standard reverse H\"{o}lder inequality 
		\begin{align*}
			\left( \fint_{B_{r/2}} \left\lvert \nabla u \right\rvert^{2} \right)^{\frac{1}{2}}&\le C    \left( \fint_{B_{3r/4}} \left\lvert \nabla u \right\rvert^{\bar{p}} \right)^{\frac{1}{\bar{p}}},
		\end{align*}
		for solutions of Legendre-Hadamard elliptic systems. 
	\end{proof}
	
	Now we need a boundary version of the above estimates. Since the boundary conditions are not the usual Dirichlet condition, these can not be deduced from the usual theory and this is where our systems differs from the well-known theory. However, this follows from the analysis in \cite{Sil_linearregularity}.  
	\begin{theorem}[Boundary Hessian decay estimates]\label{boundary hessian decay}
		Let $R>0$ and let $\bar{A} \in \operatorname{Hom}\left(\varLambda^{k+1}\right)$, $\bar{B} \in \operatorname{Hom}\left(\varLambda^{k}\right)$ be both Legendre elliptic with constant $\gamma_{\bar{A}}$ and $\gamma_{\bar{B}}$, respectively. Assume one of the following holds.
		\begin{enumerate}[(i)] 
			\item Let $\alpha \in W^{1,2}(\mathcal{U}_{R} ; \varLambda^{k})$ satisfy $e_{n}\wedge \alpha = 0$ on $\Gamma_{3R/4}$ and 
			\begin{align}\label{bndry hessian decay estimate constant}
				\int_{\mathcal{U}_{R}} \langle \bar{A}d\alpha ; d\psi \rangle + \int_{\mathcal{U}_{R}} \langle d^{\ast}\left( \bar{B} \alpha \right) ; d^{\ast} \left( \bar{B}\psi\right) \rangle   = 0
			\end{align} for all $\psi \in W_{T, flat}^{1,2}\left( \mathcal{U}_{R} ; \varLambda^{k}\right)$. 
			\item Let $\alpha \in W^{1,2}(\mathcal{U}_{R} ; \varLambda^{k})$  satisfy $e_{n}\lrcorner \alpha = 0$ on $\Gamma_{3R/4}$ and  
			\begin{align}\label{bndry hessian decay estimate constant normal}
				\int_{\mathcal{U}_{R}} \langle \bar{A}d\left( \bar{B}^{-1} \alpha \right) ; d\left( \bar{B}^{-1} \psi \right) \rangle + \int_{\mathcal{U}_{R}} \langle d^{\ast}\alpha ; d^{\ast} \psi \rangle   = 0   
			\end{align} for all $\psi \in W_{N, flat}^{1,2}\left( \mathcal{U}_{R} ; \varLambda^{k}\right).$
		\end{enumerate}\noindent 
		Then there exists a constant $C = C\left( \gamma_{\bar{A}}, \gamma_{\bar{B}}, \left\lVert \bar{A} \right\rVert_{L^{\infty}}, \left\lVert \bar{B} \right\rVert_{L^{\infty}}, k, n, N \right) >0$ such that for any $0 < \rho \le r \le  R/2,$ we have 
		\begin{align}\label{half ball hessian mean osc decay}
			\left( \fint_{B^{+}_{\rho}} \left\lvert \nabla^2 \alpha - \left( \nabla^2 \alpha \right)_{B^{+}_{\rho}}\right\rvert^{2} \right)^{\frac{1}{2}}\le C \left(\frac{\rho}{r}\right)\left( \fint_{B^{+}_{r}} \left\lvert \nabla^2 \alpha \right\rvert^{2}\right)^{\frac{1}{2}}. 
		\end{align} 
		Moreover, for any $1 < \bar{p} < 2$ and any $0 < r \le R,$ there exists a constant $C = C\left( \bar{p}, \gamma_{\bar{A}}, \gamma_{\bar{B}}, \left\lVert \bar{A} \right\rVert_{L^{\infty}}, \left\lVert \bar{B} \right\rVert_{L^{\infty}}, k, n, N \right) >0$ such that 
		\begin{align}\label{Hessian decay}
			\left( \fint_{B^{+}_{r/2}} \left\lvert \nabla^2 \alpha \right\rvert^{2} \right)^{\frac{1}{2}}&\le C    \left( \fint_{B^{+}_{3r/4}} \left\lvert \nabla^2 \alpha \right\rvert^{\bar{p}} \right)^{\frac{1}{\bar{p}}}.
		\end{align}
	\end{theorem}
	\begin{proof}
		Though not explicitly stated, the estimate \ref{half ball hessian mean osc decay} is already contained  in Section 3.2.1, 3.2.2 and 3.2.3 in \cite{Sil_linearregularity} when $\bar{B}$ is the identity matrix and the modifications to handle the general case is detailed in Sections 5.1 and 5.2 in \cite{Sil_linearregularity}. We only sketch the proof here. For details, see \cite{Mahato_Thesis}. 
		
		By scaling, we can assume $r=1.$ By Poincar\'e inequality and $L^{2}$ estimates and Sobolev embedding, we have 
		\begin{align}\label{hessian avg estimate}
			\int_{B^{+}_{\rho}} \left\lvert \nabla^2 \alpha - \left( \nabla^2 \alpha \right)_{B^{+}_{\rho}}\right\rvert^{2} &\le c\rho^{2}  \int_{B^{+}_{\rho}} \left\lvert D^{3}\alpha \right\rvert^{2} \notag\\&\le c\rho^{n+2}\sup\limits_{B_{\rho}^{+}} \left\lvert D^{3}\alpha\right\rvert^2 \notag\\&\le c\rho^{n+2}\sup\limits_{B_{1/2^{m+1}}^{+}} \left\lvert D^{3}\alpha\right\rvert^2 \notag\\&\le c \rho^{n+2}\left\lVert \alpha\right\rVert_{W^{m,2}\left(B^{+}_{1/2^m}\right)}^2 \le c \rho^{n+2}\int_{B_{1}^{+}}\left\lvert \alpha\right\rvert^2, 
		\end{align}
		for $0 < \rho < 1/2^{m+1},$ where $m \ge 3$ is an integer large enough such that $W^{m, 2} \hookrightarrow L^{\infty}.$ But by Lemma \ref{Hessian Poincare Sobolev}, there exists $\bar{\alpha}$ such that $\nabla^{2}\bar{\alpha} = \nabla^{2}\alpha $ in $\mathcal{U}_{R/r}$, $\bar{\alpha}$ satisfies the same boundary conditions in $\Gamma_{3R/4r},$ satisfies the same equation in $\mathcal{U}_{R/r}$ and we have the estimate 
		\begin{align}\label{L2normby hessian}
			\int_{B_{1}^{+}}\left\lvert \bar{\alpha}\right\rvert^2 \le c \int_{B_{1}^{+}}\left\lvert \nabla^{2}\bar{\alpha}\right\rvert^2. 
		\end{align}  Since $\bar{\alpha}$ satisfies the  same equation and the same boundary conditions, the estimate \eqref{hessian avg estimate} also holds for $\bar{\alpha}.$ Combining this with \eqref{L2normby hessian}, we obtain 
		\begin{align*}
			\int_{B^{+}_{\rho}} \left\lvert \nabla^2 \alpha - \left( \nabla^2 \alpha \right)_{B^{+}_{\rho}}\right\rvert^{2} &= \int_{B^{+}_{\rho}} \left\lvert \nabla^2 \bar{\alpha} - \left( \nabla^2 \bar{\alpha} \right)_{B^{+}_{\rho}}\right\rvert^{2} \\&\le c \rho^{n+2}\int_{B_{1}^{+}}\left\lvert \bar{\alpha}\right\rvert^2 
			\\&\le c \rho^{n+2}\int_{B_{1}^{+}}\left\lvert \nabla^{2}\bar{\alpha}\right\rvert^2 = c \rho^{n+2}\int_{B_{1}^{+}}\left\lvert \nabla^{2}\alpha\right\rvert^2. 
		\end{align*}
		This implies the estimate \eqref{half ball hessian mean osc decay} by scaling back. Estimate \ref{Hessian decay} is proved in \cite[Theorem 4.3]{SenguptaSil_MorreyLorentz_Hodge} for $r=R$, but the same proof works for any $0 < r \le R.$  
	\end{proof}
As a consequence, we have the following important estimate. 
	\begin{corollary}\label{rev. holder}
		Assume the hypothesis of Theorem \ref{boundary hessian decay} hold. Then there exists a constant $C = C\left( \gamma_{\bar{A}}, \gamma_{\bar{B}}, k, n, N\right) >0$ such that for any $z \in \partial\mathbb{R}^{n}_{+}$ and any $0 < r \le R$ with $B^{+}_{3r/4} \subset B^{+}_{3R/4}$, we have 
		\begin{align}\label{hessian2-1}
			\left( \fint_{B^+_\frac{r}{2}(z)}|\nabla^2 v|^2\right)^\frac{1}{2}
			&\le C\fint_{B^+_\frac{3r}{4}(z)}|\nabla^2 v| . 
		\end{align}
	\end{corollary}
	\begin{proof}
		Fix $1 < \bar{p} < 2.$ Thanks to the translation invariance of the problem, the estimate \eqref{Hessian decay} holds for any center $z \in \partial\mathbb{R}^{n}_{+}$ and any $0 < r \le R$ with $B^{+}_{3r/4}(z) \subset B^{+}_{3R/4}.$ More precisely, there exists a constant $C = C\left( \bar{p}, \gamma_{\bar{A}}, \gamma_{\bar{B}}, k, n, N \right) >0$ such that 
		\begin{align*}
			\left( \fint_{B^{+}_{r/2}(z)} \left\lvert \nabla^2 \alpha \right\rvert^{2} \right)^{\frac{1}{2}}&\le C    \left( \fint_{B^{+}_{3r/4}(z)} \left\lvert \nabla^2 \alpha \right\rvert^{\bar{p}} \right)^{\frac{1}{\bar{p}}}.
		\end{align*}  Now estimate \eqref{hessian2-1} follows by applying Theorem \ref{gehring} and Theorem \ref{s-2to2-1}. 
	\end{proof}
	\subsection{Flattening the boundary}
	\begin{lemma}[Flattening lemma]\label{flattening lemma}
		Let $\Omega \subset \mathbb{R}^{n}$ be an open, bounded, $C^{2,1}$ subset. Let $A \in C^{0,1}\left( \overline{\Omega}; \operatorname{Hom}\left(\varLambda^{k+1}\right)\right)$ and $B \in C^{1,1}\left( \overline{\Omega}; \operatorname{Hom}\left(\varLambda^{k}\right)\right)$ both be uniformly Legendre elliptic. Let $f \in L^{2}\left( \Omega; \varLambda^{k} \right)$  and $\lambda \in \mathbb{R}$ be given. 
		Suppose one of the following holds. 
		\begin{enumerate}[(a)]
			\item\textbf{tangential case: } $\omega \in W^{1,2}_{T}\left(\Omega; \varLambda^{k}\right)$ satisfy 
			\begin{align}\label{maineqn}
				\int_{\Omega} \left\langle A\left(x\right)d\omega, d\phi\right\rangle &+ 	\int_{\Omega} \left\langle d^{\ast}\left( B\left(x\right)\omega\right), d^{\ast}\left( B\left(x\right)\phi\right)\right\rangle + 	\lambda\int_{\Omega} \left\langle B\left(x\right)\omega, \phi \right\rangle  \notag \\&= 	\int_{\Omega} \left\langle f, \phi\right\rangle \qquad \text{ for all } \phi \in W^{1,2}_{T}\left(\Omega; \varLambda^{k}\right).  
			\end{align}
			\item\textbf{normal case: } $\omega \in W^{1,2}_{N}\left(\Omega; \varLambda^{k}\right)$ satisfy
			\begin{align}\label{maineq normal}
				\int_{\Omega}  &\langle A (x)d \left( B^{-1}(x) \omega \right) , d \left( B^{-1}(x) \phi \right) \rangle   + \int_{\Omega}\langle d^{\ast} \omega  , d^{\ast} \phi \rangle 
				+ \lambda  \int_{\Omega}\langle  \omega , B^{-1}(x)\phi \rangle \notag \\
				&\qquad =\int_{\Omega} \langle  f, B^{-1}(x)\phi \rangle   \qquad \text{ for all }   \phi \in W_{N}^{1,2}(\Omega; \varLambda^{k}). 
			\end{align}
		\end{enumerate}
		
		\noindent Let $x_{0} \in \partial\Omega.$ Then there exists a neighborhood $U$ of $x_{0}$ in $\mathbb{R}^{n}$ and a positive number $0 < R_{0}< 1,$ such that there exists an admissible boundary coordinate system $\Phi \in \operatorname{Diff}^{2,1}\left( \overline{B_{R_{0}}}; \overline{U}\right)$ satisfying 
		\begin{align*}
			\Phi \left( 0\right)= x_{0},\   D\Phi \left( 0\right) \in \mathbb{SO}\left(n\right), \  \Phi \left( B_{R_{0}}^{+}\right) = \Omega \cap U, \  \Phi \left( \Gamma_{R_{0}}\right) = \partial\Omega \cap U,
		\end{align*}
		and for any $\theta \in C_{c}^{\infty}\left( U\right),$	there exist 
		tensor fields 
		\begin{align*}
			\mathrm{P} &\in L^{\infty}\left( B_{R_{0}}^{+}; \operatorname{Hom}\left(\varLambda^{k}\right)\right), \\	\mathrm{Q} &\in L^{\infty}\left( B_{R_{0}}^{+}; \operatorname{Hom}\left( \varLambda^{k};\varLambda^{k}\otimes \mathbb{R}^{n}\right)\right), \\\mathrm{R} &\in L^{\infty}\left( B_{R_{0}}^{+}; \operatorname{Hom}\left(\varLambda^{k}\otimes \mathbb{R}^{n}; \varLambda^{k}\right)\right) , \\
			\mathcal{A} &\in L^{2}\left( B_{R_{0}}^{+}; \varLambda^{k}\right), \\
			\mathcal{B} &\in W^{1,2}\left( B_{R_{0}}^{+}; \varLambda^{k}\otimes\mathbb{R}^{n}\right)\\
			\intertext{ and } 
			\mathrm{S} &\in C^{0,1} \left( \overline{B_{R_{0}}^{+}}; \operatorname{Hom}\left(\varLambda^{k}\otimes \mathbb{R}^{n}\right)\right),
		\end{align*}
		satisfying
		\begin{align}\label{S vanishes at the origin}
			S\left(0\right) =0
		\end{align} 
		and constant coefficient matrices $\bar{A} \in \operatorname{Hom}\left(\varLambda^{k+1}\right)$, $\bar{B} \in \operatorname{Hom}\left(\varLambda^{k}\right)$, both being Legendre elliptic with the same constants as $A$ and $B$, respectively,  such that 
		
		\begin{enumerate}[(A)]
			\item \label{flattenning tangential} If $(a)$ holds, then $u = \Phi^{\ast}\left( \theta \omega\right) \in W^{1,2}_{T, \text{flat}}\left( B_{R_{0}}^{+}, \varLambda^{k}\right)$ satisfies 
			\begin{align}\label{flattened equation}
				\int_{B_{R_{0}}^{+}} &\left\langle \bar{A}du, d\psi\right\rangle + 	\int_{B_{R_{0}}^{+}} \left\langle d^{\ast}\left( \bar{B} u\right), d^{\ast}\left( \bar{B}\psi\right)\right\rangle \notag\\
				&= 	\int_{B_{R_{0}}^{+}} \left\langle \tilde{f} + \mathrm{P}u + \mathrm{R}\nabla u + \mathcal{A}, \psi \right\rangle  + 	\int_{B_{R_{0}}^{+}} \left\langle  \mathrm{Q}u + \mathcal{B}, \nabla\psi\right\rangle \notag\\
				&\qquad +	\int_{B_{R_{0}}^{+}} \left\langle \mathrm{S}\nabla u, \nabla \psi\right\rangle, \qquad \text{ for all }  \psi \in W^{1,2}_{T, \text{ flat}}\left(B_{R_{0}}^{+}; \varLambda^{k}\right). 
			\end{align}
			
			Moreover, we have the following pointwise dependencies 
			\begin{align*}
				\left\lvert \mathrm{P} \right\rvert &\lesssim \left( \left|\nabla B\right|\left|D\Phi^{-1}\right|+|B|\left|D^2\Phi^{-1}\right|\right)^{2}
				\\
				\left\lvert \mathrm{Q}\right\rvert, 	\left\lvert \mathrm{R}\right\rvert 
				&\lesssim |B|^2\left|D\Phi^{-1}\right|\left|D^2\Phi^{-1}\right|+|B|\left|\nabla B\right|\left|D\Phi^{-1}\right|^2, \\
				\left\lvert \mathrm{S}\right\rvert &\lesssim  \left|D\Phi\right|^{2}|A|+|B|^2\left|D\Phi\right|^3, \\
				\left\lvert \nabla \mathrm{S}\right\rvert &\lesssim  \begin{aligned}[t]
					\left\lvert \nabla A \right\rvert \left\lvert D\Phi\right\rvert^{2} + \left\lvert A \right\rvert \left\lvert D\Phi \right\rvert\left\lvert D^{2}\Phi \right\rvert &+ \left\lvert B \right\rvert\left\lvert \nabla B \right\rvert\left\lvert D\Phi \right\rvert^{3} \\ &\qquad + \left\lvert B \right\rvert^{2}\left\lvert D\Phi \right\rvert^{2}\left\lvert D^{2}\Phi \right\rvert,
				\end{aligned} \\
				\left\lvert \tilde{f} \right\rvert &\lesssim \left\lvert D\Phi \right\rvert \cdot \left\lvert f \right\rvert \cdot \mathbbm{1}_{\operatorname{supp} \theta}, \\
				\left\lvert \mathcal{B} \right\rvert&\lesssim \left(\left\lvert B \right\rvert^{2}\left\lvert \nabla \theta \right\rvert \left\lvert D\Phi^{-1}\right\rvert + \left\lvert A \right\rvert\left\lvert \nabla \theta \right\rvert \left\lvert D\Phi \right\rvert^{2} 	 \right) \left\lvert \omega\right\rvert \cdot \mathbbm{1}_{\operatorname{supp} \theta} , \\
				\left\lvert \mathcal{A} \right\rvert&\lesssim \begin{aligned}[t]
					\left\lvert B\right\rvert & \left\lvert \nabla \theta \right\rvert \left[  \left\lvert B \right\rvert \left\lvert \nabla^{2} \Phi^{-1} \right\rvert +  \left\lvert \nabla B \right\rvert  \left\lvert D\Phi^{-1} \right\rvert    \right]  \left\lvert \omega \right\rvert +  \left\lvert A \right\rvert  \left\lvert \nabla \theta \right\rvert  \left\lvert D\Phi \right\rvert^{2}  \left\lvert \nabla \omega \right\rvert  \\&+  \left\lvert B \right\rvert  \left\lvert \nabla \theta \right\rvert  \left\lvert D\Phi \right\rvert^{2} \left[  \left\lvert \nabla B \right\rvert  \left\lvert \omega \right\rvert +  \left\lvert B \right\rvert \left\lvert \nabla \omega \right\rvert  \right] \\&\qquad\qquad \qquad +  \left\lvert B \right\rvert  \left\lvert D\Phi \right\rvert^{2} \left\lvert \omega \right\rvert   \cdot \mathbbm{1}_{\operatorname{supp} \theta}, 
				\end{aligned} 
			\end{align*}
			and consequently 
			\begin{align*}
				\left\lvert \nabla \mathcal{B} \right\rvert&\lesssim \begin{aligned}[t]
					&\left[ \left\lvert B\right\rvert\left\lvert \nabla B\right\rvert\left\lvert \nabla \theta\right\rvert\left\lvert D\Phi^{-1} \right\rvert + \left\lvert B\right\rvert^{2}\left\lvert \nabla^{2} \theta\right\rvert\left\lvert D\Phi^{-1}\right\rvert  \right]\left\lvert \omega\right\rvert  \\&\quad + \left[ \left\lvert B\right\rvert^{2}\left\lvert D^{2}\Phi^{-1}\right\rvert + \left\lvert A\right\rvert\left\lvert D\Phi\right\rvert\left\lvert D^{2}\Phi\right\rvert\right] \left\lvert \nabla \theta\right\rvert\left\lvert \omega\right\rvert \\&\qquad + \left[ \left\lvert \nabla A\right\rvert \left\lvert \nabla \theta\right\rvert\left\lvert D\Phi \right\rvert^{2} + \left\lvert A\right\rvert\left\lvert \nabla^{2}\theta\right\rvert\left\lvert D\Phi\right\rvert^{2} \right]\left\lvert \omega\right\rvert \\&\quad\qquad   + \left[\left\lvert B\right\rvert^{2}\left\lvert \nabla \theta\right\rvert\left\lvert D\Phi^{-1} \right\rvert + \left\lvert A\right\rvert\left\lvert \nabla \theta\right\rvert\left\lvert D\Phi\right\rvert^{2}\right]\left\lvert \nabla \omega\right\rvert
				\end{aligned} \\
				\left\lvert \nabla \mathrm{Q}\right\rvert &\lesssim \begin{multlined}[t]
					\left\lvert B\right\rvert\left\lvert \nabla B \right\rvert\left|D\Phi^{-1}\right|\left|D^2\Phi^{-1}\right| + |B|^2\left|D^2\Phi^{-1}\right|^{2} + \left\lvert \nabla B \right\rvert^{2}\left|D\Phi^{-1}\right|^{2} \\ + 	\left\lvert B\right\rvert^{2}\left|D\Phi^{-1}\right|\left|D^3\Phi^{-1}\right|  + \left\lvert B\right\rvert\left\lvert \nabla^2 B\right\rvert\left|D\Phi^{-1}\right|^{2},
				\end{multlined} 
			\end{align*}
			where the symbol $\lesssim$ implies that the bounds on the right hold up to constants that depend on $n, k, N.$ 
			
			\item \label{flattenning normal} If (b) holds, then $u = \Phi^{\ast}(\theta\omega) \in W_{N, flat}^{1,2}(B_{R_{0}}^{+} ; \varLambda^{k}) $  satisfies 
			\begin{align}\label{flattened equation normal}
				\int_{B_{R}^{+}} &\langle \bar{A}(d \left( \bar{B}^{-1}u \right)) ; d\left( \bar{B}^{-1}\psi \right)\rangle + \int_{B_{R}^{+}} \langle d^{\ast}  u ; d^{\ast} \psi   \rangle  \notag\\ 
				&=\int_{B_{R}^{+}} \langle \widetilde{f} + \mathrm{P}u + \mathrm{R}\nabla u + \mathcal{A}; \psi \rangle + \int_{B_{R}^{+}} \langle \mathrm{Q} u + \mathcal{B} ; \nabla\psi \rangle
				\notag \\ &\qquad \qquad    + \int_{B_{R}^{+}} \langle \mathrm{S} \nabla u , \nabla\psi \rangle \quad \text{ for all } \psi \in W_{N, flat}^{1,2}(B_{R}^{+} ; \varLambda^{k}). 
			\end{align} 
			for all $\psi \in W_{N, flat}^{1,2}(B_{R}^{+} ; \varLambda^{k}).$ 
		\end{enumerate}
		
		Moreover, we have the following pointwise dependencies 
		\begin{align*}
			| \mathrm{P}|
			&\lesssim |A||D\Phi|^2\left(|\nabla B^{-1}||D\Phi|^2+|B^{-1}||D^2\Phi|\right),\\
			|\mathrm{Q}|,|\mathrm{R}|
			&\lesssim |A||B^{-1}||D\Phi|^3\left(|\nabla B^{-1}||D\Phi|^2+|B^{-1}||D^2\Phi|\right),\\
			|\mathrm{S}|
			&\lesssim |A||B^{-1}||D\Phi|^4+|D \Phi^{-1}|^2, \\
			|\nabla\mathrm{S}|
			&\lesssim \begin{multlined}[t]
				\left[|\nabla A|\left|B^{-1}\right|+|A|\left|\nabla B^{-1}\right|\right]|D\Phi|^4+|A||B^{-1}||D\Phi|^3|D^2\Phi| \\ +\left|D\Phi^{-1}\right|\left|D^2\Phi^{-1}\right|,
			\end{multlined} \\
			\left\lvert \tilde{f} \right\rvert &\lesssim \left\lvert B^{-1} \right\rvert\left\lvert D\Phi \right\rvert \cdot \left\lvert f \right\rvert \cdot \mathbbm{1}_{\operatorname{supp} \theta}, \\
			|\mathcal{B}|
			&\lesssim |\nabla\theta||D\Phi^{-1}|^2|\omega| + |A||B^{-1}|^2|\nabla\theta||D\Phi|^2|\omega|,\\
			|\mathcal{A}|
			&\lesssim \begin{aligned}[t]
				&\left[|\nabla\theta|\left|D^2\Phi^{-1}\right||\omega|+|A||B^{-1}||\nabla\theta||D\Phi|^2\left(|\nabla B^{-1}||\omega|+|B^{-1}||\nabla\omega|\right)\right]\\
				&\quad+\left[|D\Phi|^2||\nabla\theta||\nabla\omega|+|B^{-1}||D\Phi|^2|\theta||\omega|\right]\\
				&\qquad +|A||B^{-1}|||\nabla\theta||D\Phi|\left(|\nabla B^{-1}||D\Phi|^2+|B^{-1}||D^2\Phi|\right)|\omega|,
			\end{aligned}
		\end{align*}
		and consequently 
		\begin{align*}
			|\nabla\mathrm{Q}|
			&\lesssim \begin{aligned}[t]
				&\left[|\nabla A||B^{-1}|+| A||\nabla B^{-1}|\right]|\nabla B^{-1}||D\Phi|^5+|A||B^{-1}||\nabla^2 B^{-1}||D\Phi|^5\\
				&\quad+|A||B^{-1}||\nabla  B^{-1}||D\Phi|^4|D^2\Phi| \\&\qquad +\left[|\nabla A||B^{-1}|^2+|A||B^{-1}||\nabla B^{-1}|\right]|D\Phi|^3|D^2\Phi|\\
				&\qquad\quad+|A||B^{-1}|^2\left[|D\Phi|^2|D^2\Phi|^2+|D\Phi|^3|D^3\Phi|\right], 
			\end{aligned} \\
			\left\lvert \nabla \mathcal{B}\right\rvert &\lesssim \begin{aligned}[t]
				&\left[ |\nabla\theta||D\Phi^{-1}|^2 + |A||B^{-1}|^2|\nabla\theta||D\Phi|^2 \right]|\nabla \omega| \\
				&\quad + \left[ |D\Phi^{-1}|^2 + |A||B^{-1}|^2|D\Phi|^2 \right] |\nabla^{2}\theta||\omega| \\
				&\qquad +\left[ |D\Phi^{-1}||D^{2}\Phi^{-1}| + |A||B^{-1}|^2|D\Phi||D^{2}\Phi|\right]|\nabla\theta||\omega| \\
				&\qquad \quad+\left[ |\nabla A||B^{-1}|^2 + |A||B^{-1}||\nabla B^{-1}|\right]|D\Phi|^2|\nabla\theta||\omega|.
			\end{aligned} 
		\end{align*}
		where the symbol $\lesssim$ implies that the bounds on the right holds up to constants that depend on $n, k, N.$
	\end{lemma}
	\begin{proof}
		The proof is a straight-forward but lengthy calculation using the usual Korn's freezing trick, localization and then using the change of variable and grouping the terms. Details can be found in \cite{Mahato_Thesis} (see also \cite[Lemma 4]{Sil_linearregularity}) 
	\end{proof}

	\subsection{Boundary estimates}\label{boundary estimate section} 
	\subsubsection{Basic setup}
	For this section, we assume $\Omega \subset \mathbb{R}^{n}$ is open, bounded and $C^{2,1}$ and $1< p < \infty$ is a real number. We assume $A \in  C^{0,1}\left(\overline{\Omega},\operatorname{Hom}\left(\varLambda^{k+1}\right)\right)$  and $B\in C^{1,1}\left(\overline{\Omega},\operatorname{Hom}\left(\varLambda^k\right)\right)$ be both uniformly Legendre elliptic and $\lambda \in \mathbb{R}$. Fix a point $x_0\in\partial\Omega$. \smallskip 
	
	Our basic setup for the rest of this section would be any one of the following, depending on the boundary condition. \smallskip

	\begin{itemize}
		\item\textbf{Tangential boundary condition:} 
		Let $\omega \in W_{T}^{1,2}\left(\Omega; \varLambda^{k}\right) \cap W^{2,p}\left(\Omega; \varLambda^{k}\right)$ satisfy \eqref{maineqn}.  By Lemma \ref{flattening lemma}, there exists a neighborhood $U$ of $x_0$ in $\mathbb R^n,$ an admissible boundary coordinate system  $\Phi\in \operatorname{Diff}^{2,1}\left(\overline{B}_{R_0},\overline{U}\right)$ for some $R_0>0$ with $\Phi(0)=x_0,$ $\Phi^{-1}(U\cap\Omega)\subset B^+_{R_0}$ and $\Phi^{-1}(U\cap\partial\Omega)\subset \Gamma^+_{R_0},$ such that for any $\theta\in C^\infty_c(U)$, $u= \Phi^*(\theta\omega)\in W^{1,2}_{T,flat}\left(B^+_{R_0},\varLambda^k\right) \cap W^{2,p}\left(B^+_{R_0},\varLambda^k\right)$ satisfies 
		\begin{align}\label{eqr0}
			\int_{B_{R_{0}}^{+}} &\left\langle \bar{A}du, d\psi\right\rangle + 	\int_{B_{R_{0}}^{+}} \left\langle d^{\ast}\left( \bar{B} u\right), d^{\ast}\left( \bar{B}\psi\right)\right\rangle \notag\\
			&= 	\int_{B_{R_{0}}^{+}} \left\langle \tilde{f} + \mathrm{P}u + \mathrm{R}\nabla u + \mathcal{A}, \psi \right\rangle  + 	\int_{B_{R_{0}}^{+}} \left\langle  \mathrm{Q}u + \mathcal{B}, \nabla\psi\right\rangle \notag\\
			&\qquad+	\int_{B_{R_{0}}^{+}} \left\langle \mathrm{S}\nabla u, \nabla \psi\right\rangle  \qquad \text{ for all } \psi \in W^{1,2}_{T,flat}\left(B^+_{R_0},\varLambda^k\right),
		\end{align}
		where $\mathcal{A}$, $\mathcal{B}$, $\mathrm{P},$ $\mathrm{Q}$, $\mathrm{R},$ $\mathrm{S}$, $\bar{A}$, $\bar{B}$  and $\tilde{f}$ are given by Lemma  \ref{flattening lemma}, Part (\ref{flattenning tangential}). \smallskip
		
		\item\textbf{Normal boundary condition:} Let $\omega \in W_{N}^{1,2}\left(\Omega; \varLambda^{k}\right) \cap W^{2,p}\left(\Omega; \varLambda^{k}\right)$ satisfy \eqref{maineq normal}.	By Lemma \ref{flattening lemma}, there exists a neighborhood $U$ of $x_0$ in $\mathbb R^n,$ an admissible boundary coordinate system  $\Phi\in \operatorname{Diff}^{2,1}\left(\overline{B}_{R_0},\overline{U}\right)$ for some $R_0>0$ with $\Phi(0)=x_0,$ $\Phi^{-1}(U\cap\Omega)\subset B^+_{R_0}$ and $\Phi^{-1}(U\cap\partial\Omega)\subset \Gamma^+_{R_0},$ such that for any $\theta\in C^\infty_c(U)$, $u = \Phi^{\ast}(\theta\omega) \in W_{N, flat}^{1,2}(B_{R_{0}}^{+} ; \varLambda^{k})\cap W^{2,p}\left(B^+_{R_0},\varLambda^k\right) $  satisfies 
		\begin{align}\label{eqr0 normal}
			\int_{B_{R}^{+}} &\langle \bar{A}(d \left( \bar{B}^{-1}u \right)) ; d\left( \bar{B}^{-1}\psi \right)\rangle + \int_{B_{R}^{+}} \langle d^{\ast}  u ; d^{\ast} \psi   \rangle  \notag\\ 
			&=\int_{B_{R}^{+}} \langle \widetilde{f} + \mathrm{P}u + \mathrm{R}\nabla u + \mathcal{A}; \psi \rangle + \int_{B_{R}^{+}} \langle \mathrm{Q} u + \mathcal{B} ; \nabla\psi \rangle
			\notag \\ &\qquad \qquad    + \int_{B_{R}^{+}} \langle \mathrm{S} \nabla u , \nabla\psi \rangle \quad \text{ for all } \psi \in W_{N, flat}^{1,2}(B_{R}^{+} ; \varLambda^{k}),  
		\end{align}
		where $\mathcal{A}$, $\mathcal{B}$, $\mathrm{P},$ $\mathrm{Q}$, $\mathrm{R},$ $\mathrm{S}$, $\bar{A}$, $\bar{B}$  and $\tilde{f}$ are given by Lemma  \ref{flattening lemma}, Part (\ref{flattenning normal}).  
	\end{itemize} 
	In either case, let  $L \geq 1$ be any real number such that 
	$$   \left\lVert A \right\rVert_{L^{\infty}}, \left\lVert \bar{A} \right\rVert_{L^{\infty}},  \left\lVert B \right\rVert_{L^{\infty}}, \left\lVert \bar{B} \right\rVert_{L^{\infty}}\le L.$$ Our plan now to derive local estimates for the Hessian of $u$ near the boundary. For this, we want to derive pointwise estimates for the truncated sharp maximal function of $\left\lvert \nabla^{2}\tilde{u} \right\rvert,$ where $\tilde{u}$ is a suitable localization of $u$. 
	\subsubsection{Pointwise maximal inequalities}
	Let $k_0=k_0(n,p,M_0)>\sqrt{n}$ be the constant given by Lemma \ref{Feff-Stein}. As mentioned, we now localize $u$ further and derive a crucial decay estimate for the mean oscillation of $\left\lvert \nabla^{2}\tilde{u} \right\rvert,$ where $\zeta$ is a suitable cutoff function and $\tilde{u}=\zeta u$.  
	\begin{lemma}\label{poinwise decay lemma}
		Let $ 0 < \bar{R}< R_{0}$ be any real number. Let $h\ge 48$ and $d>0$ be real numbers such that $B^+(0,4hk_0d)\subset B^+(0,\bar{R}).$ Let $\zeta\in C^\infty_c\left(B(0,d)\right)$ be any cut-off function such that $\zeta\equiv 1$ in $B\left(0,d/2\right),$ $0\le \zeta\le 1,$ $|\nabla\zeta|\le 16/d$ and $|\nabla^2\zeta|\le 16/d^2.$ \smallskip 
		
		Assume one of the following holds. 
		\begin{enumerate}[(a)]
			\item $u \in  W^{1,2}_{T,flat}\left(B^+_{R_0},\varLambda^k\right) \cap W^{2,p}\left(B^+_{R_0},\varLambda^k\right) $ satisfy \eqref{eqr0}.
			\item $u \in  W^{1,2}_{N,flat}\left(B^+_{R_0},\varLambda^k\right) \cap W^{2,p}\left(B^+_{R_0},\varLambda^k\right) $ satisfy \eqref{eqr0 normal}.
		\end{enumerate} Set $\tilde{u}=\zeta u$. Then for any $1 < \bar{p} <  p ,$  there exists a constant $$C_{1} = C_{1}\left( \bar{p}, \gamma_{\bar{A}}, \gamma_{\bar{B}}, L, n, k, N \right)>0$$ such that for any $x\in B^+\left(0,k_0d\right)$, any $0 < R \le hk_{0}d$ and any $0<\rho\le R/48,$ we have 
		\begin{align}\label{localized hessian mean osc decay}
			\fint_{B^+_\rho(x)}&\left\lvert
			|\nabla^2\tilde{u}|-(|\nabla^2\tilde{u}|)_{x,\rho}\right\rvert \notag \\
			&\le C_{1}\left(\frac{\rho}{R}\right)\left(\fint_{B^+_R(x)}|\nabla^2 \tilde{u}|^{\bar{p}}\right)^\frac{1}{\bar{p}}+C_{1}\left(\frac{R}{\rho}\right)^\frac{n}{\bar{p}}\left( \fint_{B^+_{R}(x)}|\mathrm{S}\nabla^2\tilde{u}|^{\bar{p}}\right)^\frac{1}{\bar{p}} \notag \\
			&\qquad\quad+C_{1}\left(\frac{R}{\rho}\right)^\frac{n}{\bar{p}}\left[\left( \fint_{B^+_{R}(x)}|\mathcal{P}_1|^{\bar{p}}\right)^\frac{1}{p}+\left( \fint_{B^+_{R}(x)}|\mathcal{P}_2|^{\bar{p}}\right)^\frac{1}{\bar{p}}\right],
		\end{align}
		with \begin{align*}
			\mathcal{P}_1&=\mathcal{P}-\nabla\zeta\otimes\mathcal{Q}, \\
			\mathcal{P}_2&=\operatorname{div}\left(\zeta\mathcal{Q}\right)-\mathrm{S}\nabla^2\tilde{u}, \\
			\mathcal{Q}&=\mathcal{B}+\mathrm{Q}u+\mathrm{S}\nabla u, \\
			\mathcal{P} &= \mathcal{P}_{0} +\zeta\mathcal{A}+\zeta \mathrm{P}u+\zeta \mathrm{R}\nabla u +\zeta\tilde{f} 
		\end{align*}
		where 
		\begin{align*}
			\mathcal{P}_{0} = 
			\left\lbrace \begin{aligned}
				&\begin{multlined}[b]
					\bar{B}^{\intercal}\left[-d\left(d\zeta\lrcorner\bar{B}u\right)+d\zeta\wedge d^*\left(\bar{B}u\right)\right] +d\zeta\lrcorner\bar{A}du \\-d^*\left[\bar{A}\left(d\zeta\wedge u\right)\right], 
				\end{multlined} &&\text{ if }(a) \text{ holds,} \\
				& &&\\
				&\begin{multlined}[b]
					\left[\bar{B}^{-1}\right]^{\intercal}\left[ d^{\ast}\left(\bar{A}\left(d\zeta \wedge \bar{B}^{-1}u\right)\right) - d\zeta \lrcorner \bar{A}\left(d\left(\bar{B}^{-1}u\right)\right)\right] \\
					-d\zeta \wedge d^{\ast}u + d\left(d\zeta \lrcorner u \right),
				\end{multlined} &&\text{ if } (b) \text{ holds.}
			\end{aligned}\right. 
		\end{align*}
	\end{lemma}
	\begin{proof}
		We prove only the tangential boundary condition case, the other case is similar. For any $ 0 < \bar{R}< R_{0}$ and any $\psi\in W^{1,2}_{T,flat}\left(B^+_{\bar{R}}, \varLambda^k\right)$, let $\tilde{\psi}$ denote its extension by zero to $B^+_{R_0}.$ Since $\tilde{\psi} \in W^{1,2}_{T,flat}(B^+_{R_0},\varLambda^k)$, plugging $\tilde{\psi}$  as the test function in the weak formulation \eqref{eqr0}, we see by a direct computation that $\tilde{u}\in W^{1,2}_{T,flat}\left(B^+_{\bar{R}},\varLambda^k\right)\cap W^{2,p}\left(B^+_{\bar{R}},\varLambda^k\right)$ weakly solves the following equation
		\begin{align}\label{tilde u}
			d^*\left(\bar{A}d\tilde{u}\right)+\bar{B}^{\intercal}dd^*\left(\bar{B}\tilde{u}\right)&=\mathcal{P}+\zeta\operatorname{div}(\mathcal{Q})\qquad\text{ in } B^+_{\bar{R}}.  
		\end{align}
		\noindent Let $r=\operatorname{dist(x,\partial\mathbb R^n_+)}$ and $y$ be the projection onto the $\partial\mathbb R^n_+.$  If $r>R/5,$ then  
		\begin{align*}
			B(x,\rho)\subset B\left(x, R/10\right)\subset B\left(x, R/5\right)\subset B(x,R)\cap\mathbb R^n_+.
		\end{align*}
		Thus, the estimate in this case easily follows from the interior estimates. We skip the details and focus on the `near the boundary' case, i.e., the case $r \leq R/5.$ 
		In this case, we have  
		\begin{align*}
			B_R(x)\cap\mathbb R^n_+\subset \mathcal{U}_R(y)\subset B_{3R}(x)\cap \mathbb R^n_+.
		\end{align*} 
		For any $\psi\in W^{1,2}_{T,flat}\left(\mathcal{U}_R(y), \varLambda^k\right)$, its extension by zero is in $W^{1,2}_{T,flat}(B^+_{\bar{R}},\varLambda^k)$ and thus from \eqref{tilde u}, we find that  $\tilde{u}$ weakly solves the equation
		\begin{equation}
			d^*\left(\bar{A}d\tilde{u}\right)+\bar{B}^{\intercal}dd^*\left(\bar{B}\tilde{u}\right)=\mathcal{P}+\zeta\operatorname{div}(\mathcal{Q})\qquad\text{ in } \mathcal{U}_R(y).
		\end{equation}
		Let $w\in W^{1,2}_0(\mathcal{U}_R(y), \varLambda^{k})$ be the unique weak solution of the equation 
		\begin{align}\label{eq for w}
			\left\lbrace\begin{aligned}
				d^*\left(\bar{A}dw\right)+\bar{B}^{\intercal}dd^*\left(\bar{B}w\right)&=\mathcal{P}_1+\operatorname{div}\left(\zeta\mathcal{Q}\right)&&\qquad\text{ in } \mathcal{U}_R(y),\\
				w&=0&&\qquad\text{ on }\partial \mathcal{U}_R(y).
			\end{aligned}
			\right.
		\end{align}
		By Lemma \ref{ellipticity lemma}, this linear system is Legendre-Hadamard elliptic. Using Proposition \ref{Global Lp estimates LH system prop}, we deduce that $w \in W^{2, p}\left(\mathcal{U}_R(y), \varLambda^{k} \right).$ Now set $\tilde{u}=v+w.$ Then $v\in W^{1,2}(\mathcal{U}_R(y), \varLambda^{k})\cap  W^{2, p}\left(\mathcal{U}_R(y), \varLambda^{k} \right)$ weakly solves the following equation 
		\begin{align*}
			\left\lbrace\begin{aligned}
				d^*\left(\bar{A}dv\right)+\bar{B}^{\intercal}dd^*\left(\bar{B}v\right)&=0&&\text{ in } \mathcal{U}_R(y),\\
				v&=\tilde{u}&&\text{ on } \partial \mathcal{U}_R(y).\\
			\end{aligned}\right.
		\end{align*}
		Now we claim 
		\begin{claim}\label{claim for sharp fn of hess v} We have the estimate
			\begin{multline}\label{claim estimate}
				\fint_{B^+_\rho(x)}\left\lvert\nabla^2 v-(\nabla^2 v)_{B^+_\rho(x)}\right\rvert \\
				\le C\left(\frac{\rho}{R}\right)\left(\fint_{B^+_{3R}(x)}|\nabla^2 \tilde{u}|^{\bar{p}}\right)^\frac{1}{\bar{p}}+C \left(\fint_{B^+_{5R/8}(y)}|\nabla^2 w|^{\bar{p}}\right)^\frac{1}{\bar{p}}.
			\end{multline}
		\end{claim}
		We divide the proof of the claim into two cases.\smallskip 
		
		\noindent\textbf{Case I: } When $\rho<r/8,$ we have 
		\begin{align*}
			B(x,\rho)\subset B\left(x, r/4\right)\subset B^+\left(y, 5r/4\right)
		\end{align*}
		and 
		\begin{align*}
			B^+\left(y, 5r/4\right)\subset B^+\left(y, R/2\right)\subset B^+\left(y, 3R/4\right)\subset B(x,R)\cap\mathbb R^n_+.
		\end{align*}
		First we apply the interior decay estimate for $\nabla^2v$. By Theorem \ref{interior hessian decay} and Lemma \ref{minimality of mean}, we deduce 
		\begin{align*}
			\fint_{B_\rho(x)}&\left|\nabla^2 v-(\nabla^2 v)_{B_\rho(x)}\right| \\
			&\le \left(\fint_{B_\rho(x)}\left|\nabla^2 v-(\nabla^2 v)_{B_\rho(x)}\right|^2\right)^\frac{1}{2}\\
			&\stackrel{\eqref{ball hessian mean osc decay}}{\le} C\left(\frac{\rho}{r}\right)\left(\fint_{B_{r/4}(x)}\left\lvert\nabla^2 v - \left(\nabla^2 v\right)_{B_{r/4}(x)}\right\rvert^2 \right)^\frac{1}{2}\\
			&\le C\left(\frac{\rho}{r}\right)\left(\frac{\left|B^+_{5r/4}(y)\right|}{\left|B_{r/4}(x)\right|}\right)^\frac{1}{2}\left(\fint_{B^+_{5r/4}(y)}\left|\nabla^2 v - \left(\nabla^2 v\right)_{B^{+}_{5r/4}(y)}\right|^2\right)^\frac{1}{2} 
			\\&\stackrel{\eqref{half ball hessian mean osc decay}}{\le}C\left(\frac{\rho}{R}\right)\left(\fint_{B^+_{R/2}(y)}|\nabla^2 v|^2\right)^\frac{1}{2}.
		\end{align*}
		Applying the reverse H\"older inequality  and Jensen's inequality, we get
		\begin{align*}
			\left(\fint_{B^+_{R/2}(y)}|\nabla^2 v|^2\right)^\frac{1}{2}
			&\stackrel{\eqref{hessian2-1}}{\le} C \fint_{B^+_{5R/8}(y)}|\nabla^2 v|\\
			&\le C \left(\fint_{B^+_{5R/8}(y)}|\nabla^2 \tilde{u}|^{\bar{p}}\right)^\frac{1}{\bar{p}}+C\left(\fint_{B^+_{5R/8}(y)}|\nabla^2 w|^{\bar{p}}\right)^\frac{1}{\bar{p}}\\
			&\le C \left(\fint_{B^+_{3R}(x)}|\nabla^2 \tilde{u}|^{\bar{p}}\right)^\frac{1}{\bar{p}}+C\left(\fint_{B^+_{5R/8}(y)}|\nabla^2 w|^{\bar{p}}\right)^\frac{1}{\bar{p}}.\\
		\end{align*}
		Combining the last two estimates, we get \eqref{claim estimate}.\smallskip   
		
		\noindent\textbf{Case II: }	When $\rho\ge \frac{r}{8},$ clearly we have  
		\begin{align*}
			B(x,\rho)\cap\mathbb R^n_+\subset B^+(y,9\rho)\subset B^+\left(y, R/2\right)
		\end{align*}
		and \begin{align*}
			B^+\left(y, R/2\right)\subset B^+\left(y, 3R/4\right)\subset B(x,R)\cap\mathbb R^n.
		\end{align*}
		
		\noindent Then by minimality of mean type inequality and the mean oscillation decay estimates for $\nabla^2v$ in half ball, we have 
		\begin{align*}
			\int_{B^+(x,\rho)}\left|\nabla^2 v-(\nabla^2v)_{B^+(x,\rho)}\right|
			&\le c\int_{B_{9\rho}^+(y)}\left|\nabla^2 v-(\nabla^2 v)_{B_{9\rho}^+(y)}\right| \\
			&\le  \left(\fint_{B_{9\rho}^+(y)}\left|\nabla^2 v-(\nabla^2 v)_{B_{9\rho}^+(y)}\right|^2\right)^\frac{1}{2} \\
			&\stackrel{\eqref{half ball hessian mean osc decay}}{\le} C\left(\frac{\rho}{R}\right)\left(\fint_{B_{R/2}^+\left(y\right)}|\nabla^2v|^2\right)^\frac{1}{2}.
		\end{align*}
		From this, \eqref{claim estimate} follows by the same arguments as in Case I. This establishes Claim \ref{claim for sharp fn of hess v}. 
		We now complete the proof. Applying minimality of mean and Jensen's inequality, we have 
		\begin{align*}
			\fint_{B^+_\rho(x)}&\left\lvert
			|\nabla^2\tilde{u}|-(|\nabla^2\tilde{u}|)_{B^+_\rho(x)}\right\rvert \\
			&\le c\fint_{B^+_\rho(x)}\left\lvert
			|\nabla^2\tilde{u}|-(|\nabla^2 v|)_{B^+_\rho(x)}\right\rvert\\
			&\le  c\fint_{B^+_\rho(x)}\left||\nabla^2\tilde{u}|-|\nabla^2 v|\right|+c \fint_{B^+_\rho(x)}\left||\nabla^2 v|-(|\nabla^2 v|)_{B^+_\rho(x)}\right|\\
			&\le c\fint_{B^+_\rho(x)}\left|\nabla^2w\right|+2c \fint_{B^+_\rho(x)}\left|\left|\nabla^2 v\right|-\left|(\nabla^2 v)_{B^+_\rho(x)}\right|\right|\\
			&\le c\left( \fint_{B^+_\rho(x)}|\nabla^2 w|^{\bar{p}}\right)^\frac{1}{\bar{p}}+2c \fint_{B^+_\rho(x)}\left|\nabla^2 v-(\nabla^2 v)_{B^+_\rho(x)}\right| \\
			&\stackrel{\eqref{claim estimate}}{\le}C\left(\frac{\rho}{R}\right)\left(\fint_{B^+_{3R}(x)}|\nabla^2 \tilde{u}|^{\bar{p}}\right)^\frac{1}{\bar{p}}+C\left(\frac{R}{\rho}\right)^\frac{n}{\bar{p}} \left(\fint_{B^+_{5R/8}(y)}|\nabla^2 w|^{\bar{p}}\right)^\frac{1}{\bar{p}} .
		\end{align*} Since $w$ weakly solves \eqref{eq for w}, by Proposition \ref{Global Lp estimates LH system prop}, we have 
		\begin{align*}
			\|\nabla^2w\|_{L^{\bar{p}}(\mathcal{U}_R(y))} &\le C\left[\|\mathcal{P}_1\|_{L^{\bar{p}}(\mathcal{U}_R(y))}+\|\operatorname{div}\left(\zeta\mathcal{Q}\right)\|_{L^{\bar{p}}(\mathcal{U}_R(y))}\right].
		\end{align*}
		A simple scaling argument proves that the constant in the estimate is independent of $R.$ Now observe that 
		\begin{align}\label{second order term isolated}
			\operatorname{div}\left(\zeta\mathcal{Q}\right) &=	\operatorname{div}\left(\zeta\mathcal{B}\right) + 	\operatorname{div}\left(\zeta\mathrm{Q}u\right) + 	\operatorname{div}\left(\zeta\mathrm{S}\nabla u\right) \notag \\
			&=\operatorname{div}\left(\zeta\mathcal{B}\right) + 	\operatorname{div}\left(\zeta\mathrm{Q}u\right) +\operatorname{div}\left(\mathrm{S}\nabla \tilde{u}\right) - \operatorname{div}\left(\mathrm{S}\left( \nabla \zeta \wedge u\right)\right).
		\end{align} Hence, we have 
		\begin{align*}
			& \fint_{B^+_\rho(x)}\left\lvert
			|\nabla^2\tilde{u}|-(|\nabla^2\tilde{u}|)_{x,\rho}\right\rvert\\
			&\le C\left(\frac{\rho}{R}\right)\left(\fint_{B^+_{3R}(x)}|\nabla^2 \tilde{u}|^{\bar{p}}\right)^\frac{1}{\bar{p}}+C\left(\frac{R}{\rho}\right)^\frac{n}{\bar{p}}\left( \fint_{B^+_{3R}(x)}|\mathrm{S}\nabla^2\tilde{u}|^{\bar{p}}\right)^\frac{1}{\bar{p}}\\
			&\qquad\quad+C\left(\frac{R}{\rho}\right)^\frac{n}{\bar{p}}\left[\left( \fint_{B^+_{3R}(x)}|\mathcal{P}_1|^{\bar{p}}\right)^\frac{1}{\bar{p}}+\left( \fint_{B^+_{3R}(x)}|\mathcal{P}_2|^{\bar{p}}\right)^\frac{1}{\bar{p}}\right],
		\end{align*}
		where 
		\begin{align*}
			\mathcal{P}_2&=	\operatorname{div}\left(\zeta\mathcal{Q}\right)-\mathrm{S}\nabla^2\tilde{u}.
		\end{align*}
		This completes the proof.  
	\end{proof}
	Observe that the first two terms right hand side of the estimate \eqref{localized hessian mean osc decay} contains second derivatives of $\tilde{u},$ whereas $\mathcal{P}_1$ and $\mathcal{P}_2$ can be estimated in terms of up to first order derivatives of $\omega.$ The crux of our entire work is the following lemma, which says that \emph{we can choose parameters to make the constants appearing in front of the first two terms to be as small as we wish}. This allows us to absorb these terms in the left hand side.  
	\begin{lemma}\label{maximal ineq lemma}
		For any $\delta>0,$ there exist constants 
		\begin{align*}
			\bar{R} &= \bar{R}\left(\delta, \bar{p}, R_{0}, \left\lVert A\right\rVert_{C^{0,1}}, \left\lVert B\right\rVert_{C^{1,1}},  \left\lVert \Phi\right\rVert_{C^{2,1}}, \left\lVert \Phi^{-1}\right\rVert_{C^{2,1}}  \right) \in (0, R_{0}), \\
			C_\delta&=C_\delta\left( \delta, \bar{p}, \mathrm{data},  R_{0},  \left\lVert A\right\rVert_{C^{0,1}}, \left\lVert B\right\rVert_{C^{1,1}},  \left\lVert \Phi\right\rVert_{C^{2,1}}, \left\lVert \Phi^{-1}\right\rVert_{C^{2,1}} \right)>0, \\
			d&=d\left( \delta, k_0, \bar{p}, \mathrm{data}, R_{0}, \left\lVert A\right\rVert_{C^{0,1}}, \left\lVert B\right\rVert_{C^{1,1}}, \left\lVert \Phi\right\rVert_{C^{2,1}}, \left\lVert \Phi^{-1}\right\rVert_{C^{2,1}} \right)>0, 
		\end{align*}
		where $$ \mathrm{data} = \left( \gamma_{\bar{A}}, \gamma_{\bar{B}}, L, n, k, N \right),$$ such that in the setting of Lemma \ref{poinwise decay lemma}, for any $x\in B^+(0,k_0d)$ and any $0 < \rho \le k_0d,$ we have 
		\begin{align}\label{pointwise maximal ineq estimate}
			\fint_{B^+_\rho(x)} &\left\lvert \left\lvert \nabla^2\tilde{u}\right\rvert-\left( \left\lvert \nabla^2\tilde{u}\right\rvert \right)_{B^+_\rho(x)}\right\rvert \notag \\
			&\le \delta\left[M(|\nabla^2\tilde{u}|^{\bar{p}} \mathbbm{1}_{\mathbb{R}^{n}_{+}})(x)\right]^\frac{1}{\bar{p}} \notag \\
			&\qquad\qquad+C_\delta\left(\left[M(|\mathcal{P}_1|^{\bar{p}}\mathbbm{1}_{\mathbb{R}^{n}_{+}})(x)\right]^\frac{1}{\bar{p}}+\left[M(|\mathcal{P}_2|^{\bar{p}}\mathbbm{1}_{\mathbb{R}^{n}_{+}})(x)\right]^\frac{1}{\bar{p}}\right).
	\end{align}\end{lemma}
	\begin{proof}
		From Lemma \ref{poinwise decay lemma}, for any $0 < \bar{R}< R_{0}$, any $h\ge 48$ and any $d>0$ with $B^+(0,4hk_0d)\subset B^+(0,\bar{R}),$ we have 
		\begin{align*}
			\fint_{B^+_\rho(x)}&\left\lvert
			|\nabla^2\tilde{u}|-(|\nabla^2\tilde{u}|)_{x,\rho}\right\rvert \notag \\
			&\le C_{1}\left(\frac{\rho}{R}\right)\left(\fint_{B^+_R(x)}|\nabla^2 \tilde{u}|^{\bar{p}}\right)^\frac{1}{\bar{p}}+C_{1}\left(\frac{R}{\rho}\right)^\frac{n}{\bar{p}}\left( \fint_{B^+_{R}(x)}|\mathrm{S}\nabla^2\tilde{u}|^{\bar{p}}\right)^\frac{1}{\bar{p}} \notag \\
			&\qquad\qquad\quad+C_{1}\left(\frac{R}{\rho}\right)^\frac{n}{\bar{p}}\left[\left( \fint_{B^+_{R}(x)}|\mathcal{P}_1|^{\bar{p}}\right)^\frac{1}{p}+\left( \fint_{B^+_{R}(x)}|\mathcal{P}_2|^{\bar{p}}\right)^\frac{1}{\bar{p}}\right],
		\end{align*}
		for any $0 < R \le hk_{0}d$ and any $0<\rho\le R/48.$ Now since $\mathrm{S}(0)=0,$ we have 
		\begin{align*}
			\left\lVert \mathrm{S} \right\rVert_{L^{\infty}\left(B^+_R(x)\right)} \le \left\lVert \mathrm{S} \right\rVert_{L^{\infty}\left(B^+_{\bar{R}}\right)} \leq \left\lVert \nabla \mathrm{S} \right\rVert_{L^{\infty}\left(B^+_{R_{0}}\right)}\bar{R} \leq C_{0}\bar{R}, 
		\end{align*}
		where $$C_{0} = C_{0} \left( n, k, N, x_{0}, \Omega, \left\lVert A\right\rVert_{C^{0,1}}, \left\lVert B\right\rVert_{C^{1,1}}, \left\lVert \Phi\right\rVert_{C^{2,1}}\right)>0$$ is a constant. 
		Combining the last two estimates and using the definition of maximal functions, we deduce 
		\begin{align*}
			\fint_{B^+_\rho(x)}&\left\lvert
			|\nabla^2\tilde{u}|-(|\nabla^2\tilde{u}|)_{x,\rho}\right\rvert \notag \\
			&\le \left[ C_{1}\left(\frac{\rho}{R}\right)+ C_{1}\left\lVert \mathrm{S} \right\rVert_{L^{\infty}\left(B^+_R(x)\right)}\left(\frac{R}{\rho}\right)^\frac{n}{\bar{p}} \right]\left[M(|\nabla^2\tilde{u}|^{\bar{p}} \mathbbm{1}_{\mathbb{R}^{n}_{+}})(x)\right]^\frac{1}{\bar{p}}\notag \\
			&\qquad\qquad\quad+C_{1}\left(\frac{R}{\rho}\right)^\frac{n}{\bar{p}}\left[\left( \fint_{B^+_{R}(x)}|\mathcal{P}_1|^{\bar{p}}\right)^\frac{1}{p}+\left( \fint_{B^+_{R}(x)}|\mathcal{P}_2|^{\bar{p}}\right)^\frac{1}{\bar{p}}\right] \\
			&\le \left[ C_{1}\left(\frac{\rho}{R}\right)+ C_{1}C_{0}\bar{R}\left(\frac{R}{\rho}\right)^\frac{n}{\bar{p}} \right]\left[M(|\nabla^2\tilde{u}|^{\bar{p}} \mathbbm{1}_{\mathbb{R}^{n}_{+}})(x)\right]^\frac{1}{\bar{p}}\notag \\
			&\qquad\qquad\quad+C_{1}\left(\frac{R}{\rho}\right)^\frac{n}{\bar{p}}\left(\left[M(|\mathcal{P}_1|^{\bar{p}}\mathbbm{1}_{\mathbb{R}^{n}_{+}})(x)\right]^\frac{1}{\bar{p}}+\left[M(|\mathcal{P}_2|^{\bar{p}}\mathbbm{1}_{\mathbb{R}^{n}_{+}})(x)\right]^\frac{1}{\bar{p}}\right).
		\end{align*}
		Now assume that $R=\rho h,$ for $\rho\in(0,k_0d].$ Then $$\frac{C_{1}}{h}<\frac{\delta}{2} \qquad \text{ whenever } \qquad h>\frac{2C_{1}}{\delta}.$$ We fix $h > \max \left\lbrace 48, 2C_{1}/\delta \right\rbrace$. Now choose $\bar{R} \in (0, R_{0})$ small enough such that $C_{1}C_{0}\bar{R}h^\frac{n}{\bar{p}} < \delta/2.$ Observe that since we need $B^+(0,4hk_0d)\subset B^+(0,\bar{R}),$ we have to choose $$d<\frac{\delta}{8C_{0}C_{1}k_0h^{\left(\frac{n}{\bar{p}}+1\right)}}.$$ Then from above estimates, we have 
		\begin{align*}
			\fint_{B^+_\rho(x)}&\left\lvert
			|\nabla^2\tilde{u}|-(|\nabla^2\tilde{u}|)_{x,\rho}\right\rvert	\notag \\		&\le \delta\left[M(|\nabla^2\tilde{u}|^{\bar{p}} \mathbbm{1}_{\mathbb{R}^{n}_{+}})(x)\right]^\frac{1}{\bar{p}}\\
			&\qquad \qquad  +C_{1}\left(\frac{R}{\rho}\right)^\frac{n}{\bar{p}}\left(\left[M(|\mathcal{P}_1|^{\bar{p}}\mathbbm{1}_{\mathbb{R}^{n}_{+}})(x)\right]^\frac{1}{\bar{p}}+\left[M(|\mathcal{P}_2|^{\bar{p}}\mathbbm{1}_{\mathbb{R}^{n}_{+}})(x)\right]^\frac{1}{\bar{p}}\right).
		\end{align*}
		Now choosing $C_\delta=C_{1}h^\frac{n}{\bar{p}},$ completes the proof. 		\end{proof}
	
	\subsection{Weighted estimate near the boundary}
	Now we show that Lemma \ref{maximal ineq lemma} allows us to get a local estimates near the boundary in weighted spaces for the Hessian.  
	\begin{lemma}\label{weighted boundary estimate lemma}
		Let $\Omega \subset \mathbb{R}^{n}$ be open, bounded and $C^{2,1}.$ Let $f \in C^{0,1}\left( \overline{\Omega}; \varLambda^{k} \right)$  and $\lambda \in \mathbb{R}$ be given. Let $A \in C^{0,1}\left( \overline{\Omega}; \operatorname{Hom}\left(\varLambda^{k+1}\right)\right)$ and $B \in C^{1,1}\left( \overline{\Omega}; \operatorname{Hom}\left(\varLambda^{k}\right)\right)$ both be uniformly Legendre elliptic. Assume one of the following holds. 
		\begin{enumerate}[(a)]
			\item $\omega \in W_{T}^{1,2}\left(\Omega; \varLambda^{k}\right) \cap W^{2,p}\left(\Omega; \varLambda^{k}\right)$ satisfy \eqref{maineqn}. 
			\item $\omega \in W_{N}^{1,2}\left(\Omega; \varLambda^{k}\right) \cap W^{2,p}\left(\Omega; \varLambda^{k}\right)$ satisfy \eqref{maineq normal}.
		\end{enumerate}
		Let $1 < p < \infty$ and $M_{0} >0$ be real numbers. Let $w \in A_{p}$ with $[w]_{A_p}\le M_0.$\smallskip

		Then for every $x_{0} \in \partial\Omega,$ there exist   
		\begin{enumerate}[(i)]
			\item a neighborhood $U_{x_{0}}$ of $x_{0}$ in $\mathbb{R}^{n}$, a positive number $0 < R_{0, x_{0}}< 1,$ such that there exists an admissible boundary coordinate system $\Phi_{x_{0}} \in \operatorname{Diff}^{2,1}\left( \overline{B_{R_{0, x_{0}}}}; \overline{U_{x_{0}}}\right)$ satisfying $\Phi_{x_{0}} \left( 0\right)= x_{0}$ and 
			\begin{align*}
				D\Phi_{x_{0}} \left( 0\right) \in \mathbb{SO}\left(n\right), \  \Phi_{x_{0}} \left( B_{R_{0}}^{+}\right) = \Omega \cap U_{x_{0}}, \  \Phi_{x_{0}} \left( \Gamma_{R_{0}}\right) = \partial\Omega \cap U_{x_{0}},
			\end{align*}
			all depending only on $x_{0}$ and $\Omega,$
			\item a real number 
			\begin{align*}
				d_{x_{0}} = d_{x_{0}}\left( x_{0}, \Omega, k_0, p, M_{0}, \gamma_{A}, \gamma_{B}, k, n, N,  \left\lVert A\right\rVert_{C^{0,1}}, \left\lVert B\right\rVert_{C^{1,1}}\right) >0
			\end{align*}
		\end{enumerate}
		such that  for any $\theta \in C_{c}^{\infty}\left( U_{x_{0}}\right)$ satisfying 
		\begin{align*}
			0\le \theta \le 1, \quad |\nabla\theta|\le 16/d_{x_{0}} \quad \text{ and } \quad |\nabla^2\theta|\le 16/d_{x_{0}}^2,
		\end{align*} 
		we have the estimate 
		\begin{align}\label{near the boundary weighted estimate}
			&\int_{\Phi_{x_{0}}\left(B^{+}_{d_{x_{0}}/2}(0)\right)} \left\lvert \nabla^{2} \left(\theta \omega\right) \right\rvert^{p}w(x)\mathrm{d}x \notag
			\\&\qquad  \le \frac{C}{d_{x_{0}}^{4p}} \left[ \int_{\Omega} \left\lvert \omega \right\rvert^{p}w(x)\mathrm{d}x + \int_{\Omega} \left\lvert \nabla \omega \right\rvert^{p}w(x)\mathrm{d}x + \int_{\Omega} \left\lvert f \right\rvert^{p}w(x)\mathrm{d}x\right],
		\end{align}
		where  
		\begin{align*}
			C_{x_{0}} = C_{x_{0}}\left( x_{0}, \Omega, k_0, p, M_{0}, \gamma_{A}, \gamma_{B}, k, n, N,  \left\lVert A\right\rVert_{C^{0,1}}, \left\lVert B\right\rVert_{C^{1,1}}\right) >0
		\end{align*}
		is a constant. 
	\end{lemma}
	\begin{remark}
		The dependence of $d_{x_{0}}$ and $C_{x_{0}}$ on $x_{0}$ and $\Omega$ is only via the number $R_{0, x_{0}}$ and the bounds on $ \left\lVert \Phi_{x_{0}}\right\rVert_{C^{2,1}},  \left\lVert \Phi^{-1}_{x_{0}}\right\rVert_{C^{2,1}}.$
	\end{remark}
	\begin{proof}
		Given $x_{0} \in \partial\Omega,$ the existence of $U_{x_{0}}$, $R_{0, x_{0}}$ and $\Phi_{x_{0}}$ is given by Lemma \ref{flattening lemma}. By Lemma \ref{poinwise decay lemma} and Lemma \ref{maximal ineq lemma}, for any $\delta >0$ and any $1 < \bar{p} < p,$ there exist constants $d_{\delta}>0$ and $C_{\delta} >0$ such that we have the estimate 
		\begin{align}\label{pre FS estimate}
			\fint_{B^+_\rho(x)} &\left\lvert \left\lvert \nabla^2\tilde{u}\right\rvert-\left( \left\lvert \nabla^2\tilde{u}\right\rvert \right)_{B^+_\rho(x)}\right\rvert \notag\\
			&\le \delta\left[M(|\nabla^2\tilde{u}|^{\bar{p}} \mathbbm{1}_{\mathbb{R}^{n}_{+}})(x)\right]^\frac{1}{\bar{p}} \notag \\
			&\qquad\qquad+C_\delta\left(\left[M(|\mathcal{P}_1|^{\bar{p}}\mathbbm{1}_{\mathbb{R}^{n}_{+}})(x)\right]^\frac{1}{\bar{p}}+\left[M(|\mathcal{P}_2|^{\bar{p}}\mathbbm{1}_{\mathbb{R}^{n}_{+}})(x)\right]^\frac{1}{\bar{p}}\right),
		\end{align}
		for any cut-off function $\theta \in C_{c}^{\infty}\left( U_{x_{0}}\right)$ satisfying 
		\begin{align*}
			0\le \theta \le 1, \quad |\nabla\theta|\le 16/d_{\delta} \quad \text{ and } \quad |\nabla^2\theta|\le 16/d_{\delta}^2,
		\end{align*} and any cut-off function $\zeta\in C^\infty_c\left(B(0,d_{\delta})\right)$ with $\zeta\equiv 1$ in $B\left(0,d_{\delta}/2\right),$ $0\le \zeta\le 1,$ $|\nabla\zeta|\le 16/d_{\delta}$ and $|\nabla^2\zeta|\le 16/d_{\delta}^2,$ where $\tilde{u},\mathcal{P}_{1}$, $\mathcal{P}_{2}$ are as given in Lemma \ref{maximal ineq lemma}. Now for $x\in\mathbb R^n_+,$ set $\psi(x)=|\nabla^2\tilde{u}|.$ Let $\Psi$ be the extension of $\psi$ defined in Definition \ref{extension by reflection}. Then by Lemma \ref{extension estimate lemma}, we deduce 
		\begin{align*}
			M^\sharp_{k_0d_{\delta}}\Psi(x)
			&\le c\sup_{0<\rho\le k_0d_{\delta}}\fint_{B^+_\rho(x)}\left||\nabla^2\tilde{u}|-(|\nabla^2\tilde{u}|)_{x,\rho}\right|\\
			&\stackrel{\eqref{pre FS estimate}}{\le} c\delta\left[M(|\nabla^2\tilde{u}|^{\bar{p}} \mathbbm{1}_{\mathbb{R}^{n}_{+}})(x)\right]^\frac{1}{\bar{p}}\\
			&\qquad\qquad+c C_\delta\left(\left[M(|\mathcal{P}_1|^{\bar{p}}\mathbbm{1}_{\mathbb{R}^{n}_{+}})(x)\right]^\frac{1}{\bar{p}}+\left[M(|\mathcal{P}_2|^{\bar{p}}\mathbbm{1}_{\mathbb{R}^{n}_{+}})(x)\right]^\frac{1}{\bar{p}}\right).
		\end{align*}
		Next we raise the above inequality to the $p$-th power and multiply by the weight $\tilde{w},$ where $\tilde{w}$ is the extension of $w$ given by Lemma \ref{extension of weight lemma}, to get the following pointwise inequality
		\begin{align*}
			&\left [ M^\sharp_{k_0d_{\delta}}\Psi(x)\right]^p\tilde{w}(x) \\
			&\quad\le  (c\delta)^p \left[M(|\nabla^2\tilde{u}|^{\bar{p}} \mathbbm{1}_{\mathbb{R}^{n}_{+}})(x)\right]^\frac{p}{\bar{p}}\tilde{w}(x)\\
			&\quad\qquad +  (cC_\delta)^p\left(\left[M(|\mathcal{P}_1|^{\bar{p}}\mathbbm{1}_{\mathbb{R}^{n}_{+}})(x)\right]^\frac{p}{\bar{p}}\tilde{w}(x)+\left[M(|\mathcal{P}_2|^{\bar{p}}\mathbbm{1}_{\mathbb{R}^{n}_{+}})(x)\right]^\frac{p}{\bar{p}}\tilde{w}(x)\right).
		\end{align*}
		By applying Fefferman--Stein inequality in Lemma \ref{Feff-Stein}, we have 
		\begin{align*}
			&  \int_{B(0,d_{\delta})}|\Psi(x)|^p\tilde{w}(x)dx\\
			&\le C_1\int_{B(0,k_0d_{\delta})}\left [ M^\sharp_{k_0d_{\delta}}\Psi(x)\right]^p\tilde{w}(x)dx\\
			&= 2C_1\int_{B^+(0,k_0d_{\delta})}\left [ M^\sharp_{k_0d_{\delta}}\Psi(x)\right]^p\tilde{w}(x)dx\\
			&\le\begin{aligned}[t]
				2C_1&(c\delta)^p \int_{\mathbb R^n}\left[M(|\nabla^2\tilde{u}|^{\bar{p}} \mathbbm{1}_{\mathbb{R}^{n}_{+}})(x)\right]^\frac{p}{\bar{p}}\tilde{w}(x)\mathrm{d}x \\ &+  2C_1(cC_\delta)^p\int_{\mathbb R^n}\left[M(|\mathcal{P}_1|^{\bar{p}}\mathbbm{1}_{\mathbb{R}^{n}_{+}})(x)\right]^\frac{p}{\bar{p}}\tilde{w}(x)\mathrm{d}x \\&\qquad \qquad \qquad \qquad+2C_1(cC_\delta)^p\int_{\mathbb{R}^{n}}\left[M(|\mathcal{P}_2|^{\bar{p}}\mathbbm{1}_{\mathbb{R}^{n}_{+}})(x)\right]^\frac{p}{\bar{p}}\tilde{w}(x)\mathrm{d}x.
			\end{aligned}
		\end{align*}
		Now let $p_{0} = p_{0}\left(n, p, M_{0}\right) \in (1, p)$ be the exponent given by Proposition \ref{higher power Ap}. Choosing $\bar{p} = p/p_{0} \in (1,p),$ we see that $\tilde{w} \in A_{p/\bar{p}}.$ Hence, using Lemma \ref{Ap maximal fn boundedness} with the exponent $p/\bar{p}$,  we deduce  
		\begin{align*}
			&\int_{B(0,d_{\delta})}|\Psi(x)|^p\tilde{w}(x)dx \\
			&\le 2C_1C_{p/\bar{p}}(c\delta)^p \int_{\mathbb R^n}|\nabla^2\tilde{u}|^{p} \mathbbm{1}_{\mathbb{R}^{n}_{+}}(x)\tilde{w}(x)\mathrm{d}x \\
			&\qquad +2C_1C_{p/\bar{p}}(cC_\delta)^p\left[ \int_{\mathbb R^n}|\mathcal{P}_1|^{p}\mathbbm{1}_{\mathbb{R}^{n}_{+}}(x)\tilde{w}(x)\mathrm{d}x + \int_{\mathbb R^n}|\mathcal{P}_2|^{p}\mathbbm{1}_{\mathbb{R}^{n}_{+}}(x)\tilde{w}(x)\mathrm{d}x\right]. 
		\end{align*}
		Observe that since $\zeta$ is supported in $B(0,d_{\delta}),$ $\tilde{u}= \zeta u$, $\mathcal{P}_{1}, \mathcal{P}_{2}$ also have supports in $B(0,d_{\delta})$ and thus all the integrals on the right hand side above vanishes outside $B(0,d_{\delta}).$ This is really the reason for using the cut-off function $\zeta.$ Thus, we deduce 	
		\begin{align*}
			\int_{B(0,d_{\delta})}&|\Psi(x)|^p\tilde{w}(x)dx \\&\begin{aligned}[t]
				\le 2C_1&C_{p/\bar{p}}(c\delta)^p \int_{B(0,d_{\delta})}|\nabla^2\tilde{u}|^{p} \mathbbm{1}_{\mathbb{R}^{n}_{+}}(x)\tilde{w}(x)\mathrm{d}x \\
				&+2C_1C_{p/\bar{p}}(cC_\delta)^p \int_{B(0,d_{\delta})}|\mathcal{P}_1|^{p}\mathbbm{1}_{\mathbb{R}^{n}_{+}}(x)\tilde{w}(x)\mathrm{d}x \\ &\qquad \qquad +2C_1C_{p/\bar{p}}(cC_\delta)^p \int_{B(0,d_{\delta})}|\mathcal{P}_2|^{p}\mathbbm{1}_{\mathbb{R}^{n}_{+}}(x)\tilde{w}(x)\mathrm{d}x. 
			\end{aligned}
		\end{align*}
		Now we choose $\delta>0$ small enough such that $2C_1C_{p/\bar{p}}(c\delta)^p < 1/2,$ to arrive at the inequality 
		\begin{align*}
			\int_{B(0,d_{\delta})}|\Psi(x)|^p\tilde{w}(x)\mathrm{d}x \le C \left[ \int_{B^{+}(0,d_{\delta})}|\mathcal{P}_1|^{p}\tilde{w}(x)\mathrm{d}x + \int_{B^{+}(0,d_{\delta})}|\mathcal{P}_2|^{p}\tilde{w}(x)\mathrm{d}x\right].
		\end{align*}
		Since $\zeta \equiv 1$ in $B(0,d_{\delta}/2),$ we deduce  
		\begin{align}\label{pre final estimate boundary}
			\int_{B^{+}(0,d_{\delta}/2)}&|\nabla^2 u|^p\tilde{w}(x)\mathrm{d}x \notag \\&\le \int_{B^{+}(0,d_{\delta})}|\nabla^{2}\tilde{u}|^p\tilde{w}(x)\mathrm{d}x \notag \\
			&= \int_{B(0,d_{\delta})}|\nabla^2\tilde{u}|^{p} \mathbbm{1}_{\mathbb{R}^{n}_{+}}(x)\tilde{w}(x)\mathrm{d}x \notag\\
			&\le 	\int_{B(0,d_{\delta})}|\Psi(x)|^p\tilde{w}(x)\mathrm{d}x \notag\\
			&\le C \left[ \int_{B^{+}(0,d_{\delta})}|\mathcal{P}_1|^{p}\tilde{w}(x)\mathrm{d}x + \int_{B^{+}(0,d_{\delta})}|\mathcal{P}_2|^{p}\tilde{w}(x)\mathrm{d}x\right].
		\end{align}
		Now that we have chosen our $\delta,$ we just write $d_{\delta}= d_{x_{0}}.$ Now we establish \eqref{near the boundary weighted estimate}.  We only prove the case when (a) holds, the other is similar. To do this, we want to estimate the right hand side of \eqref{pre final estimate boundary}.  We recall that in this case, we have  
		\begin{align*}
			\mathcal{P}_1&=\mathcal{P}-\nabla\zeta\otimes\mathcal{Q}, \\
			\mathcal{P}_2&=\operatorname{div}\left(\zeta\mathcal{Q}\right)-\mathrm{S}\nabla^2\tilde{u}, \\
			\mathcal{Q}&=\mathcal{B}+\mathrm{Q}u+\mathrm{S}\nabla u, \\ 
			\mathcal{P}&= \bar{B}^{\intercal}\left[-d\left(d\zeta\lrcorner\bar{B}u\right)+d\zeta\wedge d^*\left(\bar{B}u\right)\right]+d\zeta\lrcorner\bar{A}du -d^*\left[\bar{A}\left(d\zeta\wedge u\right)\right]\\
			&\qquad \qquad \qquad \qquad \qquad \qquad \qquad \qquad \qquad +\zeta\mathcal{A}+\zeta \mathrm{P}u+\zeta \mathrm{R}\nabla u +\zeta\tilde{f}.
		\end{align*}
		Since $|\nabla\zeta|\le 16/d_{x_{0}}$, $|\nabla^2\zeta|\le 16/d_{x_{0}}^2$ and $d_{x_{0}} < 1,$ we have 
		\begin{multline}\label{P1 estimate}
			\left\lvert \mathcal{P}_1 \right\rvert \le \frac{C}{d_{x_{0}}^2}\left( \left\lvert u \right\rvert + \left\lvert \nabla u \right\rvert \right) + \left\lvert \tilde{f}\right\rvert  + \left\lvert \mathrm{P} \right\rvert  \left\lvert u \right\rvert + \left\lvert \mathrm{R} \right\rvert  \left\lvert \nabla u \right\rvert + \left\lvert \mathcal{A} \right\rvert
			\\ + \frac{1}{d_{x_{0}}^{2}}\left( \left\lvert \mathcal{B} \right\rvert + \left\lvert \mathrm{Q} \right\rvert  \left\lvert u \right\rvert + \left\lvert \mathrm{S} \right\rvert  \left\lvert \nabla u \right\rvert\right). 
		\end{multline}
		In view of \eqref{second order term isolated}, we have 
		\begin{multline}
			\left\lvert \mathcal{P}_2 \right\rvert \le \frac{1}{d_{x_{0}}^{2}}\left( \left\lvert \mathcal{B} \right\rvert + \left\lvert \nabla \mathcal{B} \right\rvert + \left\lvert \mathrm{S} \right\rvert\left\lvert u \right\rvert\right)  + \frac{1}{d_{x_{0}}}\left( \left\lvert \mathrm{Q} \right\rvert\left\lvert u \right\rvert + \left\lvert \nabla \mathrm{S} \right\rvert\left\lvert u \right\rvert + \left\lvert \mathrm{S} \right\rvert\left\lvert \nabla u \right\rvert\right) \\ + \left\lvert \nabla \mathrm{Q} \right\rvert\left\lvert u \right\rvert + \left( \left\lvert \mathrm{Q} \right\rvert + \left\lvert \nabla \mathrm{S} \right\rvert\right)\left\lvert \nabla u \right\rvert. 
		\end{multline}
		Since $u=\Phi_{x_{0}}^{\ast}\left(\theta \omega\right),$ in view of the pointwise dependencies in Lemma \ref{flattening lemma}, we get 
		\begin{align*}
			\left\lvert \mathcal{P}_1 \right\rvert \le \frac{C}{d_{x_{0}}^3}\left( \left\lvert \omega \right\rvert + \left\lvert \nabla \omega \right\rvert \right) +C \left\lvert f \right\rvert \qquad \text{ and } \qquad \left\lvert \mathcal{P}_2 \right\rvert \le \frac{C}{d_{x_{0}}^4}\left( \left\lvert \omega \right\rvert + \left\lvert \nabla \omega \right\rvert \right). 
		\end{align*}
		Combining these estimates with \eqref{pre final estimate boundary}, we have \eqref{near the boundary weighted estimate}. 
	\end{proof}

	\section{Main results}\label{main theorems}
	\subsection{Global apriori estimates for Hodge systems}
	\subsubsection{Main estimate}
	\begin{theorem}\label{mainthm}
		Let $r \ge 0$ be an integer and let $\Omega\subset\mathbb{R}^n$ be an open, bounded,  $C^{r+2,1}$ domain. Let $A\in C^{r,1}\left(\overline{\Omega}, \operatorname{Hom}\left(\varLambda^{k+1}\right)\right)$  and $B\in C^{r+1,1}\left(\overline{\Omega},\operatorname{Hom}\left(\varLambda^k\right)\right)$ be both uniformly Legendre elliptic with constants $\gamma_A$ and $\gamma_B$, respectively. Let $f \in C^{r,1}\left(\overline{\Omega},\varLambda^{k}\right),$ $\lambda \in \mathbb{R}$ and $\omega_{0} \in C^{r+2,1}\left(\overline{\Omega},\varLambda^{k}\right)$ be given. Let $\omega\in W^{1,2}(\Omega,\varLambda^k)$ weakly solves either of the following systems: 
		\begin{align*}
			\left\lbrace\begin{aligned}
				\mathfrak{L}\omega&=\lambda B\omega + f && \text{ in } \Omega,\\
				\mathfrak{b}_{\text{t}}\omega&=\mathfrak{b}_{\text{t}}\omega_{0} && \text{ on } \partial\Omega,		\end{aligned}\right. \text{ or } \left\lbrace\begin{aligned}
				\mathfrak{L}\omega&=\lambda B\omega + f&& \text{ in } \Omega,\\
				\mathfrak{b}_{\text{n}}\omega&=\mathfrak{b}_{\text{n}}\omega_{0} && \text{ on } \partial\Omega.		\end{aligned}\right.
		\end{align*}
		Let $M_0>0,$ $1 < p < \infty$ and let $w\in A_p$ and $[w]_{A_p}\le M_0.$ Then there exists a constant $C=C(r, n,N,\Omega,\gamma_A,\gamma_B, \lambda, \|A\|_{C^{r,1}},\|B\|_{C^{r+1,1}},M_0,p)>0$ such that 
		\begin{align}\label{main apriori estimate}
			\left\lVert \omega \right\rVert_{W^{r+2, p}_{w}\left(\Omega; \varLambda^{k}\right)}\le C\left[\left\lVert \omega \right\rVert_{L^{p}_{w}\left(\Omega; \varLambda^{k}\right)}+ \left\lVert f \right\rVert_{W^{r, p}_{w}\left(\Omega; \varLambda^{k}\right)} + \left\lVert \omega_{0} \right\rVert_{W^{r+2, p}_{w}\left(\Omega; \varLambda^{k}\right)}\right].
		\end{align}
	\end{theorem}
	\begin{proof}
		Since the equations are linear, we can assume $\omega_{0}=0$ by replacing $\omega$ by $\omega - \omega_{0}.$ We only show the case $r=0,$ other cases can be obtained by iterating the argument in the usual manner. Using Lemma \ref{weighted boundary estimate lemma}, for every $x_{0} \in \partial\Omega,$ there exist $d_{x_{0}} >0$ and $C^{2,1}$ diffeomorphism $\Phi_{x_{0}}.$ Clearly, $\left\{\Phi_{x_{0}}\left(B^+\left(0,\frac{d_{x_{0}}}{2}\right)\right):x_{0}\in \partial\Omega\right\}$ is an open cover of $\partial\Omega.$ So by compactness of $\partial\Omega,$ there exists $m\in \mathbb N$ and points $x_{1}, \ldots, x_{m} \in \partial\Omega$ such that \begin{align*}
			\partial\Omega\subset\bigcup\limits_{i=1}^{m}\Phi_{x_{i}}\left(B\left(0,\frac{d_{x_i}}{2}\right)\right).
		\end{align*} Thus we can choose $\Omega_0\subset\subset\Omega$ such that 
		\begin{align*}
			\Omega \subset \Omega_0 \bigcup \left[ \bigcup\limits_{i=1}^{m}\Phi_{x_{i}}\left(B^{+}\left(0,\frac{d_{x_i}}{2}\right)\right)\right]. 
		\end{align*} Set $\Omega_{i}= \Phi_{x_{i}}\left(B^{+}\left(0,\frac{d_{x_i}}{2}\right)\right)$ for $1 \le i \le m.$ Set 
		\begin{align*}
			d = \frac{1}{2} \min \left\lbrace \operatorname{dist}\left(\Omega_0, \partial \Omega\right), \left\lbrace d_{x_{i}}\right\rbrace_{1\le i \le m} \right\rbrace, 
		\end{align*} Let $\left\lbrace \theta_{i} \right\rbrace_{i=0}^{m}$ be a partition of unity subordinate to the cover $\Omega \subset \bigcup\limits_{i=0}^{m} \Omega_{i}$ such that 
		\begin{align*}
			|\nabla\theta_{i}|\le 16/d \quad \text{ and } \quad |\nabla^2\theta_{i}|\le 16/d^2 \qquad \text{ for all } 0 \le i \le m.
		\end{align*} 
		The interior estimate 
		\begin{multline*}
			\int_{\Omega_{0}} \left\lvert \nabla^{2} \left(\theta_{0} \omega\right) \right\rvert^{p}w(x)\mathrm{d}x \notag
			\\  \le \frac{C_{0}}{d^{4p}} \left[ \int_{\Omega} \left\lvert \omega \right\rvert^{p}w(x)\mathrm{d}x + \int_{\Omega} \left\lvert \nabla \omega \right\rvert^{p}w(x)\mathrm{d}x + \int_{\Omega} \left\lvert f \right\rvert^{p}w(x)\mathrm{d}x\right],
		\end{multline*} is considerably simpler and hence we skip the details. From Lemma \ref{weighted boundary estimate lemma}, using $\theta = \theta_{i}$, $1 \le i \le m,$ we deduce  
		\begin{align*}
			&\int_{\Omega_{i}} \left\lvert \nabla^{2} \left(\theta_{i} \omega\right) \right\rvert^{p}w(x)\mathrm{d}x \notag
			\\&\qquad  \le \frac{C_{i}}{d^{4p}} \left[ \int_{\Omega} \left\lvert \omega \right\rvert^{p}w(x)\mathrm{d}x + \int_{\Omega} \left\lvert \nabla \omega \right\rvert^{p}w(x)\mathrm{d}x + \int_{\Omega} \left\lvert f \right\rvert^{p}w(x)\mathrm{d}x\right],
		\end{align*}
		for all $1 \le i \le m.$ Thus, we have 
		\begin{align*}
			\int_{\Omega} \left\lvert \nabla^{2} \omega \right\rvert^{p}w(x)\mathrm{d}x  &= \int_{\Omega} \left\lvert \nabla^{2} \left( \sum\limits_{i=0}^{m} \theta_{i}\omega \right)  \right\rvert^{p}w(x)\mathrm{d}x \\
			&\le C \sum\limits_{i=0}^{m} \int_{\Omega_{i}} \left\lvert \nabla^{2} \left( \theta_{i}\omega \right)  \right\rvert^{p}w(x)\mathrm{d}x \\
			&\le C \left[ \int_{\Omega} \left\lvert \omega \right\rvert^{p}w(x)\mathrm{d}x + \int_{\Omega} \left\lvert \nabla \omega \right\rvert^{p}w(x)\mathrm{d}x + \int_{\Omega} \left\lvert f \right\rvert^{p}w(x)\mathrm{d}x\right].
		\end{align*}
		Using standard interpolation techniques to get rid of the gradient term on the right, this implies \eqref{main apriori estimate} for $r=0$ and $\omega_{0}=0$ and completes the proof.   
	\end{proof}
\subsubsection{Uniqueness and apriori estimates}
	Now we show that the apriori estimates in Theorem \ref{mainthm} improves to a sharper one if uniqueness holds. 
	\begin{lemma}\label{uniqueness lemma}
		Assume the hypotheses of Theorem \ref{mainthm} holds. If $\lambda \in \mathbb{R}$ is not a tangential (resp. normal eigenvalue), then we can drop the $\omega$ term on the right hand side of the estimate 
		\eqref{main apriori estimate}. More precisely, in these cases the estimate \eqref{main apriori estimate} improves to 
		\begin{align*}
			\left\lVert \omega \right\rVert_{W^{r+2, p}_{w}\left(\Omega; \varLambda^{k}\right)}\le C\left[ \left\lVert f \right\rVert_{W^{r, p}_{w}\left(\Omega; \varLambda^{k}\right)} + \left\lVert \omega_{0} \right\rVert_{W^{r+2, p}_{w}\left(\Omega; \varLambda^{k}\right)}\right].
		\end{align*}
	\end{lemma}
	\begin{proof}
		We show only for $r=0.$ Without loss of generality we can assume $\omega_{0}=0.$ In view of \eqref{main apriori estimate}, we only need to prove the estimate 
		\begin{align*}
			\left\lVert \omega \right\rVert_{L^{p}_{w}\left(\Omega; \varLambda^{k}\right)}\le C \left\lVert f \right\rVert_{W^{r, p}_{w}\left(\Omega; \varLambda^{k}\right)}.
		\end{align*}
		If the estimate is false, then there exist sequences $\left\lbrace \omega_{s}\right\rbrace_{s \in \mathbb{N}} \subset W^{2, p}_{w}\left(\Omega; \varLambda^{k}\right)$ and $\left\lbrace f_{s}\right\rbrace_{s \in \mathbb{N}} \subset L^{p}_{w}\left(\Omega; \varLambda^{k}\right)$ such that 
		\begin{align*}
			\left\lVert \omega_{s} \right\rVert_{L^{p}_{w}\left(\Omega; \varLambda^{k}\right)} =1 \qquad \text{ and } \qquad 	\left\lVert f_s \right\rVert_{L^{p}_{w}\left(\Omega; \varLambda^{k}\right)} < \frac{1}{s} \qquad \text{ for all } s \in \mathbb{N}, 
		\end{align*}
		and $\omega_{s}$ solves 
		\begin{align*}
			\left\lbrace\begin{aligned}
				\mathfrak{L}\omega_{s}&=\lambda B\omega_{s} + f_{s} && \text{ in } \Omega,\\
				\mathfrak{b}_{\text{t}}\omega_{s}&=0 && \text{ on } \partial\Omega,		\end{aligned}\right. \text{ or } \left\lbrace\begin{aligned}
				\mathfrak{L}\omega_{s}&=\lambda B\omega_{s} + f_{s}&& \text{ in } \Omega,\\
				\mathfrak{b}_{\text{n}}\omega_{s}&=0 && \text{ on } \partial\Omega.		\end{aligned}\right.
		\end{align*}
		By \eqref{main apriori estimate}, $\left\lbrace \omega_{s}\right\rbrace_{s \in \mathbb{N}}$ is uniformly bounded in $W^{2, p}_{w}\left(\Omega; \varLambda^{k}\right)$ and up to the extraction of a subsequence that we do not relabel, there exists $\omega \in W^{2, p}_{w}\left(\Omega; \varLambda^{k}\right)$	such that $\omega_{s} \rightharpoonup \omega$ weakly in $W^{2, p}_{w}\left(\Omega; \varLambda^{k}\right).$ As $f_{s} \rightarrow 0$ in $L^{p}_{w}\left(\Omega; \varLambda^{k}\right)$, this implies $\omega$ is a tangential (resp. normal) eigenform with eigenvalue $\lambda$ for $\mathfrak{L}.$ By our hypothesis, we must have $\omega =0.$ But this contradicts the fact that we must have $$	\left\lVert \omega \right\rVert_{L^{p}_{w}\left(\Omega; \varLambda^{k}\right)} = \lim\limits_{s \rightarrow \infty}	\left\lVert \omega_{s} \right\rVert_{L^{p}_{w}\left(\Omega; \varLambda^{k}\right)} =1,$$
		as the embedding $W^{2, p}_{w} \hookrightarrow L^{p}_{w}$ is always compact. 
	\end{proof}
	\subsection{Hodge systems}
	\begin{theorem}\label{Hodge system}
		Let $r \geq 0$ be an integer and let $\Omega \subset \mathbb{R}^{n}$ be an open, bounded, $C^{r+2,1}$ subset. Let $A\in C^{r,1}\left(\overline{\Omega}; \operatorname{Hom}\left(\varLambda^{k+1}\right)\right)$, $B\in C^{r+1,1}\left(\overline{\Omega}; \operatorname{Hom}\left(\varLambda^{k}\right)\right)$ be both uniformly $\gamma$-Legendre elliptic. Then the following holds:
		\begin{enumerate}[(a)]
			\item   For any $1< p < \infty$ and $M_{0}>0,$ if $\lambda \in \mathbb{R}$ is not a tangential eigenvalue for $\mathfrak{L},$ and $w \in A_{p}$ with $[w]_{A_p}\le M_0,$ then for any given $f \in W^{r,p}_{w}\left(\Omega; \varLambda^{k}\right)$ and any $\omega_{0} \in W^{r+2,p}_{w}\left(\Omega; \varLambda^{k}\right),$ there exists a unique solution $\omega \in W^{r+2,p}_{w}\left(\Omega; \varLambda^{k}\right)$ satisfying  
			\begin{align*}
				\left\lbrace \begin{aligned}
					\mathfrak{L}\omega &= \lambda B(x)\omega + f &&\text{ in } \Omega, \\
					\mathfrak{b}_{\text{t}}\omega &= \mathfrak{b}_{\text{t}}\omega_{0} &&\text{ on } \partial\Omega. 
				\end{aligned}\right. 
			\end{align*} 
			Furthermore, there exists a constant $$C = C \left( r, p, n, k, N, M_{0}, \gamma_{A}, \gamma_{B},  \lambda, \Omega, \left\lVert A \right\rVert_{C^{r,1}}, \left\lVert B \right\rVert_{C^{r+1,1}}\right) >0$$ such that we have the estimate 
			\begin{align}\label{tangential existence estimate}
				\left\lVert \omega\right\rVert_{W^{r+2, p}_{w}\left(\Omega; \varLambda^{k}\right)} \leq C \left( \left\lVert \omega_{0}\right\rVert_{W^{r+2,p}_{w}\left(\Omega; \varLambda^{k}\right)} + \left\lVert f\right\rVert_{W^{r, p}_{w}\left(\Omega; \varLambda^{k+1}\right)}\right). 
			\end{align}
			Moreover, if $\lambda =0$ is a tangential eigenvalue, then for any $f \in W^{r,p}_{w}\left(\Omega; \varLambda^{k}\right)$ and any $\omega_{0} \in W^{r+2,p}_{w}\left(\Omega; \varLambda^{k}\right)$ satisfying $f - \mathfrak{L}\omega_{0} \in \mathcal{H}_{T}^{\perp}\left(\Omega; \varLambda^{k}\right)$, there exists a unique solution $\omega \in W^{r+2,p}_{w}\left(\Omega; \varLambda^{k}\right)\cap \mathcal{H}_{T}^{\perp}\left(\Omega; \varLambda^{k}\right)$
			such that the estimate \eqref{tangential existence estimate} holds.   
			\begin{align*}
				\left\lVert \omega\right\rVert_{W^{r+2, p}_{w}\left(\Omega; \varLambda^{k}\right)} \leq C \left( \left\lVert \omega_{0}\right\rVert_{W^{r+2,p}_{w}\left(\Omega; \varLambda^{k}\right)} + \left\lVert f\right\rVert_{W^{r, p}_{w}\left(\Omega; \varLambda^{k+1}\right)}\right). 
			\end{align*}
			
			\item For any $1< p < \infty$ and $M_{0}>0,$ if $\lambda \in \mathbb{R}$ is not a normal eigenvalue for $\mathfrak{L},$ and $w \in A_{p}$ with $[w]_{A_p}\le M_0,$ then for any given $f \in W^{r,p}_{w}\left(\Omega; \varLambda^{k}\right)$ and any $\omega_{0} \in W^{r+2,p}_{w}\left(\Omega; \varLambda^{k}\right),$ there exists a solution $\omega \in W^{r+2,p}_{w}\left(\Omega; \varLambda^{k}\right)$ satisfying  
			\begin{align*}
				\left\lbrace \begin{aligned}
					\mathfrak{L}\omega &= \lambda B(x)\omega + f &&\text{ in } \Omega, \\
					\mathfrak{b}_{\text{n}}\omega &= \mathfrak{b}_{\text{n}}\omega_{0} &&\text{ on } \partial\Omega. 
				\end{aligned}\right. 
			\end{align*} 
			Furthermore, there exists a constant $$C = C \left( r, p, n, k, N, M_{0}, \gamma_{A}, \gamma_{B}, \lambda, \Omega, \left\lVert A \right\rVert_{C^{r,1}}, \left\lVert B \right\rVert_{C^{r+1,1}}\right) >0$$ such that we have the estimate 
			\begin{align}\label{normal existence estimate}
				\left\lVert \omega\right\rVert_{W^{r+2, p}_{w}\left(\Omega; \varLambda^{k}\right)} \leq C \left( \left\lVert \omega_{0}\right\rVert_{W^{r+2,p}_{w}\left(\Omega; \varLambda^{k}\right)} + \left\lVert f\right\rVert_{W^{r, p}_{w}\left(\Omega; \varLambda^{k+1}\right)}\right). 
			\end{align}
			Moreover, if $\lambda =0$ is a normal eigenvalue, then for any $f \in W^{r,p}_{w}\left(\Omega; \varLambda^{k}\right)$ and any $\omega_{0} \in W^{r+2,p}_{w}\left(\Omega; \varLambda^{k}\right)$ satisfying $f - \mathfrak{L}\omega_{0} \in \mathcal{H}_{N}^{\perp}\left(\Omega; \varLambda^{k}\right)$, there exists a unique solution $\omega \in W^{r+2,p}_{w}\left(\Omega; \varLambda^{k}\right)\cap \mathcal{H}_{N}^{\perp}\left(\Omega; \varLambda^{k}\right)$
			such that the estimate \eqref{normal existence estimate} holds. 
		\end{enumerate}
	\end{theorem}
	\begin{proof}
		We only show (a). Part (b) is similar. Also, we prove only for $r=0.$ Higher regularity can be inferred from this in the standard way. Since the problem is linear, it is enough to show the result for $\omega_{0}=0.$ Indeed, $\psi$ solves 
		\begin{align*}
			\left\lbrace \begin{aligned}
				\mathfrak{L}\psi &= \lambda B(x)\psi + f - \mathfrak{L}\omega_{0} + \lambda B(x)\omega_{0} &&\text{ in } \Omega, \\
				\mathfrak{b}_{\text{t}}\psi &= 0  &&\text{ on } \partial\Omega 
			\end{aligned}\right. 
		\end{align*} 
		if and only if $\omega = \psi + \omega_{0}$ solves our problem. So we assume $\omega_{0}=0.$ We approximate $f$ by a sequence $\left\lbrace f_{s}\right\rbrace_{s \in \mathbb{N}} \subset C^{0,1}\left( \overline{\Omega};\varLambda^{k}\right)$ in $W^{r, p}_{w}\left(\Omega; \varLambda^{k}\right)$ By usual Fredholm theory in $L^{2}$ and standard $L^{p}$ estimates established in \cite{Sil_linearregularity} implies that there exists a sequence $\left\lbrace u_{s}\right\rbrace_{s \in \mathbb{N}} \subset W^{2,2}\left(\Omega; \varLambda^{k}\right)$  such that for every $s \in \mathbb{N},$ $u_{s} \in W^{2,q}\left(\Omega; \varLambda^{k}\right)$ for all $1 \leq q < \infty$ and solves 
		\begin{align*}
			\left\lbrace \begin{aligned}
				\mathfrak{L}u_{s} &= \lambda B(x)u_{s} + f_{s} &&\text{ in } \Omega, \\
				\mathfrak{b}_{\text{t}}u_{s} &= 0  &&\text{ on } \partial\Omega.  
			\end{aligned}\right. 
		\end{align*} 
		By Theorem \ref{mainthm}, we have 
		\begin{align*}
			\left\lVert u_{s}\right\rVert_{W^{2, p}_{w}\left(\Omega; \varLambda^{k}\right)} \leq C \left( \left\lVert u_{s}\right\rVert_{L^{p}_{w}\left(\Omega; \varLambda^{k}\right)} + \left\lVert f_{s}\right\rVert_{L^{p}_{w}\left(\Omega; \varLambda^{k}\right)}\right). 
		\end{align*}
		If $\lambda$ is not a tangential eigenvalue for $\mathfrak{L},$ by Lemma \ref{uniqueness lemma}, we have 
		\begin{align*}
			\left\lVert u_{s}\right\rVert_{W^{2, p}_{w}\left(\Omega; \varLambda^{k}\right)} \leq C \left\lVert f_{s}\right\rVert_{L^{p}_{w}\left(\Omega; \varLambda^{k}\right)}. 
		\end{align*}
		Since $f_{s} \rightarrow f$ in $L^{p}_{w}\left(\Omega; \varLambda^{k}\right)$, the sequence $\left\lbrace u_{s}\right\rbrace_{s \in \mathbb{N}}$ is uniformly bounded in $W^{2, p}_{w}\left(\Omega; \varLambda^{k}\right)$. Since $W^{2, p}_{w}\left(\Omega; \varLambda^{k}\right)$ is a reflexive Banach space, there exists $u \in W^{2, p}_{w}\left(\Omega; \varLambda^{k}\right)$ such that $u_{s} \rightharpoonup u$ weakly in $W^{2, p}_{w}\left(\Omega; \varLambda^{k}\right).$ Now it is easy to see that this implies $u$ solves our problem with the desired estimates. If $\lambda=0$ is a tangential eigenvalue, then we have the result by applying the same argument in the space $W^{r+2,p}_{w}\left(\Omega; \varLambda^{k}\right)\cap \mathcal{H}_{T}^{\perp}\left(\Omega; \varLambda^{k}\right)$. This completes the proof. 
	\end{proof}
	\subsection{Hodge-Maxwell systems}
	\begin{theorem}\label{Hodge-Maxwell system}
		Let $r \geq 0$ be an integer and let $\Omega \subset \mathbb{R}^{n}$ be an open, bounded, $C^{r+2,1}$ subset. Let $A\in C^{r,1}\left(\overline{\Omega}; \operatorname{Hom}\left(\varLambda^{k+1}\right)\right)$, $B\in C^{r+1,1}\left(\overline{\Omega}; \operatorname{Hom}\left(\varLambda^{k}\right)\right)$ be both uniformly $\gamma$-Legendre elliptic. Let $M_0>0,$ $ 1 < p < \infty$ be such that $w\in A_p$ and $[w]_{A_p}\le M_0.$ Let $f\in W^{r,p}_w(\Omega; \varLambda^{k}),$ $g\in W^{r+1,p}_w(\Omega; \varLambda^{k-1})$ and $\lambda \ge 0$ be such that 
		\begin{align*}
			d^{\ast}f+\lambda  g=0\quad\text{ and } d^{\ast}g=0 \text{ in }\Omega.
		\end{align*}
		\begin{enumerate}[(a)]
			\item Suppose that $g\in \mathcal{H}^\perp_T(\Omega,\varLambda^{k-1})$. Then the following BVP 
			\begin{align} \label{problemMaxwellgeneraltangential} 
				\begin{cases}
					\begin{aligned}
						d^{\ast}(Ad\omega)&= \lambda B\omega+f&&\qquad\text{ in }\Omega,\\
						d^{\ast}(B\omega)&=g&&\qquad\text{ in }\Omega,\\
						\nu\wedge\omega&=\nu\wedge \omega_{0}&&\qquad\text{ on }\partial\Omega,
					\end{aligned} 
				\end{cases}
				\tag{$PM_T$}
			\end{align}
			admits a solution $\omega\in W^{r+2,p}_w(\Omega)$ satisfying the estimates
			\begin{align*}
				\|\omega\|_{W^{r+2,p}_w(\Omega)}
				&\le C\left[\|\omega\|_{L^p_w(\Omega)}+\|f\|_{W^{r, p}_w(\Omega)}+\|g\|_{W^{r+1,p}_w(\Omega)}\right], 
			\end{align*}
			where  $C = C\left(r, p, n, k, N, M_{0}, \gamma_{A}, \gamma_{B},  \lambda, \Omega, \left\lVert A \right\rVert_{C^{r,1}}, \left\lVert B \right\rVert_{C^{r+1,1}}\right) >0$.

			\item[(ii)]  Suppose \begin{align*}
				\nu\lrcorner g = \nu \lrcorner d^{\ast} \left(  B(x)\omega_{0} \right) \text{ and } \nu\lrcorner f 
				= \nu \lrcorner \left[ d^{\ast} \left( A(x)d\omega_{0} \right) - \lambda 
				B(x)\omega_{0} \right] \quad\text{ on } \partial\Omega
			\end{align*}
			and 
			$$ \int_{\Omega} \left\langle g ; \psi \right\rangle  - \int_{\partial\Omega} \left\langle \nu\lrcorner \left( B(x)\omega_{0}\right); \psi \right\rangle = 0 \qquad \text{ for all } \psi 
			\in \mathcal{H}_N(\Omega;\varLambda^{k-1}).$$
			If $\lambda =0,$ assume in addition that  
			$$\int_{\Omega} \left\langle f ; \phi \right\rangle  - \int_{\partial\Omega} \left\langle \nu\lrcorner \left( A(x)d\omega_{0} \right); \phi \right\rangle = 0 \qquad \text{ for all } \phi 
			\in \mathcal{H}_N(\Omega;\varLambda^{k}). $$ Then the following boundary value problem, 
			\begin{equation} \label{problemMaxwellgeneralnormal}
				\begin{cases}
					\begin{aligned}
						d^{\ast} ( Ad\omega )   &= \lambda B\omega + f  \text{ in } \Omega, \\
						d^{\ast} \left( B\omega \right) &= g \text{ in } \Omega, \\
						\nu\lrcorner \left( B \omega \right) &= \nu\lrcorner \left( B \omega_{0} \right)  \text{  on } \partial\Omega, \\
						\nu\lrcorner \left( A d\omega \right) &= \nu\lrcorner \left( A d\omega_{0} \right)  \text{  on } \partial\Omega. 
					\end{aligned} 
				\end{cases} 
				\tag{$PM_{N}$}
			\end{equation}
			admits a solution $\omega\in W^{r+2,p}_w(\Omega)$ satisfying the estimates
			\begin{align*}
				\|\omega\|_{W^{r+2,p}_w(\Omega)}
				&\le C\left[\|\omega\|_{L^p_w(\Omega)}+\|f\|_{W^{r, p}_w(\Omega)}+\|g\|_{W^{r+1,p}_w(\Omega)}\right], 
			\end{align*}
			where  $C = C\left(r, p, n, k, N, M_{0}, \gamma_{A}, \gamma_{B},  \lambda, \Omega, \left\lVert A \right\rVert_{C^{r,1}}, \left\lVert B \right\rVert_{C^{r+1,1}}\right) >0$. 
		\end{enumerate}
		
	\end{theorem}
	\begin{proof}
		We only show (i). Since $g\in \mathcal{H}^\perp_T(\Omega,\varLambda^{k-1})$, using Theorem \eqref{Hodge system},(a), we get the existence of $\alpha\in W^{3,q}_w(\Omega)\cap \mathcal{H}^\perp_T(\Omega,\varLambda^{k-1})$ solving the equation
		\begin{align}\label{alphag1}
			\left\lbrace
			\begin{aligned}
				d^{\ast}(Bd\alpha)+dd^{\ast}\alpha&=g - d^{\ast}\left(B\omega_{0}\right)&&\qquad\text{ in }\Omega,\\
				\nu\wedge\alpha&=0&&\qquad\text{ on }\partial\Omega,\\
				\nu\wedge d^{\ast}\alpha&=0&&\qquad\text{ on }\partial\Omega.
			\end{aligned}\right.
		\end{align}
		\noindent   Taking $d^{\ast}$ of the equation, we get that 
		\begin{equation*}
			\left\lbrace
			\begin{aligned}
				( d^{\ast}d+dd^{\ast})(d^{\ast}\alpha)&=0&&\qquad\text{ in }\Omega,\\
				\nu\wedge d^{\ast}\alpha&=0&&\qquad\text{ on }\partial\Omega,\\
				\nu\wedge d^{\ast}(d^{\ast}\alpha)&=0&&\qquad\text{ on }\partial\Omega.
			\end{aligned}
			\right.
		\end{equation*}
		\noindent Thus by uniqueness of solution, we have $d^{\ast}\alpha=0$ in $\Omega.$ Consequently the system \eqref{alphag1} reduced to 
		\begin{equation*}
			\left\lbrace
			\begin{aligned}
				d^{\ast}(Bd\alpha)&=g- d^{\ast}\left(B\omega_{0}\right)&&\qquad\text{ in }\Omega,\\
				d^\ast\alpha&=0&&\qquad\text{ in }\Omega,\\
				\nu\wedge\alpha&=0&&\qquad\text{ on }\partial\Omega.
			\end{aligned}\right.
		\end{equation*}
		\noindent Set 
		\begin{align*}
			\tilde{f}&=f+\lambda Bd\alpha +\lambda B\omega_{0} - d^{\ast}\left(Ad\omega_{0}\right).
		\end{align*}
		We see that if $\lambda=0,$ then $\tilde{f}=f - d^{\ast}\left(Ad\omega_{0}\right)\in\mathcal{H}^\perp_T(\omega,\varLambda^k)$ along with $d^*\tilde{f}=0.$ Thus for any $\lambda\ge 0,$ using Theorem \eqref{Hodge system},(i), we can find a solution $\beta\in W^{2,q}_w(\Omega)\cap \mathcal{H}^\perp_T(\Omega,\varLambda^{k})$ solving the equation
		\begin{equation*}
			\left\lbrace
			\begin{aligned}
				d^{\ast}(Ad\beta)+B^{\intercal}dd^{\ast}(B\beta)&=\lambda B\beta+\tilde{f}&&\qquad\text{ in }\Omega,\\
				\nu\wedge\beta&=0&&\qquad\text{ on }\partial\Omega,\\
				\nu\wedge d^*(B\beta)&=0&&\qquad\text{  on }\partial\Omega.
			\end{aligned}
			\right.
		\end{equation*}
		\noindent Taking $d^*$ to the above equation and by setting $\theta=d^*(B\beta),$ we obtain
		\begin{equation*}
			\left\lbrace
			\begin{aligned}
				d^*(B^Td\theta)&=\lambda\theta&&\qquad\text{ in }\Omega,\\
				\nu\wedge\theta&=0&&\qquad\text{ on }\partial\Omega,\\
				\nu\wedge d^*\theta&=0&&\qquad\text{ on }\partial\Omega.
			\end{aligned}\right.
		\end{equation*}
		\noindent Thus by considering $\theta$ as test function, we get 
		\begin{align*}
			\lambda\int_{\Omega}|\theta|^2&\le -\gamma_{B}\int_{\Omega}|d\theta|^2\le 0.
		\end{align*}
		\noindent If $\lambda =0$ then $\theta\in\mathcal{H}_T(\Omega),$ but no non-trivial harmonic field can be co-exact, so $\theta=0$ and for non-trivial $\theta$, $\lambda >0$ is impossible. Thus in either cases $\theta=0.$ Now take $\omega=\beta+d\alpha + \omega_{0}$ which solves the equation $(PM_T)$.
	\end{proof}
	
	\subsection{Hodge-Morrey decomposition}
	
	\begin{theorem}\label{Hodge decomposition}
		Let $r \geq 0$ be an integer. Let $\Omega \subset \mathbb{R}^{n}$ be an open, bounded, $C^{r+2,1}$ subset and let $A\in C^{r,1}\left(\overline{\Omega}; \operatorname{Hom}\left(\varLambda^{k+1}\right)\right)$, $B\in C^{r+1,1}\left(\overline{\Omega}; \operatorname{Hom}\left(\varLambda^{k}\right)\right)$ be both uniformly $\gamma$-Legendre elliptic.	Let $M_0>0,$ $1<p<\infty$ be real numbers such that $w\in A_{p}$ with $[w]_{A_{p}}\le M_0.$ Then given any $f \in W_{w}^{r,p}\left(\Omega; \varLambda^{k}\right),$ we have the following:
		\begin{enumerate}[(a)]
			\item  There exist $\alpha\in W^{r+1,p}_w(\Omega,\varLambda^{k-1}),$  $\beta\in W^{r+1,p}_w(\Omega,\varLambda^{k+1})$ and $h\in\mathcal{H}_T(\Omega,\varLambda^k)$ such that $f=Bd\alpha+d^{\ast}\left(A\beta\right)+h $ in $\Omega$, with 
			\begin{align*}
				\begin{aligned}
					d^{\ast}\alpha=0 \quad &\text{ and } \quad d\beta=0\qquad\text{ in }\Omega,\\
					\nu\wedge\alpha=0 \quad &\text{ and } \quad \nu\wedge \beta=0 \qquad\text{ on }\partial\Omega. 
				\end{aligned}
			\end{align*}
			\item  There exist $\alpha\in W^{r+1,p}_w(\Omega,\varLambda^{k-1}),$  $\beta\in W^{r+1,p}_w(\Omega,\varLambda^{k+1})$ and $h\in\mathcal{H}_N(\Omega,\varLambda^k)$ such that $f=Bd\alpha+d^*\beta+h$ in $\Omega$, with  
			\begin{align*}
				\begin{aligned}
					d^{\ast}\alpha=0 \quad &\text{ and } \quad d\left(A^{-1}\beta\right)=0\qquad\text{ in }\Omega,\\
					\nu\lrcorner\alpha=0 \quad &\text{ and } \quad \nu\lrcorner \beta=0 \qquad\text{ on }\partial\Omega. 
				\end{aligned}
			\end{align*}
		\end{enumerate}
		\noindent    Moreover, in either case we have the estimates
		\begin{align*}
			\|\alpha\|_{W^{r+1,p}_w(\Omega)},\|\beta\|_{W^{r+1,p}_w(\Omega)}, \|h\|_{W^{r+1,p}_w(\Omega)}
			&\le C\|f\|_{W^{r,p}_w(\Omega)}, 
		\end{align*}
		where  $C = C\left(r, p, n, k, N, M_{0}, \gamma_{A}, \gamma_{B},  \lambda, \Omega, \left\lVert A \right\rVert_{C^{r,1}}, \left\lVert B \right\rVert_{C^{r+1,1}}\right) >0$. 
	\end{theorem}
	\begin{proof}
		We only prove $(a)$ for $r=0.$ Let $s>1$ such that $L^p_w(\Omega)\subset L^s\left( \Omega \right).$ Then we decompose $f$ as 
		\begin{align*}
			f=g+h
		\end{align*}
		where $h\in\mathcal{H}_T(\Omega;\varLambda^k)$ and $g\in L^s(\Omega,\varLambda^k) \cap \mathcal{H}^\perp_T(\Omega;\varLambda^k).$ By Theorem \ref{Hodge system} we find a solution $\theta\in W^{2,p}_w(\Omega;\varLambda^k)$ to the following equation
		\begin{align*}
			\begin{cases}
				\begin{aligned}
					d^*(Ad\theta)+B^Tdd^*(B\theta)&=g&&\qquad\text{ in }\Omega,\\
					\nu\wedge\theta&=0&&\qquad\text{ on }\partial\Omega,\\
					\nu\wedge d^*(B\theta)&=0&&\qquad\text{ on }\partial\Omega.
				\end{aligned}
			\end{cases}
		\end{align*}
		\noindent Now by setting $\alpha=d^*(B\theta),$ $\beta=Ad\theta,$ we complete the proof of the theorem.
	\end{proof}

	\subsection{Div-curl system and Gaffney inequality}
	\begin{theorem}{(div-curl system)}\label{div-curl}
		Let $r \geq 0$ be an integer. Let $\Omega \subset \mathbb{R}^{n}$ be an open, bounded, $C^{r+2,1}$ subset. Let $M_0>0,$ $1<p<\infty$ be real numbers such that $w\in A_p$ with $[w]_{A_p}\le M_0.$ Let $A,B\in C^{r,1}(\overline{\Omega},\operatorname{Hom}(\varLambda^k))$ satisfy the Legendre condition. Let $f\in W^{r,p}_w(\Omega,\varLambda^{k+1}),$ $g\in W^{r,p}_w(\Omega,\varLambda^{k-1}).$ Then the following hold true:
		\begin{enumerate}[(a)]
			\item  Suppose $df=0,$ $d^*g=0$ in the sense of distribution, $\nu\wedge f=0$ on $\partial\Omega$ and $f\in \mathcal{H}^\perp_T(\Omega,\varLambda^{k+1}),$ $g\in\mathcal{H}^\perp_T(\Omega,\varLambda^{k-1}).$
			Then there exists a solution $\omega\in W^{r+1,p}_w(\Omega,\varLambda^k)$ to the following boundary value problem
			\begin{align*}
				\begin{cases}
					\begin{aligned}
						d(A\omega)&=f&&\qquad\text{ in }\Omega,\\
						d^{\ast}(B\omega)&=g&&\qquad\text{ in }\Omega,\\
						\nu\wedge A\omega&=0&&\qquad\text{ on }\partial\Omega.
					\end{aligned}
				\end{cases}
			\end{align*}

			\item Suppose $df=0,$ $d^*g=0$ in $\Omega$ in the distribution sense, $\nu\lrcorner g=0$ on $\partial\Omega$ and $f\in \mathcal{H}^\perp_N(\Omega,\varLambda^{k+1}),$ $g\in \mathcal{H}^\perp_N(\Omega,\varLambda^{k-1}).$ Then there exists a solution $\omega\in W^{r+1,p}_w(\Omega,\varLambda^k)$ to the following boundary value problem
			\begin{align*}
				\begin{cases}
					\begin{aligned}
						d(Ad\omega)&=f&&\qquad\text{ in }\Omega,\\
						d^*(B\omega)&=g&&\qquad\text{ in }\Omega,\\
						\nu\lrcorner B\omega&=0&&\qquad\text{ on }\partial\Omega.
					\end{aligned}
				\end{cases}
			\end{align*}
		\end{enumerate}
		Moreover, in either case we have the following estimate
		\begin{align*}
			\|\omega\|_{W^{r+1,p}_w(\Omega)}
			&\le C\left[\|\omega\|_{W^{r,p}_w(\Omega)}+\|f\|_{W^{r,p}_w(\Omega)}+\|g\|_{W^{r,p}_w(\Omega)}\right],
		\end{align*}
		where  $C = C\left(r, p, n, k, N, M_{0}, \gamma_{A}, \gamma_{B},  \lambda, \Omega, \left\lVert A \right\rVert_{C^{r,1}}, \left\lVert B \right\rVert_{C^{r,1}}\right) >0$. 
	\end{theorem}
	\begin{proof}
		We prove only case (a). Applying the Hodge-Morrey decomposition to $f$, we get the existence of $\alpha\in W^{r+1,p}_w(\Omega,\varLambda^{k-1}),$ $\beta\in W^{r+1,p}_w(\Omega,\varLambda^{k+1}$ and $h\in\mathcal{H}_T(\Omega,\varLambda^k)$ such that $f=d\alpha+d^{\ast}\beta+h$ in $\Omega$, where $d^{\ast}\alpha=0, d\beta=0$ in $\Omega$ and $\nu\wedge\alpha=0, \nu\wedge\beta=0$ on $\partial\Omega$. Taking $d$ to the both side of above expression, we get
		\begin{align*}
			0=df=dd\alpha+dd^*\beta+dh=dd^*\beta = (dd^* + d^{\ast}d)\beta\qquad\text{ in }\Omega.
		\end{align*}
		Since $\nu \wedge \beta=0$ and $\nu \wedge d^{\ast}\beta = \nu \wedge f =0$ on $\partial\Omega,$ we see that $d^*\beta$ must vanish. Integrating by parts, we get 
		\begin{align*}
			\int_{\Omega}|h|^2
			&=\int_{\Omega}\langle h, f - d\alpha - d^*\beta\rangle =0.
		\end{align*}
		Thus $h=0$ in $\Omega$ and hence $f=d\alpha.$ Now we set $C=BA^{-1},$ then $C$ is $C^{r,1}$ and satisfy the Legendre condition. Now we  find $\xi\in W^{r+2,p}_w(\Omega,\varLambda^k)$ which solves the following Maxwell type system
		\begin{align*}
			\begin{cases}
				\begin{aligned}
					d^{\ast}(Cd\xi)&=g-d^{\ast}(C\alpha)&&\qquad\text{ in }\Omega,\\
					d^{\ast}\xi&=0&&\qquad\text{ in }\Omega,\\
					\nu\wedge\xi&=0&&\qquad\text{ on }\partial\Omega.
				\end{aligned}
			\end{cases}
		\end{align*}
		\noindent Now we set $\omega=A^{-1}(\alpha+d\xi)$ in $\Omega.$ Then $\omega$ solves our div-curl system, since
		\begin{align*}
			& d^{\ast}(B\omega)=d^{\ast}(C(\alpha+d\xi))=g\qquad\text{ in }\Omega,\\
			&d(A\omega)=d(\alpha+d\xi)=d\alpha=f\qquad\text{ in }\Omega,\\
			&\nu\wedge A\omega=\nu\wedge\alpha+\nu\wedge d\xi=0\qquad\text{  on }\partial\Omega.
		\end{align*}
		This completes the proof of the theorem.
	\end{proof}
	\noindent As an immediate consequence of Theorem \eqref{div-curl}, we get the following weighted Gaffney type inequality.
	\begin{theorem}{(Gaffney inequality)}\label{Gaffney ineq}
		Let $\Omega\subset\mathbb R^n$ be an open bounded $C^{r+2,1}$ domain. Let $M_0>0,$ $1<p<\infty$ be real numbers such that $w\in A_p$ with $[w]_{A_p}\le M_0.$ Let $A, B\in C^{r,1}(\overline{\Omega},\operatorname{Hom}\left(\varLambda^{k}\right))$  be uniformly Legendre elliptic. Let $\omega\in W^{r,p}_w(\Omega;\varLambda^k),$ $d\left( A\omega\right)\in W^{r,p}_w(\Omega;\varLambda^{k+1})$ and $d^*(B\omega)\in W^{r,p}_w(\Omega;\varLambda^{k-1}).$ In additionally suppose either $\nu\wedge A \omega=0$ or $\nu\lrcorner B\omega=0$ on $\partial\Omega.$ Then $\omega\in W^{r+1,p}_w(\Omega,\varLambda^k)$ and there exists a constant $$C = C\left(r, p, n, k, N, M_{0}, \gamma_{A}, \gamma_{B},  \Omega, \left\lVert A \right\rVert_{C^{r,1}}, \left\lVert B \right\rVert_{C^{r,1}}\right) >0$$ such that 
		\begin{align*}
			\|\omega\|_{W^{r+1,p}_w(\Omega)}
			&\le C\left(\|d\left( A\omega\right)\|_{W^{r,p}_w(\Omega)}+\|d^*(B\omega)\|_{W^{r,p}_w(\Omega)}+\|\omega\|_{W^{r,p}_w(\Omega)}\right).
		\end{align*}
	\end{theorem}

	\section*{Acknowledgment} S. Sil's research is supported by the ANRF-SERB MATRICS Project grant MTR/2023/000885. Both the authors warmly thank Karthik Adimurthi for helpful and stimulating discussions on weighted estimates.


\begin{thebibliography}{10}
	\raggedright
	
	\bibitem{Karthik_Phuc2}
	{\sc Adimurthi, K., Mengesha, T., and Phuc, N.~C.}
	\newblock Gradient weighted norm inequalities for linear elliptic equations
	with discontinuous coefficients.
	\newblock {\em Appl. Math. Optim. 83}, 1 (2021), 327--371.
	
	\bibitem{ADN1}
	{\sc Agmon, S., Douglis, A., and Nirenberg, L.}
	\newblock Estimates near the boundary for solutions of elliptic partial
	differential equations satisfying general boundary conditions. {I}.
	\newblock {\em Comm. Pure Appl. Math. 12\/} (1959), 623--727.
	
	\bibitem{ADN2}
	{\sc Agmon, S., Douglis, A., and Nirenberg, L.}
	\newblock Estimates near the boundary for solutions of elliptic partial
	differential equations satisfying general boundary conditions. {II}.
	\newblock {\em Comm. Pure Appl. Math. 17\/} (1964), 35--92.
	
	\bibitem{Agranovich_ADNconditions}
	{\sc Agranovich, M.~S.}
	\newblock Elliptic boundary problems.
	\newblock In {\em Partial differential equations, {IX}}, vol.~79 of {\em
		Encyclopaedia Math. Sci.} Springer, Berlin, 1997, pp.~1--144, 275--281.
	\newblock Translated from the Russian by the author.
	
	\bibitem{Balci_Sil_Surnachev_HodgeVariableexponent}
	{\sc Balci, A., Sil, S., and Surnachev, M.}
	\newblock Hodge decomposition and potentials in variable exponent lebesgue and
	sobolev spaces.
	\newblock {\em arXiv e-prints, arXiv:2504.20772\/} (2025).
	
	\bibitem{csatothesis}
	{\sc Csat{\'o}, G.}
	\newblock {Some Boundary Value Problems Involving Differential Forms, PhD
		Thesis}.
	\newblock {\em EPFL}, Thesis No. 5414 (2012).
	
	\bibitem{CsatoDacKneuss}
	{\sc Csat{\'o}, G., Dacorogna, B., and Kneuss, O.}
	\newblock {\em {The pullback equation for differential forms}}.
	\newblock {Progress in Nonlinear Differential Equations and their Applications,
		83}. Birkh{\"a}user/Springer, New York, 2012.
	
	\bibitem{DiBenedettoManfredi}
	{\sc DiBenedetto, E., and Manfredi, J.}
	\newblock On the higher integrability of the gradient of weak solutions of
	certain degenerate elliptic systems.
	\newblock {\em Amer. J. Math. 115}, 5 (1993), 1107--1134.
	
	\bibitem{FriedrichsGaffney}
	{\sc Friedrichs, K.~O.}
	\newblock {Differential forms on {R}iemannian manifolds}.
	\newblock {\em Comm. Pure Appl. Math. 8\/} (1955), 551--590.
	
	\bibitem{GaffneyHarmonicoperator}
	{\sc Gaffney, M.~P.}
	\newblock {The harmonic operator for exterior differential forms}.
	\newblock {\em Proc. Nat. Acad. Sci. U. S. A. 37\/} (1951), 48--50.
	
	\bibitem{Gaffney1954}
	{\sc Gaffney, M.~P.}
	\newblock The heat equation method of {M}ilgram and {R}osenbloom for open
	{R}iemannian manifolds.
	\newblock {\em Ann. of Math. (2) 60\/} (1954), 458--466.
	
	\bibitem{Gaffney1954a}
	{\sc Gaffney, M.~P.}
	\newblock A special {S}tokes's theorem for complete {R}iemannian manifolds.
	\newblock {\em Ann. of Math. (2) 60\/} (1954), 140--145.
	
	\bibitem{GaffneyHarmonicintegrals}
	{\sc Gaffney, M.~P.}
	\newblock {Hilbert space methods in the theory of harmonic integrals}.
	\newblock {\em Trans. Amer. Math. Soc. 78\/} (1955), 426--444.
	
	\bibitem{giaquinta-martinazzi-regularity}
	{\sc Giaquinta, M., and Martinazzi, L.}
	\newblock {\em {An introduction to the regularity theory for elliptic systems,
			harmonic maps and minimal graphs}}, second~ed., vol.~11 of {\em {Appunti.
			Scuola Normale Superiore di Pisa (Nuova Serie) [Lecture Notes. Scuola Normale
			Superiore di Pisa (New Series)]}}.
	\newblock Edizioni della Normale, Pisa, 2012.
	
	\bibitem{Giusti_DCV}
	{\sc Giusti, E.}
	\newblock {\em Direct methods in the calculus of variations}.
	\newblock World Scientific Publishing Co. Inc. River Edge, NJ, 2003.
	
	\bibitem{Grafakos_modernFourier}
	{\sc Grafakos, L.}
	\newblock {\em Modern {F}ourier analysis}, second~ed., vol.~250 of {\em
		Graduate Texts in Mathematics}.
	\newblock Springer, New York, 2009.
	
	\bibitem{Hodge1934}
	{\sc Hodge, W. V.~D.}
	\newblock A {D}irichlet {P}roblem for {H}armonic {F}unctionals, with
	{A}pplications to {A}nalytic {V}arities.
	\newblock {\em Proc. London Math. Soc. (2) 36\/} (1934), 257--303.
	
	\bibitem{Iwaniec_Lpprojection}
	{\sc Iwaniec, T.}
	\newblock Projections onto gradient fields and {$L\sp{p}$}-estimates for
	degenerated elliptic operators.
	\newblock {\em Studia Math. 75}, 3 (1983), 293--312.
	
	\bibitem{Kinnunen_Zhou_localestiNonlinear}
	{\sc Kinnunen, J., and Zhou, S.}
	\newblock A local estimate for nonlinear equations with discontinuous
	coefficients.
	\newblock {\em Comm. Partial Differential Equations 24}, 11-12 (1999),
	2043--2068.
	
	\bibitem{Kinnunen_Zhou_boundaryestiNonlinear}
	{\sc Kinnunen, J., and Zhou, S.}
	\newblock A boundary estimate for nonlinear equations with discontinuous
	coefficients.
	\newblock {\em Differential Integral Equations 14}, 4 (2001), 475--492.
	
	\bibitem{Lieberman_morrey_from_Lp}
	{\sc Lieberman, G.~M.}
	\newblock A mostly elementary proof of {M}orrey space estimates for elliptic
	and parabolic equations with {VMO} coefficients.
	\newblock {\em J. Funct. Anal. 201}, 2 (2003), 457--479.
	
	\bibitem{Lopatinskii}
	{\sc Lopatinski\u~i, Y.~B.}
	\newblock On a method of reducing boundary problems for a system of
	differential equations of elliptic type to regular integral equations.
	\newblock {\em Ukrain. Mat. \v Z. 5\/} (1953), 123--151.
	
	\bibitem{Mahato_Thesis}
	{\sc Mahato, R.}
	\newblock {P}h{D} {T}hesis, {I}ndian {I}nstitute of {S}cience.
	\newblock {\em In Preparation\/}.
	
	\bibitem{Mahato_Sil_discontinuousWeightedLp_InPrep}
	{\sc Mahato, R., and Sil, S.}
	\newblock Estimates for {H}odge-{M}axwell systems with discontinuous
	anisotropic coefficients.
	\newblock {\em In Preparation\/} (2025).
	
	\bibitem{Mahato_Sil_OrliczMusielak_InPrep}
	{\sc Mahato, R., and Sil, S.}
	\newblock On the {H}odge systems in {O}rlicz-{M}usielak spaces.
	\newblock {\em In Preparation\/} (2025).
	
	\bibitem{Mitrea2016}
	{\sc Mitrea, D., Mitrea, I., Mitrea, M., and Taylor, M.}
	\newblock {\em The {H}odge-{L}aplacian}, vol.~64 of {\em De Gruyter Studies in
		Mathematics}.
	\newblock De Gruyter, Berlin, 2016.
	\newblock Boundary value problems on Riemannian manifolds.
	
	\bibitem{Mitrea_Mitrea_Taylor_LayerPotentials_HodgeLplacian}
	{\sc Mitrea, D., Mitrea, M., and Taylor, M.}
	\newblock Layer potentials, the {H}odge {L}aplacian, and global boundary
	problems in nonsmooth {R}iemannian manifolds.
	\newblock {\em Mem. Amer. Math. Soc. 150}, 713 (2001), x+120.
	
	\bibitem{Mitrea_Gaffneyineq}
	{\sc Mitrea, M.}
	\newblock Dirichlet integrals and {G}affney-{F}riedrichs inequalities in convex
	domains.
	\newblock {\em Forum Math. 13}, 4 (2001), 531--567.
	
	\bibitem{Morrey1955}
	{\sc Morrey, Jr., C.~B., and Eells, Jr., J.}
	\newblock A variational method in the theory of harmonic integrals.
	\newblock {\em Proc. Nat. Acad. Sci. U.S.A. 41\/} (1955), 391--395.
	
	\bibitem{MorreyHarmonic2}
	{\sc Morrey, J. C.~B.}
	\newblock {A variational method in the theory of harmonic integrals. {II}}.
	\newblock {\em Amer. J. Math. 78\/} (1956), 137--170.
	
	\bibitem{Morrey1966}
	{\sc Morrey, J. C.~B.}
	\newblock {\em {Multiple integrals in the calculus of variations}}.
	\newblock {Die Grundlehren der mathematischen Wissenschaften, Band 130}.
	Springer-Verlag New York, Inc., New York, 1966.
	
	\bibitem{Muckenhoupt_Apweights}
	{\sc Muckenhoupt, B.}
	\newblock Weighted norm inequalities for the {H}ardy maximal function.
	\newblock {\em Trans. Amer. Math. Soc. 165\/} (1972), 207--226.
	
	\bibitem{Phuc_Nguyen_WeightedFeffermanStein}
	{\sc Phuc, N.~C.}
	\newblock Weighted estimates for nonhomogeneous quasilinear equations with
	discontinuous coefficients.
	\newblock {\em Ann. Sc. Norm. Super. Pisa Cl. Sci. (5) 10}, 1 (2011), 1--17.
	
	\bibitem{SchwarzHodge}
	{\sc Schwarz, G.}
	\newblock {\em {Hodge decomposition---a method for solving boundary value
			problems}}, vol.~1607 of {\em {Lecture Notes in Mathematics}}.
	\newblock Springer-Verlag, Berlin, 1995.
	
	\bibitem{SenguptaSil_MorreyLorentz_Hodge}
	{\sc Sengupta, B., and Sil, S.}
	\newblock Morrey-lorentz estimates for hodge-type systems.
	\newblock {\em Discrete Contin. Dyn. Syst. 45}, 1 (2025), 334--360.
	
	\bibitem{silthesis}
	{\sc Sil, S.}
	\newblock {Calculus of {V}ariations for {D}ifferential {F}orms, PhD Thesis}.
	\newblock {\em EPFL}, Thesis No. 7060 (2016).
	
	\bibitem{Sil_linearregularity}
	{\sc Sil, S.}
	\newblock Regularity for elliptic systems of differential forms and
	applications.
	\newblock {\em Calc. Var. Partial Differential Equations 56}, 6 (2017), 56:172.
	
	\bibitem{Stein_HarmonicAnalysis}
	{\sc Stein, E.~M.}
	\newblock {\em Harmonic analysis: real-variable methods, orthogonality, and
		oscillatory integrals}, vol.~43 of {\em Princeton Mathematical Series}.
	\newblock Princeton University Press, Princeton, NJ, 1993.
	\newblock With the assistance of Timothy S. Murphy, Monographs in Harmonic
	Analysis, III.
	
	\bibitem{Visik_LScondition}
	{\sc Vi\v{s}ik, M.~I.}
	\newblock On general boundary problems for elliptic differential equations.
	\newblock {\em Trudy Moskov. Mat. Ob\v s\v c. 1\/} (1952), 187--246.
	
	\bibitem{Shapiro}
	{\sc \v{S}apiro, Z.~Y.}
	\newblock On general boundary problems for equations of elliptic type.
	\newblock {\em Izv. Akad. Nauk SSSR Ser. Mat. 17\/} (1953), 539--562.
	
	\bibitem{Weyl1940}
	{\sc Weyl, H.}
	\newblock The method of orthogonal projection in potential theory.
	\newblock {\em Duke Math. J. 7\/} (1940), 411--444.
	
\end{thebibliography}
\end{document}